\documentclass[a4paper, 12pt]{article}

\usepackage{amsmath} 
\usepackage{theorem}
\usepackage{amssymb}
\usepackage{amsfonts}
\usepackage{latexsym}
\usepackage{amscd}
\usepackage{diagramb}
\input GrCalc4.sty
\usepackage{graphicx}
\usepackage{enumerate}

\usepackage{pxfonts}

\usepackage{stmaryrd}

\DeclareMathAlphabet{\mathpzc}{OT1}{pzc}{m}{it}

\usepackage{xspace,colortbl}
\usepackage{color}

\definecolor{verde}{rgb}{0.,0.7,0.}
\definecolor{indigo}{rgb}{.18, .34, .78}
\definecolor{indigo1}{rgb}{.18, .24, .78}
\definecolor{indigo2}{rgb}{.18, .14, .78}
\definecolor{indigo3}{rgb}{.18, 0., .78}
\definecolor{rojo}{rgb}{1,0,0}
\definecolor{negro}{rgb}{0,0,0}
\definecolor{lila}{rgb}{.46, .16, .78}
\definecolor{lila1}{rgb}{.46, .16, .86}
\definecolor{lila2}{rgb}{.56, .16, .86}
	\definecolor{lila3}{rgb}{.63, .16, .78}
\definecolor{lila4}{rgb}{.7, .16, .78}
\definecolor{lila5}{rgb}{.78, .26, .78}
\definecolor{lila6}{rgb}{.6, 0., .78}

\theoremstyle{plain}
\theoremheaderfont{\normalfont\bfseries}

\newtheorem{thm}{Theorem}[section]

\newtheorem{lma}[thm]{Lemma}
\newtheorem{cor}[thm]{Corollary}
\newtheorem{defn}[thm]{Definition}

\theorembodyfont{\rmfamily}

\newtheorem{rem}[thm]{Remark}
\newtheorem{prop}[thm]{Proposition}
\newtheorem{ex}[thm]{Example}
\newcommand{\qed}{\hfill\quad\fbox{\rule[0mm]{0,0cm}{0,0mm}}  \par\bigskip}

\newcommand{\x}{\mbox{-}}
\newcommand{\R}{{\mathcal R}}

\newcommand{\s}{\hspace{0,12cm}}
\newcommand{\Comp}{\operatorname {Comp}}
\newcommand{\Cat}{\operatorname {Cat}}
\newcommand{\bEM}{{\rm bEM}}
\newcommand{\EM}{{\rm EM}}
\newcommand{\Inc}{\operatorname {Inc}}
\newcommand{\Mnd}{{\rm Mnd}}
\newcommand{\Comnd}{{\rm Comnd}}
\newcommand{\Bimnd}{{\rm Bimnd}}

\newcommand{\comp}{\circ}

\newcommand{\ot}{\otimes}

\newcommand{\C}{{\mathcal C}}
\newcommand{\M}{{\mathcal M}}
\newcommand{\D}{{\mathcal D}}

\newcommand{\A}{{\mathcal A}}
\newcommand{\B}{{\mathcal B}}

\newcommand{\U}{{\mathcal U}}
\newcommand{\YD}{{\mathcal YD}}

\newcommand{\crta}{\overline}

\newcommand{\Id}{\operatorname {Id}}
\newcommand{\id}{\operatorname {id}}

\newcommand{\Epsilon}{\varepsilon}
\newcommand{\End}{\operatorname {End}}

\def\K{{\mathcal K}}  

\newcommand{\Mod}{\operatorname{Mod}}

\newcommand{\cref}[1]{C.~\ref{c:#1}}

\newcommand{\exlabel}[1]{\label{ex:#1}}
\newcommand{\exref}[1]{Example~\ref{ex:#1}}

\newcommand{\eqlabel}[1]{\label{eq:#1}}
\newcommand{\equref}[1]{(\ref{eq:#1})}
\newcommand{\thlabel}[1]{\label{th:#1}}
\newcommand{\thref}[1]{Theorem~\ref{th:#1}}
\newcommand{\delabel}[1]{\label{de:#1}}
\newcommand{\deref}[1]{Definition~\ref{de:#1}}
\newcommand{\prlabel}[1]{\label{pr:#1}}
\newcommand{\prref}[1]{Proposition~\ref{pr:#1}}

\newcommand{\rmlabel}[1]{\label{rm:#1}}
\newcommand{\rmref}[1]{Remark~\ref{rm:#1}}
\newcommand{\selabel}[1]{\label{se:#1}}
\newcommand{\seref}[1]{Section~\ref{se:#1}}
\newcommand{\sslabel}[1]{\label{ss:#1}}
\newcommand{\ssref}[1]{Subsection~\ref{ss:#1}}

\title{Biwreaths: a self-contained system in a 2-category that encodes different known algebraic constructions and gives rise to new ones}
\author{Bojana Femi\'c \vspace{6pt} \\
{\small Facultad de Ingenier\'ia, \vspace{-2pt}}\\
{\small  Universidad de la Rep\'ublica} \vspace{-2pt}\\
{\small  Julio Herrera y Reissig 565,} \vspace{-2pt}\\
{\small  11 300 Montevideo, Uruguay}}

\begin{document}

\date{}

\maketitle

\bigbreak
\begin{center}
{\em Dedicated to Igor and all the new born babies of Women in Mathematics}
\end{center}

\begin{abstract}
We introduce bimonads in a 2-category $\K$ and define biwreaths as bimonads in the 2-category $\bEM(\K)$ of bimonads, in the analogous fashion as Lack and Street defined wreaths. 
A biwreath is then a system containing a wreath, a cowreath and their mixed versions, but also a 2-cell $\lambda$ in $\bEM(\K)$ governing the compatibility of the monad and the comonad 
structure of the biwreath. We deduce that the monad laws encode 2-(co)cycles and the comonad laws so called 3-(co)cycles, while the 2-cell conditions of the (co)monad 
structure 2-cells of the biwreath encode (co)actions twisted by these 2- and 3-(co)cycles. The compatibilities of $\lambda$ deliver concrete expressions of the latter 
structure 2-cells. We concentrate on the examples of biwreaths in the 2-category induced by a braided monoidal category $\C$ and take for the distributive laws in a biwreath 
the braidings of the different categories of Yetter-Drinfel'd modules in $\C$. We prove that the before-mentioned properties of a biwreath specified to the latter setting 
recover on the level of $\C$ different algebraic constructions known in the category ${}_R\M$ of modules over a commutative ring $R$, such as Radford biproduct, Sweedler's crossed product 
algebra, comodule algebras over a quasi-bialgebra and the Drinfel'd twist. In this way we obtain that the known examples of (mixed) wreaths coming from ${}_R\M$ 
are not merely examples, rather they are consequences of the structure of a biwreath, and that the form of their structure morphisms originates in the laws inside of a biwreath. 
Choosing different distributive laws and different 2-cells $\lambda$ in a biwreath, leads to different and possibly new algebraic constructions.

\bigbreak

\end{abstract}

\pagebreak
\tableofcontents

\section{Introduction}

Wreaths were introduced in \cite{LS}. The beauty of this construction is reafirmed in \cite{Laiachi} and \cite{BC},
where it was shown that many more known algebra constructions in the category of modules over a commutative ring are examples of wreaths. 
In a wreath, for a monad $T$ and a 1-cell $F$ in a 2-category $\K$ there are 2-cells $\psi: TF\to FT, \nu: FF\to FT$ and $\xi: \Id\to FT$
such that certain 7 axioms hold. If one takes $\K=\hat\C$ the 2-category induced by a monoidal category $\C$ (with a single 0-cell), the monad $T$ is nothing
but an algebra $A$ in $\C$, the 1-cell $F$ is an object $X$ of $\C$ and the three 2-cells correspond to three suitable morphisms 
in $\C$. 

In the above-mentioned examples of wreaths, a part from an algebra $A$ one gives an object $X$ (equipped with additional structures which were not originally included in the data of a wreath) 
and one gives three morphisms and it turns out that precisely because of the additional structures on $X$ the chosen morphisms comply with the necessary axioms of the wreath. 
The question that we raise and which is the 
motivation for the present research, is 
how does one guess which kind of object and morphisms would work for a wreath? We wondered whether 
there is a wreath-like object so that those additional appropriate structures and/or morphisms would come out of the intrinsic properties of the object.
On the other hand, we were interested in constructing 
an object which would have a wreath product and a cowreath coproduct so that these are compatible in a sense of a bialgebra-like object. 
An example of such a construction could be the Radford biproduct \cite{Rad1}.   

Bespalov and Drabant have studied cross product bialgebras in a braided monoidal category $\C$ in \cite{BD}. If one fixes $\xi$ to be given by the tensor product of the units,
their cross product (co)algebra coincides with the notion of a (co)wreath in $\hat\C$ from \cite{LS}.
Their cross product bialgebra is a bialgebra which is a cross product algebra and a cross product coalgebra at the same time. 
Although similar to this, our construction 
has two major differences. 
Firstly, our interest was to deepen the study of wreaths in a general 2-categorical setting. Thus, following the idea of a wreath, we define a biwreath as a bimonad in the 
(Eilenberg-Moore) 2-category of bimonads, where we introduce the 
notion of a bimonad in a 2-category. So, contrarily to Bespalov and Drabant we start with a 1-cell $B$ in $\K$ which is a bimonad and then we take a ``biwreath $F$ around $B$'', and 
the newly obtained 1-cell $FB$ does not have to be a bimonad in $\K$. Secondly, specifying our construction to $\K=\hat\C$ and without twisting the order of 1-cells/objects 
(\rmref{twist sides}): while we have an object $B\ot X$ where $B$ is a bialgebra in $\C$, 
 Bespalov and Drabant consider a bialgebra on the underlying object $A\ot C$ where $A$ is an algebra and $C$ a coalgebra in $\C$. 


For the definition of a bimonad in $\K$, rather than generalizing opmonoidal monads, named by McCrudden in \cite{McC} (they were introduced and called {\em Hopf monads} in \cite{Moe} 
and in \cite{Brug-Vir} they were called {\em bimonads}), we generalize the notion of a bimonad due to Mesablishvili and Wisbauer, \cite{Wisb}.
The reason for this is that in the former case the 0-cells of $\K$ should posses a monoidal structure, which would be a too strong restriction for our purposes. 

By the definition a biwreath has distributive laws $\psi, \phi$, a structure of a wreath (with 2-cells $\mu_M, \eta_M$), a cowreath  (with 2-cells $\Delta_C, \Epsilon_C$), 
a mixed wreath  (with 2-cells $\Delta_M, \Epsilon_M$), a mixed cowreath (with 2-cells $\mu_C, \eta_C$) and additional 2-cells $\lambda_M, \lambda_C$ for the bimonad compatibilities. 
We show, among other, that: the monad law for $\mu_M$ delivers a 2-cocycle condition on a 2-cell $\sigma$ defined via $\mu_M$; the comonad law for $\Delta_M$ delivers a ``3-cocycle condition'' 
on a 2-cell $\Phi_{\lambda}$ defined via $\Delta_M$; the 2-cell condition for $\mu_M$ gives an ``action twisted by a 2-cocycle''; the 2-cell condition for $\Delta_M$ gives a ``quasi coaction 
twisted by a 3-cocycle'', and the dual versions of these statements. 
The latter structure provides a 2-categorical formulation of ``twisted coaction'' introduced in \cite{Street} for bimonoids in braided monoidal categories. 
We call this ``quasi (co)action'' alluding to the quasi-bialgebra 
setting where these coactions naturally emerged. Namely, \cite[Definition 12]{Street} formalizes comodule algebras over a quasi-bialgebra $B$ from \cite{HN} to any braided monoidal category 
$\C$, so that $B$ is a proper bialgebra in $\C$ rather than a quasi-bialgebra. This is precisely what happens in our construction but on the 2-categorical level. 
As a matter of fact, from the structure of a biwreath we naturally obtain ``alternative quasi (co)action'' and the latter one we recognize in the setting of what we call a 
{\em left-right mixed biwreath-like object}. In this way we provide an alternative definition of a twisted (co)action in the context of braided monoidal categories. When the (quasi)bialgebra 
is (co)commutative, the two definitions coincide. 

Apart from the above said, we deduce a Yetter-Drinfel'd-like condition between $B$ and $F$
and also that they are (co)module (co)algebras (in a broader sense) one over the other at appropriate sides. More precisely, they are ``measured'' in the sense of 
\cite[Definition 7.1.1]{Mont}. Besides, there are weak (co)associative (co)multiplications on $F$. 
Moreover, from the bimonad compatibility for $\lambda_M$ (resp. $\lambda_C$) we deduce the form of the 2-cells $\mu_M, \Delta_M$ (resp. $\mu_C, \Delta_C$). 
It turns out that the 2-cells $\mu_M, \Delta_M$ can not be simultaneously determined unless one of them is {\em canonical} (\deref{canonical}) or it does not 
appear in the site, but it becomes substituted by a proper 2-cell in $\K$ rather than in $\bEM(\K)$.  
This reminds us of a sort of ``uncertainty principle'', 
though this uncertainty will not be an obstacle in the examples we treat in the present paper. 
Thus we achieve our objective to determine the additional structures of $F$ and the form of 2-cells defining a (co)wreath, from the intrinsic data of a biwreath. 

In the present paper we treat only the examples where $\K=\hat\C$. In all of them it occurs that when one fixes $\psi, \phi$, various intrinsic structures in a biwreath become trivial 
leading to the notions of (mixed) biwreath-like objects. 
The surviving data recovers some known algebraic construction. Fixing different $\psi, \phi$ ``turns off'' other intrinsic structures and some other data remains ``with lights on'' 
recovering some other known structure. Thus when 
$\psi, \phi$ come from the pre-braiding of the category ${}_B ^B\YD(\C)$ of Yetter-Drinfel'd modules over $B$, studied in \cite{Besp, Femic1}, where $\C$ is a braided monoidal category, 
we recover the Radford biproduct in $\C$, \ssref{ex psi_1}. In this case we obtain that a biwreath indeed is a bimonad in $\hat\C$, \thref{Radford}. 
When $\psi, \phi$ come from the pre-braiding of $\YD(\C)_F ^F$, we recover the Sweedler's crossed (co)product (co)algebra \cite[Section 7]{Mont} 
in $\C$ (\ssref{ex psi_2}, \exref{Sw}). For $\psi, \phi$ the pre-braiding of ${}_F ^F\YD(\C)$, we obtain the above-mentioned quasi (co)actions twisted by a 3-(co)cycle and a mixed 
wreath \cite[Proposition 13]{Street}, \cite[Proposition 5.3]{BC} in $\C$ (see \ssref{Case 3}, \exref{3-1} -- \exref{3-3}), as well as their alternative versions 
(\ssref{Case 3 alt}, \exref{alternative}). In the particular case when $B=I$ we recover the 
Drinfel'd twist and its dual version from \cite{Maj5} in $\C$ (\ssref{B=I}, \exref{Drinfeld}). 
Moreover, from our data ``biwreath-like hybrid'' we recover that Sweedler's 2-cocycle twists the multiplication and that Drinfel'd twist twists the comultiplication 
of a bialgebra in $\C$. All this illustrates that the mentioned known constructions in the category ${}_R\M$ of modules a commutative ring $R$ have their origin in the structure of a biwreath 
and associated notions. Moreover, from the above said we obtain that the known examples of (mixed) wreaths coming from the category ${}_R\M$, rather than being 
merely examples, actually come out from the structure of a biwreath, and that the form of their structure morphisms is determined by the laws of a biwreath. 

We should mention that our names {\em cycle} and {\em 3-(co)cycle} are not precise: our cycle could be understood as a {\em dual cocycle} and our 3-cocycle $\Phi_{\lambda}: I\to FFB$ 
is not really a 3-cocycle, but when $B=F$ it satisfies a condition similar to that of a reassociator $\Phi: I\to FFF$ in a quasi-bialgebra $F$, which indeed is a 3-cocycle. 

In a biwreath one may choose different distributive laws $\psi, \phi$ and 2-cells $\lambda$ from the ones we deal with here and one may choose whether to consider certain structure 
2-cell canonical or not. This gives different combinations of structure 2-cells and a fortiori different wreath and cowreath (co)products. 
These different choices give rise to further different (mixed) (co)wreath (co)product structures, which possibly have never been studied yet. The 
investigation of these different new structures we leave for a future study. 

As for the organization of the paper, 
in the second section we recall some necessary known notions but also introduce some new ones, \deref{bimonad} and \deref{(co)modules}, 
and prove a preliminary result \prref{distr-actions}. In the third section we introduce the 2-category $\bEM(\K)$ of bimonads in $\K$ whose 2-cells consist of pairs made by 2-cells in 
$\EM^M(\K)$ and $\EM^C(\K)$, being the latter the free completion 2-categories under the Eilenberg-Moore objects for monads and comonads in $\K$, respectively. 
We conjecture that the 2-category $\bEM(\K)$ is the free completion of $\Bimnd(\K)$ under the Eilenberg-Moore construction for bimonads. 
\seref{biwreaths} is devoted to the definition of a biwreath 
and analysis of the structures which lie inside of a biwreath. In the last three sections we study the examples mentioned in a previous paragraph: \seref{example1} recovers 
Radford biproduct from a biwreath, \seref{bl objects} Sweedler's crossed prodcut from a biwreath-like object and \seref{mixed bl objects} (alternative) quasi (co)actions 
from a (left-right) mixed biwreath-like object.

\section{Preliminaries}

We first say some words on the setting and notation.
Throughout $\K$ will denote a 2-category. We assume that the reader is familiar with the basic notions of 2-categories, for reference we recommend \cite{Be, Bo}. The arrows of both 1- and
2-cells in $\K$ we will denote by $\to$ and we will stress which kind of arrow is meant. (Co)monads and distributive laws in $\K$ involve 1- and 2-cells and their compositions. 
As far as one works only with 1-cells acting as endomorphisms on the same 0-cell $\A$, one works in the monoidal category $\End(\A)$ of endomorphisms of $\A$ and hence one can use 
string diagrams for monoidal categories in the computations. 
Throughout we will freely use string diagrams both for expressions in a monoidal category $\C$ (the objects of $\C$ are sources and targets of the strings and the strings stand for morphisms in $\C$), 
as for 2-categories (sources and targets of strings are 1-cells which may be composed, while the strings stand for 2-cells).
Left and right actions of a monoid, left and right coactions of a comonoid, a 2-cell \textit{i.e.} a morphism in $\C$, multiplication and unit of a monoid, commultiplication and
counit of a comonoid (both in $\C$ and $\K$) we write respectively: \vspace{-0,7cm}
\begin{center}
\begin{equation} \eqlabel{strings}
\begin{tabular}{ccp{0.5cm}cccp{0.5cm}c}
{\footnotesize left action} & {\footnotesize right action}   & & \multicolumn{1}{c}{{\footnotesize left coaction}}  &  {\footnotesize right coaction} & {\footnotesize 2-cell / morphism} \\
$\gbeg{2}{2}
\got{1}{} \gnl
\glm \gnl
\gend$ & $\gbeg{2}{2}
\got{3}{} \gnl
\grm \gnl
\gend$ & &  $\gbeg{2}{2}
\got{3}{} \gnl
\glcm \gnl
\gend$ & $\quad \gbeg{2}{2}
\got{3}{} \gnl
\grcm \gnl
\gend$ &
$\gbeg{2}{2}
\gcl{1} \gnl
\gbmp{f} \gnl
\gcl{1} \gnl
\gend
\vspace{0,6cm}
$   \\
\hspace{5pt}
{\footnotesize multiplication} & {\footnotesize unit} & & {\footnotesize comultiplication} & {\footnotesize counit} & & \\
 $\gbeg{2}{2}
\got{1}{} \gnl
\gmu \gnl
\gend$ & $\gbeg{1}{2}
\got{1}{} \gnl
\gu{1} \gnl
\gend$ & & $
\gbeg{2}{2}
\got{2}{} \gnl
\gcmu \gnl
\gend$ &  $
\gbeg{1}{2}
\got{1}{} \gnl
\gcu{1} \gnl
\gend$
\end{tabular}
\end{equation}
\end{center}

In the next subsection we recollect some known concepts,
we introduce the notion of a bimonad in a 2-category and prove a preliminary result \prref{distr-actions}.

\subsection{(Co)monads, bimonads, distributive laws and (co)actions of (co)monads in 2-categories} \sslabel{prelim 1}

Monads in 2-categories were introduced by Street in \cite{Street}, we recall the definition here.

A monad on a 0-cell $\A$ in $\K$ is a 1-cell $T:\A\to\A$ together with 2-cells $\mu:TT\to T$ and $\eta: \Id_{\A}\to T$ such that
$\mu(\mu\times\Id_T)=\mu(\Id_T\times\mu)$ and $\mu(\Id_T\times\eta)=\Id_T=\mu(\eta\times\Id_T)$.
In string diagrams we write this as:
$$\scalebox{0.86}{
\gbeg{4}{4}
\got{1}{T} \got{1}{} \got{1}{T} \got{1}{T} \gnl
\gwmu{3} \gcl{1} \gnl
\gvac{1} \gwmu{3} \gnl
\gvac{2} \gob{1}{T}
\gend}=
\scalebox{0.86}{
\gbeg{4}{4}
\got{1}{T} \got{1}{T} \got{1}{} \got{1}{T} \gnl
\gcl{1} \gwmu{3} \gnl
\gwmu{3} \gnl
\gvac{1} \gob{1}{T}
\gend}
\qquad \textnormal{and} \qquad
\scalebox{0.86}{
\gbeg{2}{4}
\got{1}{} \got{1}{T} \gnl
\gu{1} \gcl{1} \gnl
\gmu \gnl
\gob{2}{T}
\gend}=
\scalebox{0.86}{
\gbeg{1}{4}
\got{1}{T} \gnl
\gcl{2} \gnl
\gob{1}{T}
\gend}=
\scalebox{0.86}{
\gbeg{2}{4}
\got{1}{T} \got{1}{} \gnl
\gcl{1} \gu{1} \gnl
\gmu \gnl
\gob{2}{T.}
\gend}
$$

Dually, a comonad on a 0-cell $\A$ in $\K$ is a 1-cell $D:\A\to\A$ together with 2-cells $\Delta:D\to DD$ and $\Epsilon: D\to\Id_{\A}$ such that
$(\Delta\times\Id_D)\Delta=(\Id_D\times\Delta)\Delta$ and $(\Id_D\times\Epsilon)\Delta=\Id_D=(\Epsilon\times\Id_D)\Delta$.
The string diagrams for the comonad laws are vertically symmetric to the ones above.

\medskip

Distributive laws in 2-categories were defined in \cite{Beck}. Let us consider the following two cases.

\begin{defn} \delabel{distr}
Let $(\A, T, \mu, \eta)$ be a monad, $(\A, D, \Delta, \Epsilon)$ a comonad and $F:\A\to\A$ a 1-cell in $\K$.
\begin{enumerate}[(a)]
\item A 2-cell $\psi:TF\to FT$ in $\K$ is called a distributive law from the monad $T$ to $F$ if identities \equref{psi laws} hold.
\item A 2-cell $\phi:FD\to DF$ in $\K$ is called a distributive law from $F$ to the comonad $D$ if identities \equref{phi laws} hold.
\end{enumerate}
\vspace{-1,4cm}
\begin{center} \hspace{-0,2cm}
\begin{tabular}{p{7.2cm}p{0cm}p{8cm}}
\begin{equation}\eqlabel{psi laws}
\gbeg{3}{5}
\got{1}{T}\got{1}{T}\got{1}{F}\gnl
\gcl{1} \glmpt \gnot{\hspace{-0,34cm}\psi} \grmptb \gnl
\glmptb \gnot{\hspace{-0,34cm}\psi} \grmptb \gcl{1} \gnl
\gcl{1} \gmu \gnl
\gob{1}{F} \gob{2}{T}
\gend=
\gbeg{3}{5}
\got{1}{T}\got{1}{T}\got{1}{F}\gnl
\gmu \gcn{1}{1}{1}{0} \gnl
\gvac{1} \hspace{-0,34cm} \glmptb \gnot{\hspace{-0,34cm}\psi} \grmptb  \gnl
\gvac{1} \gcl{1} \gcl{1} \gnl
\gvac{1} \gob{1}{F} \gob{1}{T}
\gend;
\quad
\gbeg{2}{5}
\got{3}{F} \gnl
\gu{1} \gcl{1} \gnl
\glmptb \gnot{\hspace{-0,34cm}\psi} \grmptb \gnl
\gcl{1} \gcl{1} \gnl
\gob{1}{F} \gob{1}{T}
\gend=
\gbeg{3}{5}
\got{1}{F} \gnl
\gcl{1} \gu{1} \gnl
\gcl{2} \gcl{2} \gnl
\gob{1}{F} \gob{1}{T}
\gend
\end{equation} & &
\begin{equation}\eqlabel{phi laws}
\gbeg{3}{5}
\got{1}{F} \got{2}{D}\gnl
\gcl{1} \gcmu \gnl
\glmptb \gnot{\hspace{-0,34cm}\phi} \grmptb \gcl{1} \gnl
\gcl{1} \glmptb \gnot{\hspace{-0,34cm}\phi} \grmptb \gnl
\gob{1}{D} \gob{1}{D} \gob{1}{F}
\gend=
\gbeg{3}{5}
\got{2}{F} \got{1}{\hspace{-0,2cm}D}\gnl
\gcn{1}{1}{2}{2} \gcn{1}{1}{2}{2} \gnl
\gvac{1} \hspace{-0,34cm} \glmpt \gnot{\hspace{-0,34cm}\phi} \grmptb \gnl
\gvac{1} \hspace{-0,2cm} \gcmu \gcn{1}{1}{0}{1} \gnl
\gvac{1} \gob{1}{D} \gob{1}{D} \gob{1}{F}
\gend;
\quad
\gbeg{3}{5}
\got{1}{F} \got{1}{D} \gnl
\gcl{1} \gcl{1} \gnl
\glmptb \gnot{\hspace{-0,34cm}\phi} \grmptb \gnl
\gcu{1} \gcl{1} \gnl
\gob{3}{F}
\gend=
\gbeg{3}{5}
\got{1}{F} \got{1}{D} \gnl
\gcl{1} \gcl{1} \gnl
\gcl{2}  \gcu{1} \gnl
\gob{1}{F}
\gend
\end{equation}
\end{tabular}
\end{center} 
\end{defn}

We define bimonads in $\K$ generalizing the definition of a bimonad over an ordinary category from \cite{Wisb}.

\begin{defn} \delabel{bimonad}
A 1-cell $B:\A\to\A$ in $\K$ is called a {\em bimonad} if $(\A, B, \mu, \eta)$ is a monad, $(\A, B, \Delta, \Epsilon)$ a comonad and there is a 2-cell $\lambda: BB\to BB$ in $\K$
such that the following conditions are fulfilled:
\begin{enumerate}[(a)]
\item

\begin{center}
\begin{tabular}{p{4.2cm}p{0cm}p{4.2cm}p{0cm}p{3.6cm}}
\begin{equation} \eqlabel{eta-Delta B}
\gbeg{2}{3}
\got{1}{B} \got{1}{B} \gnl
\gcl{1} \gcl{1} \gnl
\gcu{1}  \gcu{1} \gnl
\gend=
\gbeg{2}{3}
\got{1}{B} \got{1}{B} \gnl
\gmu \gnl
\gvac{1} \hspace{-0,2cm} \gcu{1} \gnl
\gob{1}{}
\gend
\end{equation} & \qquad &  \vspace{-0,4cm} 
\begin{equation} \eqlabel{epsilon-mu B}
\gbeg{2}{3}
\gu{1}  \gu{1} \gnl
\gcl{1} \gcl{1} \gnl
\gob{1}{B} \gob{1}{B}
\gend=
\gbeg{2}{3}
\gu{1} \gnl
\hspace{-0,34cm} \gcmu \gnl
\gob{1}{B} \gob{1}{B}
\gend
\end{equation} & &  \vspace{-0,6cm}
\begin{equation} \eqlabel{epsilon-eta B}
\gbeg{1}{2}
\gu{1} \gnl
\gcu{1} \gnl
\gob{1}{}
\gend=
\Id_{id_{\A}}.
\end{equation}
\end{tabular}
\end{center}

\item $\lambda$ obeys the conditions \equref{psi laws} and \equref{phi laws} from \deref{distr}, that is:
\begin{center} \hspace{-1,2cm}
\begin{tabular}{p{7.2cm}p{-1cm}p{8cm}}
\gbeg{3}{5}
\got{1}{B}\got{1}{B}\got{1}{B}\gnl
\gcl{1} \glmptb \gnot{\hspace{-0,34cm}\lambda} \grmptb \gnl
\glmptb \gnot{\hspace{-0,34cm}\lambda} \grmptb \gcl{1} \gnl
\gcl{1} \gmu \gnl
\gob{1}{B} \gob{2}{B}
\gend=
\gbeg{3}{5}
\got{1}{B}\got{1}{B}\got{1}{B}\gnl
\gmu \gcn{1}{1}{1}{0} \gnl
\gvac{1} \hspace{-0,34cm} \glmptb \gnot{\hspace{-0,34cm}\lambda} \grmptb  \gnl
\gvac{1} \gcl{1} \gcl{1} \gnl
\gvac{1} \gob{1}{B} \gob{1}{B}
\gend;
\quad
\gbeg{2}{5}
\got{3}{B} \gnl
\gu{1} \gcl{1} \gnl
\glmptb \gnot{\hspace{-0,34cm}\lambda} \grmptb \gnl
\gcl{1} \gcl{1} \gnl
\gob{1}{B} \gob{1}{B}
\gend=
\gbeg{2}{5}
\got{1}{B} \gnl
\gcl{1} \gu{1} \gnl
\gcl{2} \gcl{2} \gnl
\gob{1}{B} \gob{1}{B}
\gend;
& &
\gbeg{3}{5}
\got{1}{B} \got{2}{B}\gnl
\gcl{1} \gcmu \gnl
\glmptb \gnot{\hspace{-0,34cm}\lambda} \grmptb \gcl{1} \gnl
\gcl{1} \glmptb \gnot{\hspace{-0,34cm}\lambda} \grmptb \gnl
\gob{1}{B} \gob{1}{B} \gob{1}{B}
\gend=
\gbeg{3}{5}
\got{2}{B} \got{1}{\hspace{-0,4cm}B} \gnl
\gcn{1}{1}{2}{2} \gcn{1}{1}{2}{2} \gnl
\gvac{1} \hspace{-0,34cm} \glmptb \gnot{\hspace{-0,34cm}\lambda} \grmptb \gnl
\gvac{1} \hspace{-0,2cm} \gcmu \gcn{1}{1}{0}{1} \gnl
\gvac{1} \gob{1}{B} \gob{1}{B} \gob{1}{B}
\gend;
\quad
\gbeg{3}{5}
\got{1}{B} \got{1}{B} \gnl
\gcl{1} \gcl{1} \gnl
\glmptb \gnot{\hspace{-0,34cm}\lambda} \grmptb \gnl
\gcu{1} \gcl{1} \gnl
\gob{3}{B}
\gend=
\gbeg{3}{5}
\got{1}{B} \got{1}{B} \gnl
\gcl{1} \gcl{1} \gnl
\gcl{2}  \gcu{1} \gnl
\gob{1}{B}
\gend
\end{tabular}
\end{center}
\item
\begin{equation} \eqlabel{bialg-lambda}
\scalebox{0.86}{
\gbeg{3}{5}
\got{1}{B} \got{3}{B} \gnl
\gwmu{3} \gnl
\gvac{1} \gcl{1} \gnl
\gwcm{3} \gnl
\gob{1}{B}\gvac{1}\gob{1}{B}
\gend}=
\scalebox{0.86}{
\gbeg{4}{5}
\got{1}{B} \got{3}{B} \gnl
\gcl{1} \gwcm{3} \gnl
\glmptb \gnot{\hspace{-0,34cm}\lambda} \grmptb \gvac{1} \gcl{1} \gnl
\gcl{1} \gwmu{3} \gnl
\gob{1}{B} \gvac{1} \gob{2}{\hspace{-0,34cm}B.}
\gend}
\end{equation}
\end{enumerate}
\end{defn}

Analogously to (co)actions of (co)algebras in monoidal categories, we define (co)actions of (co)monads. Then we will show that distributibe laws involving a (co)monad
under certain conditions give rise to (co)actions of these (co)monads.

\begin{defn} \delabel{(co)modules}
Let $(\A, T, \mu, \eta)$ be a monad and $(\A, D, \Delta, \Epsilon)$ a comonad in $\K$.
\begin{enumerate}[(a)]
\item A 1-cell $F:\B\to\A$ in $\K$ is called a {\em left $T$-module} if there is a 2-cell $\nu: TF\to F$ such that $\nu(\mu\times\Id_F)=\nu(\Id_T\times\nu)$
and $\nu(\eta\times\Id_F)=\Id_F$ holds.
\item A 1-cell $F:\A\to\B$ in $\K$ is called a {\em right $T$-module} if there is a 2-cell $\nu: FT\to F$ such that $\nu(\Id_F\times\mu)=\nu(\nu\times\Id_T)$
and $\nu(\Id_F\times\eta)=\Id_F$ holds.
\item A 1-cell $F:\B\to\A$ in $\K$ is called a {\em left $D$-comodule} if there is a 2-cell $\lambda: F\to DF$ such that $(\Delta\times\Id_F)\lambda=(\Id_D\times\lambda)\lambda$
and $(\Epsilon\times\Id_F)\lambda=\Id_F$ holds.
\item A 1-cell $F:\A\to\B$ in $\K$ is called a {\em right $D$-comodule} if there is a 2-cell $\rho: F\to FD$ such that $(\Id_F\times\Delta)\rho=(\rho\times\Id_D)\rho$
and $(\Id_F\times\Epsilon)\rho=\Id_F$ holds.
\end{enumerate}
\end{defn}

When we have a left $T$-module $F$ and we use string diagrams, instead of writing the 2-cell
$\gbeg{2}{5}
\got{1}{T}\got{1}{F}\gnl
\gcl{1} \gcl{1} \gnl
\glmpt \gnot{\hspace{-0,34cm}\nu} \grmptb \gnl
\gvac{1} \gcl{1} \gnl
\gob{3}{F}
\gend
$ 
we will usually write just
$\gbeg{2}{3}
\got{1}{T} \got{1}{F} \gnl
\glm \gnl
\gob{3}{F}
\gend$.
Thus the left module laws in string diagrams will read:
$$\scalebox{0.86}{
\gbeg{4}{5}
\got{1}{T}\got{3}{T}\got{0}{\hspace{-0,34cm}F} \gnl
\gwmu{3} \gcl{1} \gnl
\gvac{1} \gcn{2}{1}{1}{3} \gcl{1} \gnl
\gvac{2} \glm \gnl
\gvac{3} \gob{1}{F}
\gend} =
\scalebox{0.86}{
\gbeg{3}{5}
\got{1}{T}\got{1}{T}\got{1}{F} \gnl
\gcl{1} \glm \gnl
\gcn{2}{1}{1}{3} \gcl{1} \gnl
\gvac{1} \glm \gvac{1} \gnl
\gvac{2} \gob{1}{F}
\gend}
\qquad \textnormal{and} \qquad
\scalebox{0.86}{
\gbeg{2}{4}
\got{1}{}\got{1}{F} \gnl
\gu{1} \gcl{1} \gnl
\glm \gnl
\gvac{1}\gob{1}{F}
\gend}=
\scalebox{0.86}{
\gbeg{1}{4}
\got{1}{F} \gnl
\gcl{2} \gnl
\gob{1}{F}
\gend}
$$
Similarly we will proceed in the rest of the three cases in the above definition, using the corresponding symbol from \equref{strings}.
Now we find:

\begin{prop} \prlabel{distr-actions}
Let $F:\A\to\A$ be a 1-cell in $\K$.
\begin{enumerate}[(a)]
\item
Given a monad  $(\A, B, \mu, \eta)$ with a distributive law $\psi:BF\to FB$ from the monad $B$ to $F$ and a 2-cell
$\Epsilon=
\gbeg{2}{1}
\got{1}{B} \gnl
\gcu{1} \gnl
\gend$ such that
$\gbeg{2}{3}
\got{1}{B} \got{1}{B} \gnl
\gcu{1}  \gcu{1} \gnl
\gob{1}{}
\gend=
\gbeg{2}{3}
\got{1}{B} \got{1}{B} \gnl
\gmu \gnl
\gvac{1} \hspace{-0,2cm} \gcu{1} \gnl
\gob{1}{}
\gend$ and
$\gbeg{1}{2}
\gu{1} \gnl
\gcu{1} \gnl
\gob{1}{}
\gend=
\Id_{id_{\A}}$
hold, then the 2-cell:
\begin{equation} \eqlabel{F left B-mod}
\gbeg{2}{3}
\got{1}{B} \got{1}{F} \gnl
\glm \gnl
\gob{3}{F}
\gend=
\gbeg{2}{5}
\got{1}{B} \got{1}{F} \gnl
\gcl{1} \gcl{1} \gnl
\glmptb \gnot{\hspace{-0,34cm}\psi} \grmptb \gnl
\gcl{1} \gcu{1} \gnl
\gob{1}{F}
\gend
\end{equation}
makes $F$ a left $B$-module.
\item
Given a comonad $(\A, B, \Delta, \Epsilon)$ with a distributive law $\phi:FB\to BF$ from $F$ to the comonad $B$ and a 2-cell
$\eta=
\gbeg{2}{2}
\gu{1} \gnl
\gob{1}{B}
\gend$ such that
$
\gbeg{2}{2}
\gu{1} \gu{1} \gnl
\gob{1}{B} \gob{1}{B}
\gend=
\gbeg{2}{4}
\got{1}{}  \gnl
\gu{1} \gnl
\hspace{-0,34cm} \gcmu \gnl
\gob{1}{B} \gob{1}{B}
\gend$
and
$\gbeg{1}{2}
\gu{1} \gnl
\gcu{1} \gnl
\gob{1}{}
\gend=
\Id_{id_{\A}}$ hold, then the 2-cell:
\begin{equation} \eqlabel{F left B-comod}
\gbeg{2}{3}
\got{3}{F} \gnl
\glcm \gnl
\gob{1}{B} \gob{1}{F}
\gend=
\gbeg{2}{5}
\got{1}{F} \gnl
\gcl{1} \gu{1} \gnl
\glmptb \gnot{\hspace{-0,34cm}\phi} \grmptb \gnl
\gcl{1} \gcl{1} \gnl
\gob{1}{B} \gob{1}{F}
\gend
\end{equation}
makes $F$ a left $B$-comodule.
\item In particular, given a bimonad $(\A, B, \mu, \eta, \Delta, \Epsilon, \lambda)$ and distributive laws $\psi:BF\to FB$ and $\phi:FB\to BF$ as in (a) and (b), the 2-cells
\equref{F left B-mod} and \equref{F left B-comod} make $F$ a left $B$-module and a left $B$-comodule.
\end{enumerate}
\end{prop}

\begin{proof}
For (1) apply $\Epsilon$ to the first law in \equref{psi laws}, and $\eta$ to the second one. The proof of (2) is analogous, (3) is obvious.
\qed\end{proof}

\subsection{The 2-categories $\EM^M(\K)$ and $\EM^C(\K)$, wreaths and cowreaths}  \sslabel{(co)wreaths}


The 2-category $\Mnd(\K)$ of monads in $\K$ was introduced by Street in \cite{Street}. It need not posses Eilenberg?Moore objects. The free completion $\EM(\K)$ of $\Mnd(\K)$
under the Eilenberg-?oore construction was described by Street and Lack in \cite{LS}. It is a 2-category which has the same 0- and 1-cells as $\Mnd(\K)$,
but the 2-cell slightly differ. A wreath is then defined, also in \cite{LS}, as a monad in $\EM(\K)$, that is, a 0-cell of the 2-category $\EM(\EM(\K))$.
We are not going to write out here how the 2-category $\EM(\K)$ looks like, nor what a wreath is, as both will appear as integral part of our 2-category
$\bEM(\K)$ and a biwreath, respectively. 


In the context of our work the 2-category $\EM(\K)$ from \cite{LS} we will denote by $\EM^M(\K)$ emphasizing that one deals with the Eilenberg-?oore construction for monads.
We will also use the 2-category $\EM^C(\K)$ which is defined in analogous way considering comonads in $\K$. Accordingly, we will have cowreaths as comonads in the 2-category $\EM^C(\K)$,
{\em i.e.} 0-cells of $\EM^C(\EM^C(\K))$. 

Let us comment the duality between the 2-categories $\EM^M(\K)$ and $\EM^C(\K)$.
Recall that for any 2-category $\K$ there exist three dual 2-categories: (1) $\K^{op}$ which differs from $\K$ only in that the 1-cells appear in reversed order; $\K_{op}$ where only the 2-cells
are reversed with respect to those in $\K$; and $\K^{op}_{op}$ where both 1- and 2-cells are reversed.

It is a known fact that a monad in $\K_{op}$ is a comonad in $\K$: a 1-cell $T:\A\to\A$
with 2-cells $\mu:TT\to T$ and $\eta: I\to T$ with the three monad conditions in $\K_{op}$ translate into a 1-cell $T:\A\to\A$ and 2-cells $\Delta: T\to TT$ and $\Epsilon: T\to I$
satisfying dual conditions in $\K$. As observed in \cite{PW} before Definition 5.2, in a similar fashion one may define morphisms between comonads in $\K$
({\em i.e.} 1-cells in the candidate 2-category $\Comnd(\K)$ of comonads in $\K$) as 1-cells in $\Mnd(\K_{op})$, but in order to have a forgetful 2-functor $\Comnd(\K)\to\K$,
one needs to reverse the 2-cells in $\Mnd(\K_{op})$ in order to get 2-cells in $\Comnd(\K)$. Thus, one defines $\Comnd(\K)=(\Mnd(\K_{op}))_{op}$.

A similar thing happens in $\EM^C(\K)$: by construction its 0- and 1-cells coincide with those in $\EM(K_{op})$, and a 2-cell $\rho: (F,\psi)\rightarrow (G,\phi)$ in $\EM^C(\K)$ is given
by a suitable 2-cell $\hat\rho: FT\rightarrow G$ in $\K$, while a 2-cell $\tau: (F,\psi)\rightarrow (G,\phi)$ in $\EM(K_{op})$ is given by
a 2-cell $\hat\tau: GT\rightarrow F$ in $\K$. Also the 2-category $\EM^C(\K)$ will be contained in our 2-category $\bEM(\K)$, where its construction will be evident.




\section{The bimonad case: the 2-category $\bEM(\K)$} \selabel{bEM cat}

Following the idea of the construction of the 2-categories $\EM^M(\K)$ and $\EM^C(\K)$ we define here the 2-category $\bEM(\K)$. Its 0-cells are bimonads
$(\A, B, \mu, \eta, \Delta, \Epsilon, \lambda)$
in $\K$, defined in \deref{bimonad},
1-cells are triples $(F,\psi,\phi)$ where $(F,\psi)$ is a 1-cell in $\EM^M(\K)$ and $(F,\phi)$ is a 1-cell in $\EM^C(\K)$ with a certain compatibility condition between $\psi, \phi$ and $\lambda$,
and 2-cells are pairs $(\rho_M, \rho_C)$ where $\rho_M$ is a 2-cell in $\EM^M(\K)$ and $\rho_C$ is a 2-cell in $\EM^C(\K)$. We list the complete description of $\bEM(\K)$ in the sequel.

\medskip

\underline{0-cells:} are bimonads $(\A, B, \mu, \eta, \Delta, \Epsilon, \lambda)$ in $\K$ defined in \deref{bimonad}.

\underline{1-cells:} are triples $(F,\psi,\phi): (\A, B, \mu, \eta, \Delta, \Epsilon, \lambda)\to(\A', B', \mu', \eta', \Delta', \Epsilon', \lambda')$ where $F:\A\to\A'$ is a 1-cell in $\K$,
$\psi: B'F\to FB$ and $\phi: FB\to B'F$ are 2-cells in $\K$ so that $(F, \psi): (\A, B, \mu, \eta)\to(\A', B', \mu', \eta')$ is a 1-cell in $\EM^M(\K)$ and
$(F, \phi): (\A, B, \Delta, \Epsilon)\to(\A', B', \Delta', \Epsilon')$ is a 1-cell in $\EM^C(\K)$, that is, the following identities hold:
\vspace{-1,4cm}
\begin{center} \hspace{-0,6cm}
\begin{tabular}{p{7.4cm}p{0cm}p{8cm}}
\begin{equation}\eqlabel{psi laws for bimonads}
\gbeg{3}{5}
\got{1}{B'}\got{1}{B'}\got{1}{F}\gnl
\gcl{1} \glmpt \gnot{\hspace{-0,34cm}\psi} \grmptb \gnl
\glmptb \gnot{\hspace{-0,34cm}\psi} \grmptb \gcl{1} \gnl
\gcl{1} \gmu \gnl
\gob{1}{F} \gob{2}{B}
\gend=
\gbeg{3}{5}
\got{1}{B'}\got{1}{B'}\got{1}{B'}\gnl
\gmu \gcn{1}{1}{1}{0} \gnl
\gvac{1} \hspace{-0,34cm} \glmptb \gnot{\hspace{-0,34cm}\psi} \grmptb  \gnl
\gvac{1} \gcl{1} \gcl{1} \gnl
\gvac{1} \gob{1}{F} \gob{1}{B}
\gend;
\quad
\gbeg{2}{5}
\got{3}{F} \gnl
\gu{1} \gcl{1} \gnl
\glmptb \gnot{\hspace{-0,34cm}\psi} \grmptb \gnl
\gcl{1} \gcl{1} \gnl
\gob{1}{F} \gob{1}{B}
\gend=
\gbeg{3}{5}
\got{1}{F} \gnl
\gcl{1} \gu{1} \gnl
\gcl{2} \gcl{2} \gnl
\gob{1}{F} \gob{1}{B}
\gend
\end{equation} & &
\begin{equation}\eqlabel{phi laws for bimonads}
\gbeg{3}{5}
\got{1}{F} \got{2}{B}\gnl
\gcl{1} \gcmu \gnl
\glmptb \gnot{\hspace{-0,34cm}\phi} \grmptb \gcl{1} \gnl
\gcl{1} \glmptb \gnot{\hspace{-0,34cm}\phi} \grmptb \gnl
\gob{1}{B'} \gob{1}{B'} \gob{1}{F}
\gend=
\gbeg{3}{5}
\got{2}{F} \got{1}{\hspace{-0,2cm}D}\gnl
\gcn{1}{1}{2}{2} \gcn{1}{1}{2}{2} \gnl
\gvac{1} \hspace{-0,34cm} \glmpt \gnot{\hspace{-0,34cm}\phi} \grmptb \gnl
\gvac{1} \hspace{-0,2cm} \gcmu \gcn{1}{1}{0}{1} \gnl
\gvac{1} \gob{1}{B'} \gob{1}{B'} \gob{1}{F}
\gend;
\quad
\gbeg{3}{5}
\got{1}{F} \got{1}{B} \gnl
\gcl{1} \gcl{1} \gnl
\glmptb \gnot{\hspace{-0,34cm}\phi} \grmptb \gnl
\gcu{1} \gcl{1} \gnl
\gob{3}{F}
\gend=
\gbeg{3}{5}
\got{1}{F} \got{1}{B} \gnl
\gcl{1} \gcl{1} \gnl
\gcl{2}  \gcu{1} \gnl
\gob{1}{F}
\gend
\end{equation}
\end{tabular}
\end{center} 
and the following compatibility condition is fulfilled:
\begin{equation}\eqlabel{psi-lambda-phi for bimonads}
\gbeg{3}{5}
\got{1}{B'} \got{1}{F} \got{1}{B} \gnl
\glmptb \gnot{\hspace{-0,34cm}\psi} \grmptb \gcl{1} \gnl
\gcl{1} \glmptb \gnot{\hspace{-0,34cm}\lambda} \grmptb \gnl
\glmptb \gnot{\hspace{-0,34cm}\phi} \grmptb \gcl{1} \gnl
\gob{1}{B'} \gob{1}{F} \gob{1}{B}
\gend=
\gbeg{3}{5}
\got{1}{B'} \got{1}{F} \got{1}{B} \gnl
\gcl{1} \glmptb \gnot{\hspace{-0,34cm}\phi} \grmptb \gnl
\glmptb \gnot{\hspace{-0,34cm}\lambda'} \grmptb \gcl{1} \gnl
\gcl{1} \glmptb \gnot{\hspace{-0,34cm}\psi} \grmptb \gnl
\gob{1}{B'} \gob{1}{F} \gob{1}{B}
\gend
\end{equation}

\medskip

For 0-cells we will write shortly $(\A, B)$. 

\medskip

\underline{2-cells:} are pairs $(\rho_M, \rho_C): (F,\psi,\phi)\to(G,\psi',\phi')$ for 1-cells
$(F,\psi,\phi),(G,\psi',\phi'):(\A, B)\to(\A', B')$,  
where $\rho_M: (F,\psi)\to(G,\psi')$ is a 2-cell in $\EM^M(\K)$ and $\rho_C: (F,\phi)\to(G,\phi')$ is a 2-cell in $\EM^C(\K)$. This means that
there exist 2-cells in $\K$: $\hat\rho_M: F\to GB$ and $\hat\rho_C: FB\to G$ such that the identities:
\vspace{-,4cm}
\begin{center} \hspace{-0,2cm}
\begin{tabular}{p{7.2cm}p{0cm}p{8cm}}
\begin{equation}  \eqlabel{2-cells EM^M}
\gbeg{3}{5}
\gvac{1} \got{1}{B'} \got{1}{F}\gnl
\gvac{1} \glmptb \gnot{\hspace{-0,34cm}\psi} \grmptb \gnl
\glmpb \gnot{\hspace{-0,34cm}\hat\rho_M} \grmptb \gcl{1} \gnl
\gcl{1} \gmu \gnl
\gob{1}{G} \gob{2}{B}
\gend=
\gbeg{3}{5}
\got{1}{B'} \got{3}{F}\gnl
\gcl{1} \glmpb \gnot{\hspace{-0,34cm}\hat\rho_M} \grmptb  \gnl
\glmptb \gnot{\hspace{-0,34cm}\psi'} \grmptb \gcl{1} \gnl
\gcl{1} \gmu \gnl
\gob{1}{G} \gob{2}{B}
\gend
\end{equation} & &
\begin{equation}\eqlabel{2-cells EM^C}
\gbeg{3}{5}
\got{1}{F} \got{2}{B}\gnl
\gcl{1} \gcmu \gnl
\glmptb \gnot{\hspace{-0,34cm}\phi} \grmptb \gcl{1} \gnl
\gcl{1} \glmpt \gnot{\hspace{-0,34cm}\hat\rho_C} \grmptb \gnl
\gob{1}{B'} \gob{3}{G} 
\gend=
\gbeg{3}{5}
\got{1}{F} \got{2}{B}\gnl
\gcl{1} \gcmu \gnl
\glmpt \gnot{\hspace{-0,34cm}\hat\rho_C} \grmptb \gcl{1} \gnl
\gvac{1} \glmptb \gnot{\hspace{-0,34cm}\phi'} \grmptb \gnl
\gvac{1} \gob{1}{B'} \gob{1}{G} 
\gend
\end{equation}
\end{tabular}
\end{center} 
hold.

\begin{rem} \rmlabel{embed into complete}
A 2-cell $\rho_M: (F,\psi)\to(G,\psi')$ in $\Mnd(\K)$ is given by a 2-cell $\rho_M: F\to G$ in $\K$. 
Given a 2-cell $\rho_M$ in $\Mnd(\K)$, if we set $\hat\rho_M=\rho_M\times\eta_B$ in the condition \equref{2-cells EM^M} we get the corresponding law for 2-cells
in $\Mnd(\K)$. This defines an inclusion 2-functor $E_M: \Mnd(\K)\to\EM^M(\K)$ which is identity on 0- and 1-cells and sends a 2-cell $\rho_M: F\to G$ to $\hat\rho_M=\rho_M\times\eta_B$.
Similarly, one has an inclusion 2-functor $E_C: \Comnd(\K)\to\EM^C(\K)$ which is identity on 0- and 1-cells and sends a 2-cell $\rho_C: F\to G$ to $\hat\rho_C=\rho_C\times\Epsilon_B$.
\end{rem}

\begin{rem} \rmlabel{notation 2-cells}
In the sequel we will write $\rho_M$ and $\rho_C$ both for 2-cells $\rho_M: (F,\psi)\to(G,\psi')$ in $\EM^M(\K)$ and $\rho_C: (F,\phi)\to(G,\phi')$ in $\EM^C(\K)$, respectively, 
and for the 2-cells $\hat\rho_M: F\to GB$ and $\hat\rho_C: FB\to G$ in $\K$, respectively, determining the latter ones.
\end{rem}

\underline{Vertical composition of 2-cells:} \hspace{0,2cm} given 2-cells
$$\scalebox{0.84}{
\bfig
\putmorphism(0,20)(1,0)[(F,\psi, \phi)` (G,\psi', \phi')`(\rho_M, \rho_C)]{800}1a
\putmorphism(1020,20)(1,0)[` (H,\psi'', \phi'')`(\rho_M', \rho_C')]{660}1a
\efig}
\hspace{2cm} \textnormal{in $\hspace{0,14cm}\bEM(\K)$}
$$
their vertical composition is given as componentwise vertical composition of 2-cells in the 2-categories $\EM^M(\K)$ and $\EM^C(\K)$. Concretely,
vertical composition of 2-cells:
$
\scalebox{0.84}{
\bfig
\putmorphism(0,20)(1,0)[(F,\psi)` (G,\psi')`\rho_M]{500}1a
\putmorphism(630,20)(1,0)[` (H,\psi'')`\rho_M']{400}1a
\efig}
\qquad$
and
$
\scalebox{0.84}{
\bfig
\putmorphism(0,20)(1,0)[(F,\phi)` (G,\phi')`\rho_C]{500}1a
\putmorphism(630,20)(1,0)[` (H,\phi'')`\rho_C']{400}1a
\efig}
\qquad$
in $\EM^M(\K)$ and $\EM^C(\K)$, respectively, is given by:
\begin{center} \hspace{-1,4cm}
\begin{tabular}{p{5.2cm}p{1cm}p{5cm}}
\begin{equation} \eqlabel{vert comp M}
\rho_M'\comp\rho_M=
\gbeg{3}{5}
\got{3}{F}\gnl
\gvac{1} \glmptb \gnot{\hspace{-0,34cm}\rho_M} \grmpb \gnl
\glmpb \gnot{\hspace{-0,34cm}\rho_M'} \grmptb \gcl{1} \gnl
\gcl{1} \gmu \gnl
\gob{1}{H} \gob{2}{B}
\gend
\end{equation} &    &
\begin{equation}\eqlabel{vert comp C}
\rho_C'\comp\rho_C=
\gbeg{3}{5}
\got{1}{F} \got{2}{B}\gnl
\gcl{1} \gcmu \gnl
\glmpt \gnot{\hspace{-0,34cm}\rho_C} \grmptb \gcl{1} \gnl
\gvac{1} \glmptb \gnot{\hspace{-0,34cm}\rho_C'} \grmpt \gnl
\gob{3}{H}
\gend
\end{equation}
\end{tabular}
\end{center}

\underline{Horizontal composition of 2-cells:} given 2-cells $(\rho_M, \rho_C):(F,\psi,\phi)\to(F', \psi',\phi')$ and $(\rho_M', \rho_C'):(G,\crta\psi, \crta\phi)\to(G', \crta\psi', \crta\phi')$
where $(F,\psi, \phi), (F', \psi', \phi'):(\A, B)\to(\B, B')$ and $(G,\crta\psi, \crta\phi), (G', \crta\psi', \crta\phi'):(\B, B')\to(\C, B'')$ are 1-cells.
Note that we have:
$$\psi: B'F\to FB, \quad \crta\psi: B''G\to GB' \quad\textnormal{and}\quad \phi: FB\to B'F, \quad \crta\phi: GB'\to B''G$$
and similarly for $\psi',\crta\psi'$ and $\phi',\crta\phi'$.
Then we have 2-cells $\rho_M:(F,\psi)\to(F', \psi')$ and $\rho_M':(G,\crta\psi)\to(G', \crta\psi')$ in $\EM^M(\K)$ and 2-cells
$\rho_C:(F,\phi)\to(F', \phi')$ and $\rho_C':(G,\crta\phi)\to(G', \crta\phi')$ in $\EM^C(\K)$.
We first describe the horizontal composition of 1-cells:
$$(G,\crta\psi, \crta\phi)(F,\psi, \phi)=(GF, \hspace{0,3cm}
\gbeg{3}{4}
\got{1}{B''} \got{1}{G} \got{1}{F} \gnl
\glmptb \gnot{\hspace{-0,34cm}\crta\psi} \grmptb \gcl{1} \gnl
\gcl{1} \glmptb \gnot{\hspace{-0,34cm}\psi} \grmptb \gnl
\gob{1}{G} \gob{1}{F} \gob{1}{B}
\gend \hspace{0,1cm}, \hspace{0,3cm}
\gbeg{3}{4}
\got{1}{G} \got{1}{F} \got{1}{B} \gnl
\gcl{1} \glmptb \gnot{\hspace{-0,34cm}\phi} \grmptb \gnl
\glmptb \gnot{\hspace{-0,34cm}\crta\phi} \grmptb \gcl{1} \gnl
\gob{1}{B''} \gob{1}{G} \gob{1}{F}
\gend \hspace{0,1cm}).$$
Now the horizontal composition of 2-cells in $\bEM(\K)$ is given by componentwise horizontal composition of 2-cells in $\EM^M(\K)$ and $\EM^C(\K)$:
\begin{center} 
\begin{tabular}{p{7.2cm}p{0cm}p{8cm}}
\begin{equation} \eqlabel{horiz comp M} \rho_M'\times\rho_M=
\gbeg{4}{6}
\gvac{1} \got{1}{G} \got{3}{F} \gnl
\gvac{1} \glmptb \gnot{\hspace{-0,34cm}\rho_M'} \grmpb \gcl{1}\gnl
\gcn{1}{1}{3}{1} \gvac{1} \glmptb \gnot{\hspace{-0,34cm}\psi} \grmptb \gnl
\gcl{1} \glmpb \gnot{\hspace{-0,34cm}\rho_M} \grmptb \gcl{1} \gnl
\gcl{1} \gcl{1} \gmu \gnl
\gob{1}{G'} \gob{1}{F'} \gob{3}{B}
\gend=
\gbeg{4}{5}
\got{1}{G} \got{3}{F} \gnl
\glmptb \gnot{\hspace{-0,34cm}\rho_M'} \grmpb \glmptb \gnot{\hspace{-0,34cm}\rho_M} \grmpb  \gnl
\gcl{1} \glmptb \gnot{\hspace{-0,34cm}\psi'} \grmptb \gcl{1} \gnl
\gcl{1} \gcl{1} \gmu \gnl
\gob{1}{G'} \gob{1}{F'} \gob{2}{B}
\gend
\end{equation} &   &
\begin{equation}\eqlabel{horiz comp C}
\rho_C'\times\rho_C=
\gbeg{4}{6}
\got{1}{G} \got{1}{F} \got{2}{B} \gnl
\gcl{1} \gcl{1} \gcmu \gnl
\gcl{1} \glmpt \gnot{\hspace{-0,34cm}\rho_C} \grmptb \gcl{1} \gnl
\gcn{1}{1}{1}{3} \gvac{1} \glmptb \gnot{\hspace{-0,34cm}\phi'} \grmptb \gnl
\gvac{1} \glmpt \gnot{\hspace{-0,34cm}\rho_C'} \grmptb \gcl{1}\gnl
\gvac{2} \gob{1}{G'} \gob{1}{F'}
\gend=
\gbeg{4}{5}
\got{1}{G} \got{1}{F} \got{2}{B} \gnl
\gcl{1} \gcl{1} \gcmu \gnl
\gcl{1} \glmptb \gnot{\hspace{-0,34cm}\phi} \grmptb \gcl{1} \gnl
\glmpt \gnot{\hspace{-0,34cm}\rho_C'} \grmptb \glmpt \gnot{\hspace{-0,34cm}\rho_C} \grmptb  \gnl
\gob{3}{G'} \gob{1}{F'}
\gend
\end{equation}
\end{tabular}
\end{center}

\underline{Identity of 0-cells:} the identity 1-cell on a 0-cell $(\A, B)$ is given by: $(\id_{\A}, \Id_B, \Id_B): (\A, B)\to(\A, B)$.

\underline{Identity of 1-cells:} the identity 2-cell on a 1-cell $(F,\psi, \phi): (\A, B)\to(\B, S)$ is given by the pair of 2-cells $(\Id_M, \Id_C)$ which in turn are determined by:
$\Id_M=\Id_F\times\eta_B: F\to FB$ and $\Id_C=\Id_F\times\Epsilon_B: FB\to F$. In string diagrams:
\begin{equation} \eqlabel{unit 2-cells}
\Id_M=
\gbeg{3}{4}
\got{1}{F}\gnl
\gcl{1} \gu{1} \gnl
\gcl{1} \gcl{1} \gnl
\gob{1}{F} \gob{1}{B}
\gend\qquad\textnormal{and}\qquad
\Id_M=
\gbeg{3}{4}
\got{1}{F} \got{1}{B} \gnl
\gcl{1} \gcl{1} \gnl
\gcl{1} \gcu{1} \gnl
\gob{1}{F}
\gend
\end{equation}

\begin{rem} \rmlabel{sides}
The 2-cell part of a 1-cell in $\Mnd(\K)$, and thus also in $\EM^M(\K)$, can be taken to act in any of the two directions. This leads to left and right versions of the 2-categories $\Mnd(\K)$
and $\EM^M(\K)$. This in turn determines the side used in the definition of $\EM^C(\K)$. We use the version fixed in \cite{Street} and \cite{LS} without further mention of the sides. This would 
be the left hand-side version. 
\end{rem}

\subsection{The 2-category of bimonads, inclusion and underlying 2-functors} \sslabel{Inc and proj}

We have not defined the 2-category $\Bimnd(\K)$ of bimonads in $\K$ so far. Its 0- and 1-cells are the same as in $\bEM(\K)$ and the 2-cells are pairs $(\rho_M, \rho_C)$ where $\rho_M$ is a 2-cell in $\Mnd(\K)$
and $\rho_C$ is a 2-cell in $\Comnd(\K)$. Then we have inclusion and projection 2-functors $E_B: \Bimnd(\K)\to\bEM(\K)$ and $\pi: \bEM(\K)\to\Bimnd(\K)$ where both are identities on 0- and 1-cells.
For a 2-cell $(\rho_M, \rho_C)$ from $\Bimnd(\K)$ we have: $E_B((\rho_M, \rho_C))=(\rho_M\times\eta_B, \rho_C\times\Epsilon_B)$, while for a
2-cell $(\tau_M, \tau_C)$ from $\bEM(\K)$ we have: $\pi((\tau_M, \tau_C))=((G\times\Epsilon_B)\comp\tau_M, \tau_C\comp(F\times\eta_B)$. 
Then it is clear that $\pi\comp E_B=\Id$.

We also have the inclusion 2-functor $\Inc:\K\to\Bimnd(\K)$, which sends a 0-cell $\A$ into the identity bimonad $(\A, \Id_{\A})$ and so on, and the underlying functor $\U: \Bimnd(\K)\to\K$,
so that $\Inc$ is right 2-adjoint to $\U$. 


\begin{defn}
We say that $\K$ admits the Eilenberg--Moore construction for bimonads if the 2-functor $\Inc:\K\to\Bimnd(\K)$ has a right 2-adjoint.
\end{defn}

We conjecture that the 2-category $\bEM(\K)$ is the free completion of $\Bimnd(\K)$ under the Eilenberg-Moore construction for bimonads.

\section{Biwreaths}  \selabel{biwreaths}

We define a biwreath as a bimonad in the 2-category $\bEM(\K)$, in other words as a 0-cell in the 2-category $\bEM(\bEM(\K))$. Let us write out what this means.

On a 0-cell $(\A, B, \mu, \eta, \Delta, \Epsilon, \lambda)\equiv(\A, B)$ in $\bEM(\K)$, which is a bimonad in $\K$, \deref{bimonad}, we have an endomorphism 1-cell
$(F,\psi,\phi): (\A, B)\to(\A, B)$, that is a 1-cell $F:\A\to\A$ in $\K$, with 2-cells $\psi: BF\to FB$ and $\phi: FB\to BF$ in $\K$ so that:
\begin{center} \hspace{-0,8cm}
\begin{tabular}{p{7.4cm}p{0cm}p{8cm}}
\begin{equation}\eqlabel{psi laws for B}
\gbeg{3}{5}
\got{1}{B}\got{1}{B}\got{1}{F}\gnl
\gcl{1} \glmpt \gnot{\hspace{-0,34cm}\psi} \grmptb \gnl
\glmptb \gnot{\hspace{-0,34cm}\psi} \grmptb \gcl{1} \gnl
\gcl{1} \gmu \gnl
\gob{1}{F} \gob{2}{B}
\gend=
\gbeg{3}{5}
\got{1}{B}\got{1}{B}\got{1}{F}\gnl
\gmu \gcn{1}{1}{1}{0} \gnl
\gvac{1} \hspace{-0,34cm} \glmptb \gnot{\hspace{-0,34cm}\psi} \grmptb  \gnl
\gvac{1} \gcl{1} \gcl{1} \gnl
\gvac{1} \gob{1}{F} \gob{1}{B}
\gend;
\quad
\gbeg{2}{5}
\got{3}{F} \gnl
\gu{1} \gcl{1} \gnl
\glmptb \gnot{\hspace{-0,34cm}\psi} \grmptb \gnl
\gcl{1} \gcl{1} \gnl
\gob{1}{F} \gob{1}{B}
\gend=
\gbeg{2}{5}
\got{1}{F} \gnl
\gcl{1} \gu{1} \gnl
\gcl{2} \gcl{2} \gnl
\gob{1}{F} \gob{1}{B}
\gend
\end{equation} & &
\begin{equation}\eqlabel{phi laws for B}
\gbeg{3}{5}
\got{1}{F} \got{2}{B}\gnl
\gcl{1} \gcmu \gnl
\glmptb \gnot{\hspace{-0,34cm}\phi} \grmptb \gcl{1} \gnl
\gcl{1} \glmptb \gnot{\hspace{-0,34cm}\phi} \grmptb \gnl
\gob{1}{B} \gob{1}{B} \gob{1}{F}
\gend=
\gbeg{3}{5}
\got{2}{F} \got{1}{\hspace{-0,2cm}B}\gnl
\gcn{1}{1}{2}{2} \gcn{1}{1}{2}{2} \gnl
\gvac{1} \hspace{-0,34cm} \glmpt \gnot{\hspace{-0,34cm}\phi} \grmptb \gnl
\gvac{1} \hspace{-0,2cm} \gcmu \gcn{1}{1}{0}{1} \gnl
\gvac{1} \gob{1}{B} \gob{1}{B} \gob{1}{F}
\gend;
\quad
\gbeg{2}{5}
\got{1}{F} \got{1}{B} \gnl
\gcl{1} \gcl{1} \gnl
\glmptb \gnot{\hspace{-0,34cm}\phi} \grmptb \gnl
\gcu{1} \gcl{1} \gnl
\gob{3}{F}
\gend=
\gbeg{2}{5}
\got{1}{F} \got{1}{B} \gnl
\gcl{1} \gcl{1} \gnl
\gcl{2}  \gcu{1} \gnl
\gob{1}{F}
\gend
\end{equation}
\end{tabular}
\end{center} 
and
\begin{equation}\eqlabel{psi-lambda-phi}
\gbeg{3}{5}
\got{1}{B} \got{1}{F} \got{1}{B} \gnl
\glmptb \gnot{\hspace{-0,34cm}\psi} \grmptb \gcl{1} \gnl
\gcl{1} \glmptb \gnot{\hspace{-0,34cm}\lambda} \grmptb \gnl
\glmptb \gnot{\hspace{-0,34cm}\phi} \grmptb \gcl{1} \gnl
\gob{1}{B} \gob{1}{F} \gob{1}{B}
\gend=
\gbeg{3}{5}
\got{1}{B} \got{1}{F} \got{1}{B} \gnl
\gcl{1} \glmptb \gnot{\hspace{-0,34cm}\phi} \grmptb \gnl
\glmptb \gnot{\hspace{-0,34cm}\lambda} \grmptb \gcl{1} \gnl
\gcl{1} \glmptb \gnot{\hspace{-0,34cm}\psi} \grmptb \gnl
\gob{1}{B} \gob{1}{F} \gob{1}{B}
\gend
\end{equation}
hold, together with bimonad structure morphisms, that is, 2-cells in $\bEM(\K)$:
$$(\mu_M, \mu_C): (F,\psi,\phi)(F,\psi,\phi)\to(F,\psi,\phi)$$
$$(\eta_M, \eta_C): (\Id_{\A}, id_B, id_B)\to(F,\psi,\phi)$$
\begin{equation} \eqlabel{structure 2-cells}
(\Delta_M, \Delta_C): (F,\psi,\phi)\to(F,\psi,\phi)(F,\psi,\phi)
\end{equation}
$$(\Epsilon_M, \Epsilon_C): (F,\psi,\phi)\to(\Id_{\A}, id_B, id_B)$$
$$(\lambda_M, \lambda_C): (F,\psi,\phi)(F,\psi,\phi)\to(F,\psi,\phi)(F,\psi,\phi)$$
which are given by the following 2-cells in $\K$:
\begin{equation} \eqlabel{structure 2-cells in K}
\mu_M:FF\to FB, \quad \eta_M:\Id_{\A}\to FB, \quad \Delta_M: F\to FFB, \quad \Epsilon_M: F\to B, \quad \lambda_M:FF\to FFB
\end{equation}
$$\mu_C:FFB\to F, \quad \eta_C:B\to F, \quad \Delta_C: FB\to FF, \quad \Epsilon_C: FB\to\Id_{\A}, \quad \lambda_C:FFB\to FF$$
which obey (the 2-cell conditions):
\begin{center} \hspace{-1,4cm}
\begin{tabular}{p{6.2cm}p{1cm}p{7cm}}
\begin{equation} \eqlabel{2-cell mu_M}
\gbeg{3}{6}
\got{1}{B} \got{1}{F} \got{1}{F} \gnl
\glmptb \gnot{\hspace{-0,34cm}\psi} \grmptb \gcl{1} \gnl
\gcl{1} \glmptb \gnot{\hspace{-0,34cm}\psi} \grmptb \gnl
\glmptb \gnot{\hspace{-0,34cm}\mu_M} \grmptb \gcl{1} \gnl
\gcl{1} \gmu \gnl
\gob{1}{F} \gob{2}{B}
\gend=
\gbeg{3}{5}
\got{1}{B} \got{1}{F} \got{1}{F}\gnl
\gcl{1} \glmptb \gnot{\hspace{-0,34cm}\mu_M} \grmptb \gnl
\glmptb \gnot{\hspace{-0,34cm}\psi} \grmptb \gcl{1} \gnl
\gcl{1} \gmu \gnl
\gob{1}{F} \gob{2}{B}
\gend
\end{equation} &   &
\begin{equation}\eqlabel{2-cell mu_C}
\gbeg{4}{6}
\got{1}{F} \got{1}{F} \got{2}{B} \gnl
\gcl{1} \gcl{1} \gcmu \gnl
\gcl{1} \glmptb \gnot{\hspace{-0,34cm}\phi} \grmptb \gcl{2} \gnl
\glmptb \gnot{\hspace{-0,34cm}\phi} \grmptb \gcl{1} \gnl
\gcl{1} \glmpt \gnot{\mu_C} \gcmptb \grmpt \gnl
\gob{1}{B} \gob{3}{F}
\gend=
\gbeg{4}{6}
\got{1}{F} \got{1}{F} \got{2}{B} \gnl
\gcl{1} \gcl{1} \gcmu \gnl
\glmpt \gnot{\mu_C} \gcmpt \grmptb \gcl{1} \gnl
\gvac{2} \glmptb \gnot{\hspace{-0,34cm}\phi} \grmptb \gnl
\gvac{2} \gcl{1} \gcl{1} \gnl
\gvac{2} \gob{1}{B} \gob{1}{F}
\gend
\end{equation}
\end{tabular}
\end{center}

\begin{center} \hspace{-1cm}
\begin{tabular}{p{5.2cm}p{3cm}p{5cm}}
\begin{equation}\eqlabel{2-cell eta_M}
\gbeg{3}{4}
\got{1}{} \got{3}{B}\gnl
\glmpb \gnot{\hspace{-0,34cm}\eta_M} \grmpb \gcl{1} \gnl
\gcl{1} \gmu \gnl
\gob{1}{F} \gob{2}{B}
\gend=
\gbeg{3}{5}
\got{1}{B}\gnl
\gcl{1} \glmpb \gnot{\hspace{-0,34cm}\eta_M} \grmpb \gnl
\glmptb \gnot{\hspace{-0,34cm}\psi} \grmptb \gcl{1} \gnl
\gcl{1} \gmu \gnl
\gob{1}{F} \gob{2}{B}
\gend
\end{equation} &    &
\begin{equation}\eqlabel{2-cell eta_C}
\gbeg{2}{5}
\got{2}{B}\gnl
\gcmu \gnl
\gcl{1} \gbmp{\eta_C} \gnl
\gcl{1} \gcl{1} \gnl
\gob{1}{B} \gob{1}{F}
\gend=
\gbeg{3}{5}
\got{2}{B}\gnl
\gcmu \gnl
\gbmp{\eta_C} \gcl{1} \gnl
\glmptb \gnot{\hspace{-0,34cm}\phi} \grmptb  \gnl
\gob{1}{B} \gob{1}{F}
\gend
\end{equation}
\end{tabular}
\end{center}

\begin{center} \hspace{-2,6cm}
\begin{tabular}{p{7.2cm}p{2cm}p{5cm}}
\begin{equation} \eqlabel{2-cell Delta_M}
\gbeg{4}{5}
\gvac{2} \got{1}{B} \got{1}{F} \gnl
\gvac{2} \glmptb \gnot{\hspace{-0,34cm}\psi} \grmptb \gnl
\glmpb \gnot{\Delta_M} \gcmpb \grmptb \gcl{1} \gnl
\gcl{1} \gcl{1} \gmu \gnl
\gob{1}{F} \gob{1}{F} \gob{2}{B}
\gend=
\gbeg{3}{6}
\got{1}{B} \got{1}{F} \gnl
\gcl{1} \glmptb \gnot{\Delta_M} \gcmpb \grmpb \gnl
\glmptb \gnot{\hspace{-0,34cm}\psi} \grmptb \gcl{1} \gcl{2} \gnl
\gcl{1} \glmptb \gnot{\hspace{-0,34cm}\psi} \grmptb \gnl
\gcl{1} \gcl{1} \gmu \gnl
\gob{1}{F} \gob{1}{F} \gob{2}{B}
\gend
\end{equation} &    &  \vspace{-0,4cm}
\begin{equation}\eqlabel{2-cell Delta_C}
\gbeg{3}{5}
\got{1}{F} \got{2}{B} \gnl
\gcl{1} \gcmu \gnl
\glmptb \gnot{\hspace{-0,34cm}\phi} \grmptb \gcl{1} \gnl
\gcl{1} \glmptb \gnot{\hspace{-0,34cm}\Delta_C} \grmptb \gnl
\gob{1}{B} \gob{1}{F} \gob{1}{F}
\gend=
\gbeg{3}{6}
\got{1}{F} \got{2}{B} \gnl
\gcl{1} \gcmu \gnl
\glmptb \gnot{\hspace{-0,34cm}\Delta_C} \grmptb \gcl{1} \gnl
\gcl{1} \glmptb \gnot{\hspace{-0,34cm}\phi} \grmptb \gnl
\glmptb \gnot{\hspace{-0,34cm}\phi} \grmptb \gcl{1} \gnl
\gob{1}{B} \gob{1}{F} \gob{1}{F}
\gend
\end{equation}
\end{tabular}
\end{center}

\begin{center} \hspace{-1,4cm}
\begin{tabular}{p{5.2cm}p{2.6cm}p{5cm}}
\begin{equation}\eqlabel{2-cell Epsilon_M}
\gbeg{3}{5}
\got{1}{B} \got{1}{F} \gnl
\glmptb \gnot{\hspace{-0,34cm}\psi} \grmptb \gnl
\gbmp{\Epsilon_M} \gcl{1} \gnl
\gmu \gnl
\gob{2}{B}
\gend=
\gbeg{3}{5}
\got{1}{B} \got{1}{F} \gnl
\gcl{1} \gcl{1} \gnl
\gcl{1} \gbmp{\Epsilon_M} \gnl
\gmu \gnl
\gob{2}{B}
\gend
\end{equation} &    &
\begin{equation}\eqlabel{2-cell Epsilon_C}
\gbeg{3}{5}
\got{1}{F} \got{2}{B}\gnl
\gcl{1} \gcmu \gnl
\glmpb \gnot{\hspace{-0,34cm}\phi} \grmpb \gcl{1} \gnl
\gcl{1} \glmpt \gnot{\hspace{-0,34cm}\Epsilon_C} \grmpt \gnl
\gob{1}{B}
\gend=
\gbeg{3}{5}
\got{1}{F} \got{2}{B}\gnl
\gcl{1} \gcmu \gnl
\glmpt \gnot{\hspace{-0,34cm}\Epsilon_C} \grmpt \gcl{2} \gnl
\gob{5}{B}
\gend
\end{equation}
\end{tabular}
\end{center}

\begin{center} 
\begin{tabular}{p{6cm}p{1cm}p{7cm}}
\begin{equation} \eqlabel{2-cell lambda_M}
\gbeg{4}{6}
\gvac{1} \got{1}{B} \got{1}{F} \got{1}{F} \gnl
\gvac{1} \glmptb \gnot{\hspace{-0,34cm}\psi} \grmptb \gcl{1} \gnl
\gvac{1} \gcl{1} \glmptb \gnot{\hspace{-0,34cm}\psi} \grmptb \gnl
\glmpb \gnot{\lambda_M} \gcmptb \grmptb \gcl{1} \gnl
\gcl{1} \gcl{1} \gmu \gnl
\gob{1}{F} \gob{1}{F} \gob{2}{B}
\gend=
\gbeg{4}{6}
\got{1}{B} \got{1}{F} \got{1}{F} \gnl
\gcl{1} \glmptb \gnot{\lambda_M} \gcmptb \grmpb \gnl
\glmptb \gnot{\hspace{-0,34cm}\psi} \grmptb \gcl{1} \gcl{1} \gnl
\gcl{1} \glmptb \gnot{\hspace{-0,34cm}\psi} \grmptb \gcl{1} \gnl
\gcl{1} \gcl{1} \gmu \gnl
\gob{1}{F} \gob{1}{F} \gob{2}{B}
\gend\end{equation} &    &  \vspace{0,1cm}
\begin{equation}\eqlabel{2-cell lambda_C}
\gbeg{4}{6}
\got{1}{F} \got{1}{F} \got{2}{B} \gnl
\gcl{1} \gcl{1} \gcmu \gnl
\gcl{1} \glmptb \gnot{\hspace{-0,34cm}\phi} \grmptb \gcl{2} \gnl
\glmptb \gnot{\hspace{-0,34cm}\phi} \grmptb \gcl{1} \gnl
\gcl{1} \glmptb \gnot{\lambda_C} \gcmptb \grmpt \gnl
\gob{1}{B} \gob{1}{F} \gob{1}{F}
\gend=
\gbeg{4}{6}
\got{1}{F} \got{1}{F} \got{2}{B} \gnl
\gcl{1} \gcl{1} \gcmu \gnl
\glmpt \gnot{\lambda_C} \gcmpt \grmptb \gcl{1} \gnl
\gvac{1} \gcl{1} \glmptb \gnot{\hspace{-0,34cm}\phi} \grmptb \gnl
\gvac{1} \glmptb \gnot{\hspace{-0,34cm}\phi} \grmptb \gcl{1} \gnl
\gvac{1} \gob{1}{B} \gob{1}{F} \gob{1}{F}
\gend \quad ;
\end{equation}
\end{tabular}
\end{center}
the 2-cells \equref{structure 2-cells} need to satisfy the monad, comonad and bimonad compatibility laws in $\bEM(\K)$. Having in mind that the two components of the identity 2-cell on $(F, \psi, \phi)$
are given by $\Id_F\times\eta_B$ and $\Id_F\times\Epsilon_B$, and the way how vertical and horizontal compositions are defined in $\bEM(\K)$, these laws translate to the following conditions in $\K$: 
\vspace{-0,4cm}
\begin{center} \hspace{0,4cm}
\begin{tabular}{p{5.2cm}p{2cm}p{6.8cm}}
\begin{equation} \eqlabel{monad law mu_M}
\gbeg{3}{5}
\got{1}{F} \got{1}{F} \got{1}{F}\gnl
\gcl{1} \glmptb \gnot{\hspace{-0,34cm}\mu_M} \grmptb \gnl
\glmptb \gnot{\hspace{-0,34cm}\mu_M} \grmptb \gcl{1} \gnl
\gcl{1} \gmu \gnl
\gob{1}{F} \gob{2}{B}
\gend=
\gbeg{3}{6}
\got{1}{F} \got{1}{F} \got{1}{F} \gnl
\glmptb \gnot{\hspace{-0,34cm}\mu_M} \grmptb \gcl{1} \gnl
\gcl{1} \glmptb \gnot{\hspace{-0,34cm}\psi} \grmptb \gnl
\glmptb \gnot{\hspace{-0,34cm}\mu_M} \grmptb \gcl{1} \gnl
\gcl{1} \gmu \gnl
\gob{1}{F} \gob{2}{B}
\gend
\end{equation} &  & \vspace{0,1cm}
\begin{equation}\eqlabel{monad law eta_M}
\gbeg{3}{6}
\got{1}{} \got{3}{F}\gnl
\glmpb \gnot{\hspace{-0,34cm}\eta_M} \grmpb \gcl{1} \gnl
\gcl{1} \glmptb \gnot{\hspace{-0,34cm}\psi} \grmptb \gnl
\glmptb \gnot{\hspace{-0,34cm}\mu_M} \grmptb \gcl{1} \gnl
\gcl{1} \gmu \gnl
\gob{1}{F} \gob{2}{B}
\gend=
\gbeg{2}{4}
\got{1}{F}\gnl
\gcl{1} \gu{1} \gnl
\gcl{1} \gcl{1} \gnl
\gob{1}{F} \gob{1}{B}
\gend=
\gbeg{3}{5}
\got{1}{F}\gnl
\gcl{1} \glmpb \gnot{\hspace{-0,34cm}\eta_M} \grmpb \gnl
\glmptb \gnot{\hspace{-0,34cm}\mu_M} \grmptb \gcl{1} \gnl
\gcl{1} \gmu \gnl
\gob{1}{F} \gob{2}{B}
\gend
\end{equation} 
\end{tabular}
\end{center} \vspace{-0,4cm}
$$ \textnormal{ \hspace{-1cm} \footnotesize monad law for $\mu_M$}  \hspace{6,5cm}  \textnormal{\footnotesize monad law for $\eta_M$} $$ \vspace{-0,7cm}

\begin{center} \hspace{-1,4cm}
\begin{tabular}{p{8cm}p{1cm}p{6.8cm}}
\begin{equation} \eqlabel{monad law mu_C}
\gbeg{5}{6}
\got{1}{F} \got{1}{F} \got{1}{F} \got{2}{B} \gnl
\gcl{1} \gcl{1} \gcl{1} \gcmu \gnl
\gcl{1} \glmpt \gnot{\mu_C} \gcmptb \grmpt \gcl{1} \gnl
\glmpt \gnot{\hspace{0,8cm} \mu_C} \gcmp \gcmptb \gcmp \grmp  \gnl
\gvac{2} \gcl{1}  \gnl
\gob{5}{F}
\gend=
\gbeg{4}{6}
\got{1}{F} \got{1}{F} \got{1}{F} \got{2}{B} \gnl
\gcl{2} \gcl{2} \gcl{1} \gcmu \gnl
\gvac{2} \glmptb \gnot{\hspace{-0,34cm}\phi} \grmptb \gcl{2} \gnl
\glmpt \gnot{\mu_C} \gcmpt \grmpt \gcl{1} \gnl
\gvac{2} \glmptb \gnot{\mu_C} \gcmpt \grmpt \gnl
\gob{5}{F}
\gend
\end{equation} &  &  \vspace{-0,8cm}
\begin{equation}\eqlabel{monad law eta_C}
\gbeg{3}{6}
\got{1}{F} \got{2}{B} \gnl
\gcl{1}\gcmu \gnl
\gcl{1} \gbmp{\s\eta_{C}} \gcl{1} \gnl
\glmpt \gnot{\mu_C} \gcmptb \grmpt  \gnl
\gvac{1} \gcl{1} \gnl
\gob{3}{F}
\gend=
\gbeg{2}{4}
\got{1}{F} \got{1}{B} \gnl
\gcl{1} \gcl{1} \gnl
\gcl{1} \gcu{1} \gnl
\gob{1}{F}
\gend=
\gbeg{4}{6}
\got{1}{F} \got{2}{B} \gnl
\gcl{1}\gcmu \gnl
\glmptb \gnot{\hspace{-0,34cm}\phi} \grmptb \gcl{2} \gnl
\gbmp{\s\eta_{C}} \gcl{1} \gnl
\glmpt \gnot{\mu_C} \gcmptb \grmpt  \gnl
\gob{3}{F}
\gend
\end{equation} \vspace{-0,4cm}
\\ {\hspace{2,5cm} \footnotesize monad law for $\mu_C$} & &  {\hspace{2cm} \footnotesize monad law for $\eta_C$}
\end{tabular}
\end{center}

\begin{center} \hspace{-0,4cm}
\begin{tabular}{p{7.2cm}p{1cm}p{6.8cm}}
\begin{equation} \eqlabel{comonad law Delta_M}
\gbeg{5}{6}
\got{7}{F} \gnl
\gvac{2} \glmp \gnot{\Delta_M} \gcmptb \grmpb  \gnl
\glmpb \gnot{\Delta_M} \gcmpb \grmptb \gcl{1} \gcl{1} \gnl
\gcl{2} \gcl{2}  \glmptb \gnot{\hspace{-0,34cm}\psi} \grmptb \gcl{1} \gnl
\gvac{2} \gcl{1} \gmu \gnl
\gob{1}{F} \gob{1}{F} \gob{1}{F} \gob{2}{B} \gnl
\gend=
\gbeg{5}{6}
\got{3}{F} \gnl
\glmpb \gnot{\Delta_M} \gcmptb \grmpb  \gnl
\gcl{1} \gcl{1} \gcn{1}{1}{1}{5} \gnl
\gcl{1} \glmptb \gnot{\Delta_M} \gcmpb \grmpb \gcl{1} \gnl
\gcl{1} \gcl{1} \gcl{1} \gmu \gnl
\gob{1}{F} \gob{1}{F} \gob{1}{F} \gob{2}{B} \gnl
\gend
\end{equation} &  & \vspace{-0,4cm}
\begin{equation}\eqlabel{comonad law Epsilon_M}
\gbeg{3}{6}
\got{3}{F} \gnl
\glmpb \gnot{\Delta_M} \gcmptb \grmpb  \gnl
\gbmp{\Epsilon_M} \gcl{1} \gcl{2} \gnl
\glmptb \gnot{\hspace{-0,34cm}\psi} \grmptb \gnl
\gcl{1} \gmu \gnl
\gob{1}{F} \gob{2}{B} \gnl
\gend=
\gbeg{2}{4}
\got{1}{F} \gnl
\gcl{1} \gu{1} \gnl
\gcl{1} \gcl{1} \gnl
\gob{1}{F} \gob{1}{B} \gnl
\gend=
\gbeg{3}{6}
\got{3}{F} \gnl
\gvac{1} \gcl{1} \gnl
\glmpb \gnot{\Delta_M} \gcmptb \grmpb  \gnl
\gcl{1} \gbmp{\Epsilon_M} \gcl{1} \gnl
\gcl{1} \gmu \gnl
\gob{1}{F} \gob{2}{B} \gnl
\gend
\end{equation}
\\ {\qquad\qquad \footnotesize comonad law for $\Delta_M$} & &  {\qquad\qquad \footnotesize comonad law for $\Epsilon_M$}
\end{tabular}
\end{center}

\begin{center} 
\begin{tabular}{p{6cm}p{2cm}p{6.8cm}}
\begin{equation} \eqlabel{comonad law Delta_C}
\gbeg{3}{6}
\got{1}{F} \got{2}{B}\gnl
\gcl{1} \gcmu \gnl
\glmptb \gnot{\hspace{-0,34cm}\Delta_C} \grmptb \gcl{1} \gnl
\gcl{1} \glmptb \gnot{\hspace{-0,34cm}\phi} \grmptb \gnl
\glmptb \gnot{\hspace{-0,34cm}\Delta_C} \grmptb \gcl{1} \gnl
\gob{1}{F}\gob{1}{F}  \gob{1}{F}
\gend=
\gbeg{3}{5}
\got{1}{F} \got{2}{B}\gnl
\gcl{1} \gcmu \gnl
\glmptb \gnot{\hspace{-0,34cm}\Delta_C} \grmptb \gcl{1} \gnl
\gcl{1} \glmptb \gnot{\hspace{-0,34cm}\Delta_C} \grmptb \gnl
\gob{1}{F}\gob{1}{F}  \gob{1}{F}
\gend
\end{equation} &  &
\begin{equation}\eqlabel{comonad law Epsilon_C}
\gbeg{3}{6}
\got{1}{F} \got{2}{B}\gnl
\gcl{1} \gcmu \gnl
\glmptb \gnot{\hspace{-0,34cm}\Delta_C} \grmptb \gcl{1} \gnl
\gcl{1} \glmptb \gnot{\hspace{-0,34cm}\phi} \grmptb \gnl
\glmpt \gnot{\hspace{-0,34cm}\Epsilon_C} \grmpt \gcl{1} \gnl
\gob{5}{F}
\gend=
\gbeg{2}{4}
\got{1}{F} \got{1}{B} \gnl
\gcl{1} \gcl{1} \gnl
\gcl{1} \gcu{1} \gnl
\gob{1}{F}
\gend=
\gbeg{3}{5}
\got{1}{F} \got{2}{B}\gnl
\gcl{1} \gcmu \gnl
\glmptb \gnot{\hspace{-0,34cm}\Delta_C} \grmptb \gcl{1} \gnl
\gcl{1} \glmpt \gnot{\hspace{-0,34cm}\Epsilon_C} \grmpt \gnl
\gob{1}{F} \gnl
\gend
\end{equation}
\end{tabular}
\end{center} \vspace{-0,5cm}
$$ \textnormal{ \footnotesize comonad law for $\Delta_C$}  \hspace{5,5cm}  \textnormal{\footnotesize comonad law for $\Epsilon_C$} $$ 
finally, we have the 8 bimonad compatibilities (from \deref{bimonad}), we first list them for the monad- {\em i.e.} first components of the 2-cells, and then for their comonad- {\em i.e.} second components:
\begin{center} \hspace{-1,4cm}
\begin{tabular} {p{4.5cm}p{-7cm}p{5.8cm}p{-1cm}p{3.4cm}} 
\begin{equation}\eqlabel{1 lambda_M}
\gbeg{2}{4}
\got{1}{F} \got{1}{F} \gnl
\gbmp{\Epsilon_M} \gbmp{\Epsilon_M} \gnl
\gmu \gnl
\gob{2}{B} \gnl
\gend=
\gbeg{2}{5}
\got{1}{F} \got{1}{F} \gnl
\glmptb \gnot{\hspace{-0,34cm}\mu_M} \grmptb \gnl
\gbmp{\Epsilon_M} \gcl{1} \gnl
\gmu \gnl
\gob{2}{B} \gnl
\gend
\end{equation} & \qquad \qquad &
\begin{equation}\eqlabel{2 lambda_M}
\gbeg{4}{5}
\got{1}{} \gnl
\gvac{2} \glmpb \gnot{\hspace{-0,34cm}\eta_M} \grmpb \gnl
\glmpb \gnot{\Delta_M} \gcmpb \grmptb \gcl{1} \gnl
\gcl{1} \gcl{1} \gmu \gnl
\gob{1}{F} \gob{1}{F} \gob{2}{B} \gnl
\gend=
\gbeg{4}{5}
\got{1}{} \gnl
\glmpb \gnot{\hspace{-0,34cm}\eta_M} \grmpb \glmpb \gnot{\hspace{-0,34cm}\eta_M} \grmpb \gnl
\gcl{1} \glmptb \gnot{\hspace{-0,34cm}\psi} \grmptb \gcl{1} \gnl
\gcl{1} \gcl{1} \gmu \gnl
\gob{1}{F} \gob{1}{F} \gob{2}{B} \gnl
\gend
\end{equation} & \qquad\quad\quad &
\begin{equation}\eqlabel{3 lambda_M}
\gbeg{2}{5}
\got{1}{} \gnl
\glmpb \gnot{\hspace{-0,34cm}\eta_M} \grmpb \gnl
\gbmp{\Epsilon_M} \gcl{1} \gnl
\gmu \gnl
\gob{2}{B} \gnl
\gend=
\gbeg{2}{5}
\got{1}{} \gnl
\gu{1} \gnl
\gcl{2} \gnl
\gob{1}{B} \gnl
\gend
\end{equation}
\\ {\qquad\footnotesize $\Epsilon_M-\mu_M$ compatibility} & &  {\qquad \footnotesize $\Delta_M-\eta_M$ compatibility}  & &  { \footnotesize $\Epsilon_M-\eta_M$ compatibility}
\end{tabular}
\end{center}

\begin{center} 
\begin{tabular} {p{7cm}p{0cm}p{7.2cm}} 
\begin{equation}\eqlabel{4 lambda_M}
\gbeg{5}{8}
\gvac{1} \got{1}{F} \got{1}{F} \got{1}{F} \gnl
\gvac{1} \gcl{1} \glmptb \gnot{\lambda_M} \gcmptb \grmpb \gnl
\glmpb \gnot{\lambda_M} \gcmptb \grmptb \gcl{1} \gcl{3} \gnl
\gcl{1} \gcl{1} \glmptb \gnot{\hspace{-0,34cm}\psi} \grmptb \gnl
\gcl{1} \glmptb \gnot{\hspace{-0,34cm}\mu_M} \grmptb \gcl{1} \gnl
\gcl{1} \gcl{1} \gmu \gcn{1}{1}{1}{0} \gnl
\gcl{1} \gcl{1} \gvac{1} \hspace{-0,34cm} \gmu \gnl
\gob{2}{F} \gob{1}{\hspace{-0,34cm}F} \gob{2}{B} \gnl
\gend=
\gbeg{5}{6}
\gvac{1} \got{1}{F} \got{1}{F} \got{1}{F} \gnl
\gvac{1} \glmptb \gnot{\hspace{-0,34cm}\mu_M} \grmptb \gcl{1} \gnl
\gvac{1} \gcl{1} \glmptb \gnot{\hspace{-0,34cm}\psi} \grmptb \gnl
\glmpb \gnot{\lambda_M} \gcmptb \grmptb \gcl{1} \gnl
\gcl{1} \gcl{1} \gmu \gnl
\gob{1}{F} \gob{1}{F} \gob{2}{B} \gnl
\gend
\end{equation} & \qquad \qquad &  \vspace{0,4cm}
\begin{equation}\eqlabel{5 lambda_M}
\gbeg{4}{6}
\got{7}{F} \gnl
\gvac{1} \glmpb \gnot{\hspace{-0,34cm}\eta_M} \grmpb \gcl{1} \gnl
\gvac{1} \gcl{1} \glmptb \gnot{\hspace{-0,34cm}\psi} \grmptb \gnl
\glmpb \gnot{\lambda_M} \gcmptb \grmptb \gcl{1} \gnl
\gcl{1} \gcl{1} \gmu \gnl
\gob{1}{F} \gob{1}{F} \gob{2}{B} \gnl
\gend=
\gbeg{3}{5}
\got{1}{F} \gnl
\gcl{3} \glmpb \gnot{\hspace{-0,34cm}\eta_M} \grmpb \gnl
\gvac{1} \gcl{2} \gcl{2} \gnl
\gob{1}{F} \gob{1}{F} \gob{1}{B} \gnl
\gend
\end{equation}
\\ {\qquad\footnotesize $\lambda_M-\mu_M$ compatibility} & &  {\qquad\qquad \footnotesize $\lambda_M-\eta_M$ compatibility}
\end{tabular}
\end{center}

\begin{center} \hspace{-1,4cm}
\begin{tabular} {p{8cm}p{1,5cm}p{4.7cm}} 
\begin{equation}\eqlabel{6 lambda_M}
\gbeg{6}{8}
\gvac{1} \got{3}{F} \got{1}{F} \gnl
\gvac{2} \gcl{1} \glmpb \gnot{\Delta_M} \gcmptb \grmpb \gnl
\gvac{1} \glmpb \gnot{\lambda_M} \gcmptb \grmptb \gcl{1} \gcl{2} \gnl
\gcn{1}{1}{3}{1} \gvac{1} \gcl{1} \glmptb \gnot{\hspace{-0,34cm}\psi} \grmptb \gcl{1} \gnl
\gcl{1} \glmpb \gnot{\lambda_M} \gcmptb \grmptb \gmu \gnl
\gcl{1} \gcl{1} \gcl{1} \gcl{1} \gcn{1}{1}{2}{1} \gnl
\gcl{1} \gcl{1} \gcl{1} \gmu \gnl
\gob{1}{F} \gob{1}{F} \gob{1}{F} \gob{2}{B} \gnl
\gend=
\gbeg{5}{6}
\gvac{2} \got{1}{F} \got{1}{F} \gnl
\gvac{2} \glmptb \gnot{\lambda_M} \gcmptb \grmpb \gnl
\glmpb \gnot{\Delta_M} \gcmpb \grmptb \gcl{1} \gcl{2} \gnl
\gcl{1} \gcl{1} \glmptb \gnot{\hspace{-0,34cm}\psi} \grmptb \gnl
\gcl{1} \gcl{1} \gcl{1} \gmu \gnl
\gob{1}{F} \gob{1}{F} \gob{1}{F} \gob{2}{B} \gnl
\gend
\end{equation} &  &
\begin{equation}\eqlabel{7 lambda_M}
\gbeg{3}{6}
\got{1}{F} \got{1}{F} \gnl
\glmptb \gnot{\lambda_M} \gcmptb \grmpb \gnl
\gbmp{\Epsilon_M} \gcl{1} \gcl{2} \gnl
\glmptb \gnot{\hspace{-0,34cm}\psi} \grmptb \gnl
\gcl{1} \gmu \gnl
\gob{1}{F} \gob{2}{B} \gnl
\gend=
\gbeg{2}{5}
\got{1}{F} \got{1}{F} \gnl
\gcl{1} \gcl{1} \gnl
\gcl{1} \gbmp{\Epsilon_M}  \gnl
\gcl{1} \gcl{1} \gnl
\gob{1}{F} \gob{1}{B} \gnl
\gend
\end{equation}
\\  {\qquad\qquad\quad \footnotesize $\lambda_M-\Delta_M$ compatibility}  & & {\quad\footnotesize $\lambda_M-\Epsilon_M$ compatibility}
\end{tabular}
\end{center}

\begin{equation}\eqlabel{8 lambda_M}
\gbeg{4}{6}
\gvac{2} \got{1}{F} \got{1}{F} \gnl
\gvac{2} \glmptb \gnot{\hspace{-0,34cm}\mu_M} \grmptb \gnl
\glmpb \gnot{\Delta_M} \gcmpb \grmptb \gcl{1} \gnl
\gcl{2} \gcl{2} \gmu \gnl
\gvac{2} \gcn{1}{1}{2}{2} \gnl
\gob{1}{F} \gob{1}{F} \gob{2}{B} \gnl
\gend=
\gbeg{5}{8}
\gvac{1} \got{1}{F} \got{1}{F} \gnl
\gvac{1} \gcl{1} \glmptb \gnot{\Delta_M} \gcmpb \grmpb \gnl
\glmpb \gnot{\lambda_M} \gcmptb \grmptb \gcl{1} \gcl{2} \gnl
\gcl{1} \gcl{1} \glmptb \gnot{\hspace{-0,34cm}\psi} \grmptb \gnl
\gcl{2} \glmptb \gnot{\hspace{-0,34cm}\mu_M} \grmptb \gmu \gnl
\gcl{1} \gcl{1} \gcl{1} \gcn{1}{1}{2}{1} \gnl
\gcl{1} \gcl{1}  \gmu \gnl
\gob{1}{F} \gob{1}{F} \gob{2}{B} \gnl
\gend
\end{equation}

\vspace{-0,2cm}
$$\textnormal{{\small $\Delta_M-\mu_M$ compatibility}}$$

\noindent and the 8 bimonad compatibilities for the comonad- {\em i.e.} second components:
\begin{center} \hspace{-1,4cm}
\begin{tabular} {p{5.5cm}p{-7cm}p{4cm}p{-1cm}p{3.4cm}} 
\begin{equation}\eqlabel{1 lambda_C}
\gbeg{4}{4}
\got{1}{F} \got{1}{F} \got{2}{B} \gnl
\gcl{1} \gcl{1} \gcmu \gnl
\gcl{1} \glmptb \gnot{\hspace{-0,34cm}\phi} \grmptb \gcl{1} \gnl
\glmpt \gnot{\hspace{-0,34cm}\Epsilon_C} \grmpt \glmpt \gnot{\hspace{-0,34cm}\Epsilon_C} \grmpt \gnl
\gob{2}{} \gnl
\gend=
\gbeg{4}{4}
\got{1}{F} \got{1}{F} \got{2}{B} \gnl
\gcl{1} \gcl{1} \gcmu \gnl
\glmpt \gnot{\mu_C} \gcmpt \grmptb \gcl{1} \gnl
\gvac{2} \glmpt \gnot{\hspace{-0,34cm}\Epsilon_C} \grmpt \gnl
\gob{2}{} \gnl
\gend
\end{equation} & \qquad \qquad &
\begin{equation}\eqlabel{2 lambda_C}
\gbeg{2}{5}
\got{2}{B} \gnl
\gcmu \gnl
\gbmp{\eta_C}  \gcl{1} \gnl
\glmptb \gnot{\hspace{-0,34cm}\Delta_C} \grmptb \gnl
\gob{1}{F} \gob{1}{F} \gnl
\gend=
\gbeg{2}{5}
\got{2}{B} \gnl
\gcmu \gnl
\gbmp{\eta_C} \gbmp{\eta_C} \gnl
\gcl{1} \gcl{1} \gnl
\gob{1}{F} \gob{1}{F} \gnl
\gend
\end{equation} & \qquad\quad\quad &
\begin{equation}\eqlabel{3 lambda_C}
\gbeg{2}{5}
\got{2}{B} \gnl
\gcmu \gnl
\gbmp{\eta_C} \gcl{1} \gnl
\glmpt \gnot{\hspace{-0,34cm}\Epsilon_C} \grmpt \gnl
\gob{2}{} \gnl
\gend=
\gbeg{2}{3}
\got{1}{B} \gnl
\gcl{1} \gnl
\gcu{1} \gnl
\gob{1}{} \gnl
\gend
\end{equation}
\\ {\qquad\footnotesize $\Epsilon_C-\mu_C$ compatibility} & &  {\footnotesize $\Delta_C-\eta_C$ compatibility}  & &  { \footnotesize $\Epsilon_C-\eta_C$ compatibility}
\end{tabular}
\end{center}

\begin{center} \hspace{-1cm}
\begin{tabular} {p{8cm}p{0cm}p{6.2cm}} 
\begin{equation}\eqlabel{4 lambda_C}
\gbeg{7}{7}
\got{1}{F} \got{1}{F} \got{1}{F} \gvac{1} \got{1}{B} \gnl
\gcl{3} \gcl{1} \gcl{1} \gwcm{3} \gnl
\gvac{1} \glmpt \gnot{\lambda_C} \gcmptb \grmptb \gwcm{3} \gnl
\gvac{1} \gvac{1} \gcl{1} \glmptb \gnot{\hspace{-0,34cm}\phi} \grmptb \gcn{1}{1}{3}{1} \gnl
\glmpt \gnot{\hspace{0,24cm}\lambda_C} \gcmpb \gcmpt \grmptb \gcl{1} \gcl{1} \gnl
\gvac{1} \gcl{1} \gvac{1} \glmpt \gnot{\mu_C} \gcmptb \grmpt  \gnl
\gvac{1} \gob{1}{F} \gvac{2} \gob{1}{F} \gnl
\gend=
\gbeg{5}{6}
\got{1}{F} \got{1}{F} \got{1}{F} \got{2}{B} \gnl
\gcl{2} \gcl{2} \gcl{1} \gcmu \gnl
\gvac{2} \glmptb \gnot{\hspace{-0,34cm}\phi} \grmptb \gcl{2} \gnl
\glmpt \gnot{\mu_C} \gcmpt \grmptb \gcl{1} \gnl
\gvac{2} \glmpt \gnot{\lambda_C} \gcmptb \grmptb  \gnl
\gvac{3} \gob{1}{F} \gob{1}{F} \gnl
\gend
\end{equation} &  & 
\begin{equation}\eqlabel{5 lambda_C}
\gbeg{4}{6}
\got{1}{F} \got{2}{B} \gnl
\gcl{1} \gcmu \gnl
\glmptb \gnot{\hspace{-0,34cm}\phi} \grmptb \gcl{2} \gnl
\gbmp{\eta_C} \gcl{1} \gnl
\glmpt \gnot{\lambda_C} \gcmptb \grmptb \gnl
\gvac{1} \gob{1}{F} \gob{1}{F} \gnl
\gend=
\gbeg{2}{5}
\got{1}{F} \got{1}{B} \gnl
\gcl{3} \gcl{1} \gnl
\gvac{1} \gbmp{\eta_C} \gnl
\gvac{1} \gcl{1} \gnl
\gob{1}{F} \gob{1}{F} \gnl
\gend
\end{equation}
\\ {\qquad\qquad\qquad \footnotesize $\lambda_C-\mu_C$ compatibility} & &  {\qquad\qquad \footnotesize $\lambda_C-\eta_C$ compatibility}
\end{tabular}
\end{center}

\begin{center} \hspace{-0,4cm}
\begin{tabular} {p{6cm}p{1cm}p{6cm}} 
\begin{equation}\eqlabel{6 lambda_C}
\gbeg{5}{8}
\got{1}{F} \got{1}{F} \got{2}{B} \gnl
\gcl{3} \gcl{1} \gcmu \gnl
\gvac{1} \glmptb \gnot{\hspace{-0,34cm}\Delta_C} \grmptb \gcn{1}{1}{1}{2} \gnl
\gcl{1} \gcl{1} \gcl{1} \gcmu \gnl
\gcl{1} \gcl{1} \glmptb \gnot{\hspace{-0,34cm}\phi} \grmptb \gcl{2} \gnl
\glmpt \gnot{\lambda_C} \gcmptb \grmptb \gcl{1} \gnl
\gvac{1} \gcl{1} \glmpt \gnot{\lambda_C} \gcmptb \grmptb \gnl
\gvac{1} \gob{1}{F} \gvac{1} \gob{1}{F} \gob{1}{F} \gnl
\gend=
\gbeg{5}{6}
\got{1}{F} \got{1}{F} \got{2}{B} \gnl
\gcl{1} \gcl{1} \gcmu \gnl
\glmpt \gnot{\lambda_C} \gcmptb \grmptb \gcl{1} \gnl
\gvac{1} \gcl{1} \glmptb \gnot{\hspace{-0,34cm}\phi} \grmptb \gnl
\gvac{1} \glmptb \gnot{\hspace{-0,34cm}\Delta_C} \grmptb \gcl{1} \gnl
\gvac{1} \gob{1}{F} \gob{1}{F} \gob{1}{F} \gnl
\gend
\end{equation} &  &  \vspace{0,6cm}
\begin{equation}\eqlabel{7 lambda_C}
\gbeg{5}{6}
\got{1}{F} \got{1}{F} \got{2}{B} \gnl
\gcl{1} \gcl{1} \gcmu \gnl
\glmpt \gnot{\lambda_C} \gcmptb \grmptb \gcl{1} \gnl
\gvac{1} \gcl{1} \glmptb \gnot{\hspace{-0,34cm}\phi} \grmptb \gnl
\gvac{1} \glmpt \gnot{\hspace{-0,34cm}\Epsilon_C} \grmpt \gcl{1} \gnl
\gob{7}{F} \gnl
\gend=
\gbeg{2}{5}
\got{1}{F} \got{1}{F} \got{1}{B}  \gnl
\gcl{3} \gcl{1} \gcl{1} \gnl
\gvac{1} \glmpt \gnot{\hspace{-0,34cm}\Epsilon_C} \grmpt \gnl
\gob{1}{F} \gnl
\gend
\end{equation}
\\  {\qquad\qquad \footnotesize $\lambda_C-\Delta_C$ compatibility}  & & {\quad\qquad\qquad \footnotesize $\lambda_C-\Epsilon_C$ compatibility}
\end{tabular}
\end{center}

\begin{equation}\eqlabel{8 lambda_C}
\gbeg{4}{6}
\got{1}{F} \got{1}{F} \got{2}{B} \gnl
\gcl{1} \gcl{1} \gcmu \gnl
\glmpt \gnot{\mu_C} \gcmpt \grmptb \gcl{1} \gnl
\gvac{2} \glmptb \gnot{\hspace{-0,34cm}\Delta_C} \grmptb \gnl
\gvac{2} \gcl{1} \gcl{1} \gnl
\gvac{2} \gob{1}{F} \gob{1}{F} \gnl
\gend=
\gbeg{5}{8}
\got{1}{F} \got{1}{F} \got{2}{B} \gnl
\gcl{3} \gcl{1} \gcmu \gnl
\gvac{1} \glmptb \gnot{\hspace{-0,34cm}\Delta_C} \grmptb \gcn{1}{1}{1}{2} \gnl
\gcl{1} \gcl{1} \gcl{1} \gcmu \gnl
\gcl{1} \gcl{1} \glmptb \gnot{\hspace{-0,34cm}\phi} \grmptb \gcl{2} \gnl
\glmpt \gnot{\lambda_C} \gcmptb \grmptb \gcl{1} \gnl
\gvac{1} \gcl{1} \glmpt \gnot{\mu_C} \gcmptb \grmpt \gnl
\gvac{1} \gob{1}{F} \gvac{1} \gob{1}{F} \gnl
\gend
\end{equation}

\vspace{-0,2cm}
$$\textnormal{{\small $\Delta_C-\mu_C$ compatibility}}$$
Here is where the definition of a biwreath terminates.

Resuming, a biwreath consists of 1-cells $B,F:\A\to\A$ in $\K$, where $(\A, B)$ is a bimonad in $\K$, 2-cells $\psi: BF\to FB, \phi: FB\to BF$,
$$\mu_M:FF\to FB, \quad \eta_M:\Id_{\A}\to FB, \quad \Delta_B: F\to FFB, \quad \Epsilon_M: F\to B, \quad \lambda_M:FF\to FFB$$
$$\mu_C:FFB\to F, \quad \eta_C:B\to F, \quad \Delta_C: FB\to FF, \quad \Epsilon_C: FB\to\Id_{\A}, \quad \lambda_C:FFB\to FF$$
in $\K$ which satisfy the axioms from \equref{psi laws for B} to \equref{psi-lambda-phi} and from \equref{2-cell mu_M} to \equref{8 lambda_C}.

\medskip


\medskip

\begin{rem} \rmlabel{4 combs of wreaths}
Observe that in a biwreath, by the axioms \equref{psi laws for B}, \equref{2-cell mu_M}, \equref{2-cell eta_M}, \equref{monad law mu_M} and \equref{monad law eta_M},
the 1-cell $F$ is a wreath around $B$, and by the axioms \equref{phi laws for B}, \equref{2-cell Delta_C}, \equref{2-cell Epsilon_C}, \equref{comonad law Delta_C} and \equref{comonad law Epsilon_C},
the 1-cell $F$ is a cowreath around $B$. But also, by \equref{psi laws for B}, \equref{2-cell Delta_M}, \equref{2-cell Epsilon_M}, \equref{comonad law Delta_M} and \equref{comonad law Epsilon_M},
$F$ is a mixed wreath around $B$, and dually, by \equref{phi laws for B}, \equref{2-cell mu_C}, \equref{2-cell eta_C}, \equref{monad law mu_C} and \equref{monad law eta_C},
$F$ is a mixed cowreath around $B$. Mixed wreaths were defined in \cite{Street}, they are comonads in $\EM^M(\K)$
(mixed wreaths were treated also in \cite{BC}, but under a missleading name ``cowreath'', the meaning of which we have commented before).
\end{rem}

\bigskip

We recall from \cite[Section 3]{LS} that $\EM$ is a 2-monad on the 2-category $2\x\Cat$ of 2-categories. Its multiplication $\Comp: \EM(\EM(\K))\to\EM(\K)$ 
specified at a 2-category $\K$ takes a monad in $\EM(\K)$ ({\em i.e.} a wreath $F$ around $B$) to its Eilenberg-Moore object. The latter turns out to be 
the composite 1-cell $FB$ which is called a {\em wreath product}. We recall the structure of a wreath product and dually of the cowreath coproduct:
\begin{equation}  \eqlabel{wreath (co)product}
\nabla_{FB}=
\gbeg{3}{5}
\got{1}{F} \got{1}{B} \got{1}{F} \got{3}{B}  \gnl
\gcl{1}  \glmptb \gnot{\hspace{-0,34cm}\psi} \grmptb \gvac{1} \gcl{1} \gnl
\glmptb \gnot{\hspace{-0,34cm}\mu_M} \grmptb \gwmu{3} \gnl
\gcl{1} \gwmu{3} \gnl
\gob{1}{F} \gvac{1} \gob{1}{B}
\gend\hspace{1,5cm}
\eta_{FB}=
\gbeg{2}{3}
\glmpb \gnot{\hspace{-0,34cm}\eta_M} \grmpb \gnl
\gcl{1} \gcl{1} \gnl
\gob{1}{F} \gob{1}{B}
\gend; \hspace{1,5cm}
\Delta_{FB}=
\gbeg{3}{5}
\got{1}{F} \got{3}{B} \gnl
\gcl{1} \gwcm{3} \gnl
\glmptb \gnot{\hspace{-0,34cm}\Delta_C} \grmptb \gwcm{3} \gnl
\gcl{1}  \glmptb \gnot{\hspace{-0,34cm}\phi} \grmptb \gvac{1} \gcl{1} \gnl
\gob{1}{F} \gob{1}{B} \got{1}{F} \got{3}{B}
\gend\hspace{1,5cm}
\Epsilon_{FB}=
\gbeg{2}{4}
\got{1}{F} \got{1}{B} \gnl
\gcl{1} \gcl{1} \gnl
\glmpt \gnot{\hspace{-0,34cm}\Epsilon_C} \grmpt \gnl
\gend
\end{equation}

\subsection{Structures inside of a biwreath}

Biwreaths posses lots of structures, we are going to investigate them. We start with the one which is easiest to see, it follows by \prref{distr-actions}:

\begin{prop} \prlabel{left B-structures of F}
In a biwreath $F,B:\A\to\A$ in $\K$ the 1-cell $F$ is a left $B$-module and a left $B$-comodule.
\end{prop}

Accordingly, we fix the following notation for the rest of the paper:
\begin{center} 
\begin{tabular}{p{5cm}p{2cm}p{5cm}}
\begin{equation} \eqlabel{left B-mod}
\gbeg{2}{3}
\got{1}{B} \got{1}{F} \gnl
\glm \gnl
\gob{3}{F}
\gend=
\gbeg{2}{4}
\got{1}{B} \got{1}{F} \gnl
\glmptb \gnot{\hspace{-0,34cm}\psi} \grmptb \gnl
\gcl{1} \gcu{1} \gnl
\gob{1}{F}
\gend
\end{equation} & &
\begin{equation} \eqlabel{left B-comod}
\gbeg{2}{3}
\got{3}{F} \gnl
\glcm \gnl
\gob{1}{B} \gob{1}{F}
\gend=
\gbeg{2}{4}
\got{1}{F} \gnl
\gcl{1} \gu{1} \gnl
\glmptb \gnot{\hspace{-0,34cm}\phi} \grmptb \gnl
\gob{1}{B} \gob{1}{F}
\gend
\end{equation}
\end{tabular}
\end{center}

\subsubsection{Applying unit and counit of $B$} \sslabel{(co)units B}

The next series of information on biwreaths we obtain by applying at appropriate places $\eta_B$ and $\Epsilon_B$ to the axioms of a biwreath.
We first introduce the following notation:
$$
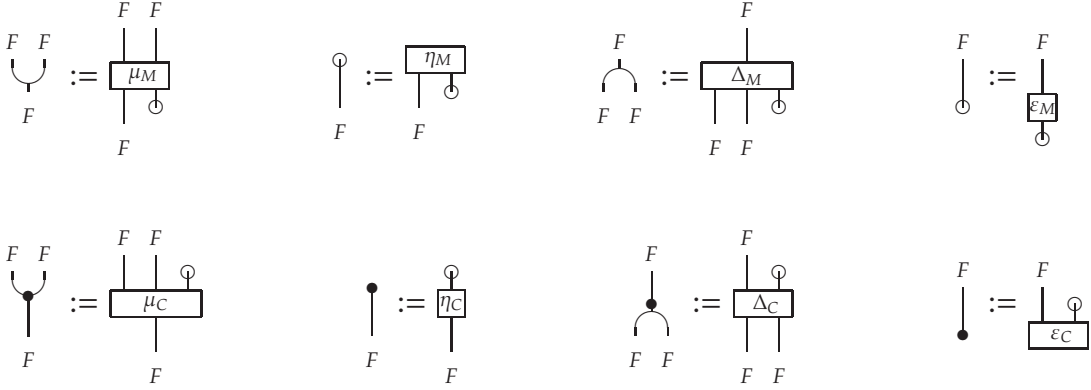
\gbeg{2}{3}
\got{1}{F} \got{1}{F} \gnl
\gmu \gnl
\gob{2}{F} \gnl
\gend:=
\gbeg{2}{5}
\got{1}{F} \got{1}{F} \gnl
\gcl{1} \gcl{1} \gnl
\glmptb \gnot{\hspace{-0,34cm}\mu_M} \grmptb \gnl
\gcl{1} \gcu{1} \gnl
\gob{1}{F} \gnl
\gend \hspace{2cm}
\gbeg{1}{4}
\got{1}{} \gnl
\gu{1} \gnl
\gcl{1} \gnl
\gob{1}{F} \gnl
\gend:=
\gbeg{1}{4}
\got{1}{} \gnl
\glmpb \gnot{\hspace{-0,34cm}\eta_M} \grmpb \gnl
\gcl{1} \gcu{1} \gnl
\gob{1}{F} \gnl
\gend \hspace{2cm}
\gbeg{2}{3}
\got{2}{F} \gnl
\gcmu \gnl
\gob{1}{F} \gob{1}{F} \gnl
\gend:=
\gbeg{3}{5}
\got{3}{F} \gnl
\gvac{1} \gcl{1} \gnl
\glmpb \gnot{\Delta_M} \gcmptb \grmpb \gnl
\gcl{1} \gcl{1} \gcu{1} \gnl
\gob{1}{F} \gob{1}{F} \gnl
\gend \hspace{2cm}
\gbeg{1}{3}
\got{1}{F} \gnl
\gcl{1} \gnl
\gcu{1} \gnl
\gob{1}{} \gnl
\gend:=
\gbeg{1}{3}
\got{1}{F} \gnl
\gcl{1} \gnl
\gbmp{\Epsilon_{M}} \gnl
\gcu{1} \gnl
\gob{1}{} \gnl
\gend
$$

$$
\gbeg{2}{4}
\got{1}{F} \got{1}{F} \gnl
\gmuf{1} \gnl
\gcn{1}{1}{2}{2} \gnl
\gob{2}{F} \gnl
\gend:=
\gbeg{3}{5}
\got{1}{F} \got{1}{F} \gnl
\gcl{1} \gcl{1} \gu{1} \gnl
\glmpt \gnot{\mu_C} \gcmptb \grmpt \gnl
\gvac{1} \gcl{1} \gnl
\gob{3}{F} \gnl
\gend \hspace{2cm}
\gbeg{1}{4}
\got{1}{} \gnl
\guf{1} \gnl
\gcl{1} \gnl
\gob{1}{F} \gnl
\gend:=
\gbeg{1}{5}
\got{1}{} \gnl
\gu{1} \gnl
\gbmp{\eta_C} \gnl
\gcl{1} \gnl
\gob{1}{F} \gnl
\gend \hspace{2cm}
\gbeg{2}{4}
\got{2}{F} \gnl
\gcn{1}{1}{2}{2} \gnl
\gcmuf{1} \gnl
\gob{1}{F} \gob{1}{F} \gnl
\gend:=
\gbeg{2}{5}
\got{1}{F} \gnl
\gcl{1} \gu{1} \gnl
\glmptb \gnot{\hspace{-0,34cm}\Delta_C} \grmptb \gnl
\gcl{1} \gcl{1} \gnl
\gob{1}{F} \gob{1}{F} \gnl
\gend \hspace{2cm}
\gbeg{1}{3}
\got{1}{F} \gnl
\gcl{1} \gnl
\gcuf{1} \gnl
\gob{1}{} \gnl
\gend:=
\gbeg{1}{3}
\got{1}{F} \gnl
\gcl{1} \gu{1} \gnl
\glmpt \gnot{\hspace{-0,34cm}\Epsilon_C} \grmpt \gnl
\gob{1}{} \gnl
\gend
$$
Here we are abusing of notation, as we use the same or similar symbols as for (co)monad structures, that is the (co)unit and (co)multiplication 2-cells, as we mentioned in \equref{strings},
for the new structures on $F$ which we are only going to investigate. We keep this in mind. The upper 2-cells we will call monadic and comonadic pre-(co)multiplications and pre-(co)units.
Let us see which structures we find in a biwreath.

\bigskip


Composing \equref{psi-lambda-phi} with $BF\eta_B$ from above and with $BF\Epsilon_B$ from below, we get:
\begin{equation} \eqlabel{YD condition}
\gbeg{3}{5}
\got{1}{B} \got{1}{F} \gnl
\glmptb \gnot{\hspace{-0,34cm}\psi} \grmptb \gu{1} \gnl
\gcl{1} \glmptb \gnot{\hspace{-0,34cm}\lambda} \grmptb \gnl
\glmptb \gnot{\hspace{-0,34cm}\phi} \grmptb \gcu{1} \gnl
\gob{1}{B} \gob{1}{F}
\gend=
\gbeg{3}{5}
\got{1}{B} \got{3}{F}  \gnl
\gcl{1} \glcm \gnl
\glmptb \gnot{\hspace{-0,34cm}\lambda} \grmptb \gcl{1} \gnl
\gcl{1} \glm \gnl
\gob{1}{B} \gob{3}{F}
\gend
\end{equation}

\vspace{-0,2cm}
$$\textnormal{{\small the Yetter-Drinfel`d condition}}$$
We call this the {\em Yetter-Drinfel`d compatibility condition}. 
The reason for this will be clear later on.

Next, to all the monadic components of the axioms from \equref{2-cell mu_M} to \equref{comonad law Epsilon_C}, {\em i.e.} those axioms which determine the 2-cell condition or (co)monad laws
in $\EM^M(\K)$, we apply $\Epsilon_B$ and then use \equref{epsilon-mu B}. 
Analogously,  to all the comonadic components of the axioms from \equref{2-cell mu_M} to \equref{comonad law Epsilon_C}, {\em i.e.} those axioms which determine the 2-cell condition or (co)monad laws
in $\EM^C(\K)$, we apply $\eta_B$ and then use \equref{eta-Delta B}. We get the following:
\begin{center} \hspace{-1,4cm}
\begin{tabular} {p{6.4cm}p{2cm}p{4cm}} 
\begin{equation} \eqlabel{mod alg}
\gbeg{3}{5}
\got{1}{B} \got{1}{F} \got{1}{F} \gnl
\glmptb \gnot{\hspace{-0,34cm}\psi} \grmptb \gcl{1} \gnl
\gcl{1} \glm \gnl
\gwmu{3} \gnl
\gob{3}{F}
\gend=
\gbeg{3}{5}
\got{1}{B} \got{1}{F} \got{3}{F}\gnl
\gcl{1} \gwmu{3} \gnl
\gcn{1}{1}{1}{3} \gvac{1} \gcl{2} \gnl
\gvac{1} \glm \gnl
\gob{5}{F}
\gend
\end{equation} & &  \vspace{-0,9cm}
\begin{equation} \eqlabel{mod alg unity}
\gbeg{3}{5}
\got{1}{B} \gnl
\gcl{1} \gu{1} \gnl
\glm \gnl
\gvac{1} \gcl{1} \gnl
\gob{3}{F}
\gend=
\gbeg{2}{6}
\got{1}{B} \gnl
\gcl{1} \gnl
\gcu{1} \gnl
\gu{1} \gnl
\gcl{1} \gnl
\gob{1}{F}
\gend
\end{equation}
\end{tabular}
\end{center} \vspace{-0,2cm}
$$ \textnormal{ \footnotesize module monad}  \hspace{5,5cm}  \textnormal{\footnotesize module monad unity} $$ \vspace{-0,7cm}
$$ \textnormal{\footnotesize from 2-cell cond. of $\mu_M$ \equref{2-cell mu_M}} \hspace{4cm} \textnormal{\footnotesize from 2-cell cond. of $\eta_M$ \equref{2-cell eta_M}} $$

\begin{center} \hspace{-1,4cm}
\begin{tabular} {p{6cm}p{1cm}p{6cm}} 
\begin{equation} \eqlabel{mod coalg}
\gbeg{3}{5}
\got{1}{B} \got{2}{F} \gnl
\gcl{1} \gcmu \gnl
\glmptb \gnot{\hspace{-0,34cm}\psi} \grmptb \gcl{1} \gnl
\gcl{1} \glm \gnl
\gob{1}{F} \gob{3}{F}
\gend=
\gbeg{3}{4}
\got{1}{B} \got{1}{F} \gnl
\glm \gnl
\gvac{1} \hspace{-0,22cm} \gcmu \gnl
\gvac{1} \gob{1}{F} \gob{1}{F}
\gend
\end{equation} & & \vspace{0,1cm}
\begin{equation} \eqlabel{mod coalg counity}
\gbeg{2}{3}
\got{1}{B} \got{1}{F} \gnl
\glm \gnl
\gvac{1} \gcu{1} \gnl
\gob{2}{}
\gend=
\gbeg{2}{3}
\got{1}{B} \got{1}{F} \gnl
\gcl{1} \gcl{1} \gnl
\gcu{1} \gcu{1} \gnl
\gob{2}{}
\gend
\end{equation}
\end{tabular}
\end{center}
\vspace{-0,2cm}
$$ \textnormal{ \footnotesize module comonad}  \hspace{5,5cm}  \textnormal{\footnotesize module comonad counity} $$ \vspace{-0,7cm}
$$ \textnormal{\footnotesize from 2-cell cond. of $\Delta_M$ \equref{2-cell Delta_M}} \hspace{4cm} \textnormal{\footnotesize from 2-cell cond. of $\Epsilon_M$ \equref{2-cell Epsilon_M}} $$


\begin{center} \hspace{-1,4cm}
\begin{tabular} {p{6cm}p{1cm}p{6cm}} 
\begin{equation}\eqlabel{comod alg}
\gbeg{3}{5}
\got{1}{F} \got{3}{F} \gnl
\gcl{1} \glcm \gnl
\glmptb \gnot{\hspace{-0,34cm}\phi} \grmptb \gcl{1} \gnl
\gcl{1} \gmuf \gnl \gnl
\gob{1}{B} \gob{2}{F}
\gend=
\gbeg{2}{4}
\got{2}{F} \got{1}{\hspace{-0,34cm}F} \gnl
\gvac{1} \hspace{-0,34cm} \gmuf \gnl \gnl
\gvac{1} \hspace{-0,22cm} \glcm \gnl
\gvac{1} \gob{1}{B} \gob{1}{F}
\gend
\end{equation} & & 
\begin{equation}\eqlabel{comod alg unity}
\gbeg{2}{4}
\got{1}{} \gnl
\gvac{1} \guf{1} \gnl
\glcm \gnl
\gob{1}{B} \gob{1}{F}
\gend=
\gbeg{2}{4}
\got{1}{} \gnl
\gu{1} \guf{1} \gnl
\gcl{1} \gcl{1} \gnl
\gob{1}{B} \gob{1}{F}
\gend
\end{equation}
\end{tabular}
\end{center}
\vspace{-0,2cm}
$$ \textnormal{ \footnotesize comodule monad}  \hspace{5,5cm}  \textnormal{\footnotesize comodule monad unity} $$ \vspace{-0,7cm}
$$ \textnormal{\footnotesize from 2-cell cond. of $\mu_C$ \equref{2-cell mu_C}} \hspace{4cm} \textnormal{\footnotesize from 2-cell cond. of $\eta_C$ \equref{2-cell eta_C}} $$

\begin{center} \hspace{-1,4cm}
\begin{tabular} {p{6cm}p{0cm}p{6cm}} 
\begin{equation}\eqlabel{comod coalg}
\gbeg{3}{6}
\got{3}{F} \gnl
\gvac{1} \gcl{1} \gnl
\gwcmf{3} \gnl
\gcl{1} \glcm \gnl
\glmptb \gnot{\hspace{-0,34cm}\phi} \grmptb \gcl{1} \gnl
\gob{1}{B} \gob{1}{F} \gob{1}{F}
\gend=
\gbeg{3}{6}
\got{5}{F}\gnl
\gvac{2} \gcl{1} \gnl
\gvac{1} \glcm \gnl
\gcn{1}{1}{3}{1} \gvac{1} \gcl{1} \gnl
\gcl{1} \gwcmf{3} \gnl
\gob{1}{B} \gob{1}{F} \gob{3}{F}
\gend
\end{equation} & & \vspace{0,1cm}
\begin{equation}\eqlabel{comod coalg counity}
\gbeg{2}{4}
\got{3}{F} \gnl
\glcm \gnl
\gcl{1} \gcuf{1} \gnl
\gob{1}{B}
\gend=
\gbeg{1}{4}
\got{1}{F} \gnl
\gcuf{1} \gnl
\gu{1} \gnl
\gob{1}{B}
\gend
\end{equation}
\end{tabular}
\end{center}
\vspace{-0,2cm}
$$ \textnormal{ \footnotesize comodule comonad}  \hspace{5,5cm}  \textnormal{\footnotesize comodule comonad counity} $$ \vspace{-0,7cm}
$$ \textnormal{\footnotesize from 2-cell cond. of $\Delta_C$ \equref{2-cell Delta_C}} \hspace{4cm} \textnormal{\footnotesize from 2-cell cond. of $\Epsilon_C$ \equref{2-cell Epsilon_C}} $$



\begin{center} \hspace{-1,4cm}
\begin{tabular} {p{7.2cm}p{1cm}p{6.8cm}}
\begin{equation}\eqlabel{weak assoc. mu_M}
\gbeg{4}{4}
\got{1}{F} \got{1}{F} \got{3}{F} \gnl
\gcl{1} \gwmu{3} \gnl
\gwmu{3} \gnl
\gob{3}{F}
\gend=
\gbeg{3}{5}
\got{1}{F} \got{1}{F} \got{1}{F} \gnl
\glmptb \gnot{\hspace{-0,34cm}\mu_M} \grmptb \gcl{1} \gnl
\gcl{1} \glm \gnl
\gwmu{3} \gnl
\gob{3}{F}
\gend
\end{equation} & &
\begin{equation}\eqlabel{weak unity eta_M}
\gbeg{3}{5}
\got{1}{} \got{3}{F}\gnl
\glmpb \gnot{\hspace{-0,34cm}\eta_M} \grmpb \gcl{1} \gnl
\gcl{1} \glm \gnl
\gwmu{3} \gnl
\gob{3}{F}
\gend=
\gbeg{1}{4}
\got{1}{F}\gnl
\gcl{2} \gnl
\gob{1}{F}
\gend=
\gbeg{2}{4}
\got{1}{F}\gnl
\gcl{1} \gu{1} \gnl
\gmu \gnl
\gob{2}{F}
\gend
\end{equation}
\end{tabular}
\end{center}
\vspace{-0,5cm}
$$ \textnormal{ \footnotesize weak associativity of $\mu_M$}  \hspace{5,5cm}  \textnormal{\footnotesize weak unity $\eta_M$} $$ \vspace{-0,7cm}
$$ \textnormal{\footnotesize from monad law for $\mu_M$ \equref{monad law mu_M}} \hspace{4cm} \textnormal{\footnotesize from monad law for $\eta_M$ \equref{monad law eta_M}} $$

\begin{center} \hspace{-1,4cm}
\begin{tabular}{p{7.2cm}p{1cm}p{6.8cm}}
\begin{equation} \eqlabel{quasi assoc. mu_C}
\gbeg{4}{4}
\got{1}{F} \got{1}{F} \got{3}{F} \gnl
\gcl{1} \gwmuf{3} \gnl
\gwmuf{3} \gnl
\gob{3}{F}
\gend=
\gbeg{4}{5}
\got{1}{F} \got{1}{F} \got{3}{F} \gnl
\gcl{1} \gcl{1} \glcm \gnl
\glmpt \gnot{\mu_C} \gcmptb \grmpt \gcl{1} \gnl
\gvac{1} \gwmuf{3} \gnl
\gob{5}{F}
\gend
\end{equation} &  &
\begin{equation}\eqlabel{quasi unity eta_C}
\gbeg{2}{5}
\got{3}{F} \gnl
\glcm \gnl
\gbmp{\s\eta_{C}} \gcl{1} \gnl
\gmuf{1} \gnl
\gob{2}{F}
\gend=
\gbeg{1}{4}
\got{1}{F} \gnl
\gcl{2} \gnl
\gob{1}{F}
\gend=
\gbeg{2}{4}
\got{1}{F} \gnl
\gcl{1} \guf{1} \gnl
\gmuf{1}  \gnl
\gob{2}{F}
\gend
\end{equation}
\end{tabular}
\end{center}
\vspace{-0,5cm}
$$ \textnormal{ \footnotesize quasi associativity of $\mu_C$}  \hspace{5,5cm}  \textnormal{\footnotesize quasi unity $\eta_C$} $$ \vspace{-0,7cm}
$$ \textnormal{\footnotesize from monad law for $\mu_C$ \equref{monad law mu_C}} \hspace{4cm} \textnormal{\footnotesize from monad law for $\eta_C$ \equref{monad law eta_C}} $$

\begin{center} \hspace{-1,4cm}
\begin{tabular}{p{7.2cm}p{1cm}p{6.8cm}}
\begin{equation} \eqlabel{quasi coass. Delta_M}
\gbeg{4}{4}
\got{3}{F} \gnl
\gwcm{3} \gnl
\gcl{1} \gwcm{3} \gnl
\gob{1}{F} \gob{1}{F} \gob{3}{F} \gnl
\gend=
\gbeg{5}{5}
\got{5}{F} \gnl
\gvac{1} \gwcm{3} \gnl
\glmpb \gnot{\Delta_M} \gcmptb \grmpb \gcl{1} \gnl
\gcl{1} \gcl{1} \glm \gnl
\gob{1}{F} \gob{1}{F} \gob{3}{F} \gnl
\gend
\end{equation} &  &
\begin{equation}\eqlabel{quasi counity Epsilon_M}
\gbeg{2}{5}
\got{2}{F} \gnl
\gcmu \gnl
\gbmp{\Epsilon_M} \gcl{1} \gnl
\glm \gnl
\gob{3}{F} \gnl
\gend=
\gbeg{1}{4}
\got{1}{F} \gnl
\gcl{2} \gnl
\gob{1}{F} \gnl
\gend=
\gbeg{2}{4}
\got{2}{F} \gnl
\gcmu  \gnl
\gcl{1} \gcu{1} \gnl
\gob{1}{F} \gnl
\gend
\end{equation}
\end{tabular}
\end{center}
\vspace{-0,5cm}
$$ \textnormal{ \footnotesize quasi coassociativity of $\Delta_M$}  \hspace{5,5cm}  \textnormal{\footnotesize quasi counity $\Epsilon_M$} $$ \vspace{-0,7cm}
$$ \textnormal{\footnotesize from comonad law for $\Delta_M$ \equref{comonad law Delta_M}} \hspace{4cm} \textnormal{\footnotesize from comonad law for $\Epsilon_M$ \equref{comonad law Epsilon_C}} $$

\begin{center} \hspace{-1,4cm}
\begin{tabular}{p{7.2cm}p{2cm}p{6.8cm}}
\begin{equation} \eqlabel{weak coass. Delta_C}
\gbeg{4}{5}
\got{3}{F} \gnl
\gvac{1} \gcl{1} \gnl
\gwcmf{3} \gnl
\gcl{1} \gwcmf{3} \gnl
\gob{1}{F} \gob{1}{F} \gob{3}{F} \gnl
\gend=
\gbeg{3}{6}
\got{3}{F} \gnl
\gvac{1} \gcl{1} \gnl
\gwcmf{3} \gnl
\gcl{1} \glcm \gnl
\glmptb \gnot{\hspace{-0,34cm}\Delta_C} \grmptb \gcl{1} \gnl
\gob{1}{F}\gob{1}{F}  \gob{1}{F}
\gend
\end{equation} &  &
\begin{equation}\eqlabel{weak counit Epsilon_C}
\gbeg{3}{6}
\got{3}{F} \gnl
\gcn{1}{1}{3}{3} \gnl
\gwcmf{3} \gnl
\gcl{1} \glcm \gnl
\glmpt \gnot{\hspace{-0,34cm}\Epsilon_C} \grmpt \gcl{1} \gnl
\gob{5}{F}
\gend=
\gbeg{1}{4}
\got{1}{F} \gnl
\gcl{2} \gnl
\gob{1}{F}
\gend=
\gbeg{2}{5}
\got{2}{F} \gnl
\gcn{1}{1}{2}{2} \gnl
\gcmuf{1} \gnl
\gcl{1} \gcuf{1} \gnl
\gob{1}{F} \gnl
\gend
\end{equation}
\end{tabular}
\end{center}
\vspace{-0,5cm}
$$ \textnormal{ \footnotesize weak coassociativity of $\Delta_C$}  \hspace{5,5cm}  \textnormal{\footnotesize weak counity $\Epsilon_C$} $$ \vspace{-0,7cm}
$$ \textnormal{\footnotesize from comonad law for $\Delta_C$ \equref{comonad law Delta_C}} \hspace{4cm} \textnormal{\footnotesize from comonad law for $\Epsilon_C$ \equref{comonad law Epsilon_C}} $$


\noindent from the 8 monadic bimonad rules \equref{1 lambda_M} - \equref{8 lambda_M} for $F$, with the same algorythm, we get: 
\begin{center} \hspace{-1,4cm} 
\begin{tabular} {p{4.5cm}p{-7cm}p{5.8cm}p{-1cm}p{3.4cm}} 
\begin{equation}\eqlabel{1 lambda_M implies}
\gbeg{2}{3}
\got{1}{F} \got{1}{F} \gnl
\gcl{1} \gcl{1} \gnl
\gcu{1} \gcu{1} \gnl
\gend=
\gbeg{2}{3}
\got{1}{F} \got{1}{F} \gnl
\gmu \gnl
\gvac{1} \hspace{-0,22cm} \gcu{1} \gnl
\gend
\end{equation} & \qquad \qquad &
\begin{equation}\eqlabel{2 lambda_M implies}
\gbeg{1}{3}
 \gu{1} \gnl
\hspace{-0,34cm} \gcmu \gnl
\gob{1}{F} \gob{1}{F} \gnl
\gend=
\gbeg{3}{3}
\glmpb \gnot{\hspace{-0,34cm}\eta_M} \grmpb \gu{1} \gnl
\gcl{1} \glm \gnl
\gob{1}{F} \gob{3}{F} \gnl
\gend\stackrel{\equref{mod alg unity} }{=}
\gbeg{2}{3}
\gu{1} \gu{1} \gnl
\gcl{1} \gcl{1} \gnl
\gob{1}{F} \gob{1}{F}
\gend
\end{equation} & \qquad\quad\quad &
\begin{equation}\eqlabel{3 lambda_M implies}
\gbeg{1}{2}
\gu{1} \gnl
\gcu{1} \gnl
\gend=
\Id_{id_{\A}}
\end{equation} 
\end{tabular}
\end{center}

\begin{center} 
\begin{tabular} {p{6cm}p{1,5cm}p{5cm}} 
\begin{equation}\eqlabel{4 lambda_M implies}
\gbeg{5}{6}
\gvac{1} \got{1}{F} \got{1}{F} \got{1}{F} \gnl
\gvac{1} \gcl{1} \glmptb \gnot{\lambda_M} \gcmptb \grmpb \gnl
\glmpb \gnot{\lambda_M} \gcmptb \grmptb \gcl{1} \gcu{1} \gnl
\gcl{1} \gcl{1} \glm \gnl
\gcl{1} \gwmu{3} \gnl
\gob{1}{F} \gob{3}{F}  \gnl
\gend=
\gbeg{5}{6}
\got{1}{F} \got{1}{F} \got{1}{F} \gnl
\glmptb \gnot{\hspace{-0,34cm}\mu_M} \grmptb \gcl{1} \gnl
\gcl{1} \glm \gnl
\glmptb \gnot{\lambda_M} \gcmpb \grmptb \gnl
\gcl{1} \gcl{1} \gcu{1} \gnl
\gob{1}{F} \gob{1}{F} \gnl
\gend
\end{equation} &  &  
\begin{equation}\eqlabel{5 lambda_M implies}
\gbeg{4}{6}
\got{5}{F} \gnl
\glmpb \gnot{\hspace{-0,34cm}\eta_M} \grmpb \gcl{1} \gnl
\gcl{1} \glm \gnl
\glmptb \gnot{\lambda_M} \gcmpb \grmptb \gnl
\gcl{1} \gcl{1} \gcu{1} \gnl
\gob{1}{F} \gob{1}{F} \gnl
\gend=
\gbeg{3}{5}
\got{1}{F} \gnl
\gcl{3} \gu{1} \gnl
\gvac{1} \gcl{2} \gnl
\gob{1}{F} \gob{1}{F} \gnl
\gend
\end{equation} 
\end{tabular}
\end{center}

\begin{center} \hspace{-1,5cm} 
\begin{tabular} {p{8cm}p{1cm}p{4.7cm}} 
\begin{equation}\eqlabel{6 lambda_M implies}
\gbeg{4}{7}
\gvac{1} \got{1}{F} \got{2}{F} \gnl
\gvac{1} \gcl{1} \gcmu \gnl
\glmpb \gnot{\lambda_M} \gcmptb \grmptb \gcl{1} \gnl
\gcl{1} \gcl{1} \glm \gnl
\gcl{1} \glmptb \gnot{\lambda_M} \gcmpb \grmptb \gnl
\gcl{1} \gcl{1} \gcl{1} \gcu{1}  \gnl
\gob{1}{F} \gob{1}{F} \gob{1}{F} \gnl
\gend=
\gbeg{5}{6}
\gvac{2} \got{1}{F} \got{1}{F} \gnl
\gvac{2} \gcl{1} \gcl{1} \gnl
\gvac{2} \glmptb \gnot{\lambda_M} \gcmptb \grmpb \gnl
\glmpb \gnot{\Delta_M} \gcmpb \grmptb \gcl{1} \gcu{1} \gnl
\gcl{1} \gcl{1} \glm \gnl
\gob{1}{F} \gob{1}{F} \gob{3}{F} \gnl
\gend
\end{equation} &  &
\begin{equation}\eqlabel{7 lambda_M implies}
\gbeg{3}{5}
\got{1}{F} \got{1}{F} \gnl
\glmptb \gnot{\lambda_M} \gcmptb \grmpb \gnl
\gbmp{\Epsilon_M} \gcl{1} \gcu{1} \gnl
\glm \gnl
\gob{3}{F} \gnl
\gend=
\gbeg{2}{5}
\got{1}{F} \got{1}{F} \gnl
\gcl{1} \gcl{1} \gnl
\gcl{1} \gcl{1} \gnl
\gcl{1} \gcu{1} \gnl
\gob{1}{F} \gnl
\gend
\end{equation} 
\end{tabular}
\end{center}

\begin{equation}\eqlabel{8 lambda_M implies}
\gbeg{3}{5}
\got{1}{F} \got{3}{F} \gnl
\gwmu{3} \gnl
\gvac{1} \gcl{1} \gnl
\gwcm{3} \gnl
\gob{1}{F} \gob{3}{F}
\gend=
\gbeg{5}{6}
\gvac{1} \got{1}{F} \got{2}{F} \gnl
\gvac{1} \gcl{1} \gcmu \gnl
\glmpb \gnot{\lambda_M} \gcmptb \grmptb \gcl{1} \gnl
\gcl{1} \gcl{1} \glm \gnl
\gcl{1} \gwmu{3} \gnl
\gob{1}{F} \gob{3}{F} \gnl
\gend 
\end{equation} 
and from the 8 comonadic bimonad rules \equref{1 lambda_C} - \equref{8 lambda_C} for $F$, with the same algorythm, we get: 
\begin{center} \hspace{-1,4cm} 
\begin{tabular} {p{6cm}p{-7cm}p{4cm}p{-1cm}p{3.4cm}} 
\begin{equation}\eqlabel{1 lambda_C implies}
\gbeg{3}{4}
\got{1}{F} \got{3}{F} \gnl
\gcl{1} \glcm \gnl
\glmpt \gnot{\hspace{-0,34cm}\Epsilon_C} \grmpt \gcuf{1} \gnl
\gend\stackrel{\equref{comod coalg counity}}{=}
\gbeg{2}{3}
\got{1}{F} \got{1}{F} \gnl
\gcl{1} \gcl{1} \gnl
\gcuf{1} \gcuf{1} \gnl
\gend=
\gbeg{2}{3}
\got{1}{F} \got{1}{F} \gnl
\gmuf{1} \gnl
\gvac{1} \hspace{-0,34cm} \gcuf{1} \gnl
\gend
\end{equation} & \qquad \qquad &
\begin{equation}\eqlabel{2 lambda_C implies}
\gbeg{2}{4}
\guf{1} \gnl
\gcl{1} \gnl
\hspace{-0,34cm} \gcmuf{1} \gnl
\gob{1}{F} \gob{1}{F} \gnl
\gend=
\gbeg{2}{4}
\guf{1} \guf{1} \gnl
\gcl{2} \gcl{2} \gnl
\gob{1}{F} \gob{1}{F} \gnl
\gend
\end{equation} & \qquad\quad\quad &  \vspace{0,1cm}
\begin{equation}\eqlabel{3 lambda_C implies}
\gbeg{1}{2}
\guf{1} \gnl
\gcuf{1} \gnl
\gend=\Id_{id_{\A}}
\end{equation} 
\end{tabular}
\end{center}

\begin{center} \hspace{-1,4cm} 
\begin{tabular} {p{7cm}p{1cm}p{6.2cm}} 
\begin{equation}\eqlabel{4 lambda_C implies}
\gbeg{6}{7}
\got{1}{F} \gvac{1} \got{1}{F} \got{1}{F} \gnl
\gcl{3} \gvac{1} \gcl{1} \gcl{1} \gu{1} \gnl
\gvac{2} \glmptb \gnot{\lambda_C} \gcmpt \grmptb \gnl
\gvac{1} \gvac{1} \gcl{1} \glcm \gnl
\glmpt \gnot{\hspace{0,24cm}\lambda_C} \gcmpb \gcmpt \grmptb \gcl{1} \gu{1} \gnl
\gvac{1} \gcl{1} \gvac{1} \glmpt \gnot{\mu_C} \gcmptb \grmpt  \gnl
\gvac{1} \gob{1}{F} \gvac{2} \gob{1}{F} \gnl
\gend=
\gbeg{6}{6}
\got{1}{F} \got{1}{F} \got{3}{F} \gnl
\gcl{1} \gcl{1} \glcm \gnl
\glmpt \gnot{\mu_C} \gcmpt \grmptb \gcl{1} \gu{1} \gnl
\gvac{2} \glmpt \gnot{\lambda_C} \gcmptb \grmptb  \gnl
\gvac{3} \gcl{1} \gcl{1} \gnl
\gvac{3} \gob{1}{F} \gob{1}{F} \gnl
\gend
\end{equation} &  &  \vspace{0,4cm}
\begin{equation}\eqlabel{5 lambda_C implies}
\gbeg{4}{5}
\got{3}{F} \gnl
\glcm \gnl
\gbmp{\eta_C} \gcl{1} \gu{1} \gnl
\glmpt \gnot{\lambda_C} \gcmptb \grmptb \gnl
\gvac{1} \gob{1}{F} \gob{1}{F} \gnl
\gend=
\gbeg{2}{4}
\got{1}{F} \gnl
\gcl{2} \guf{1} \gnl
\gvac{1} \gcl{1} \gnl
\gob{1}{F} \gob{1}{F} \gnl
\gend
\end{equation} 
\end{tabular}
\end{center}

\begin{center} 
\begin{tabular} {p{6cm}p{1cm}p{6cm}} 
\begin{equation}\eqlabel{6 lambda_C implies}
\gbeg{5}{7}
\got{1}{F} \got{3}{F} \gnl
\gcl{1} \gvac{1} \gcl{1} \gnl
\gcl{2} \gwcmf{3} \gnl
\gvac{1} \gcl{1} \glcm \gnl
\glmpt \gnot{\lambda_C} \gcmptb \grmptb \gcl{1} \gu{1} \gnl
\gvac{1} \gcl{1} \glmp \gnot{\lambda_C} \gcmptb \grmptb \gnl
\gvac{1} \gob{1}{F} \gvac{1} \gob{1}{F} \gob{1}{F} \gnl
\gend=
\gbeg{5}{6}
\got{1}{F} \got{1}{F} \gnl
\gcl{1} \gcl{1} \gu{1} \gnl
\glmptb \gnot{\lambda_C} \gcmpt \grmptb \gnl
\gcl{1} \glcm \gnl
\glmptb \gnot{\hspace{-0,34cm}\Delta_C} \grmptb \gcl{1} \gnl
\gob{1}{F} \gob{1}{F} \gob{1}{F} \gnl
\gend
\end{equation} & &  \vspace{0,2cm}
\begin{equation}\eqlabel{7 lambda_C implies}
\gbeg{3}{6}
\got{1}{F} \got{1}{F} \gnl
\gcl{1} \gcl{1} \gu{1} \gnl
\glmptb \gnot{\lambda_C} \gcmpt \grmptb \gnl
\gcl{1} \glcm \gnl
\glmpt \gnot{\hspace{-0,34cm}\Epsilon_C} \grmpt \gcl{1} \gnl
\gob{5}{F} \gnl
\gend=
\gbeg{2}{5}
\got{1}{F} \got{1}{F} \gnl
\gcl{3} \gcl{1} \gnl
\gvac{1} \gcuf{1} \gnl
\gob{1}{F} \gnl
\gend
\end{equation} 
\end{tabular}
\end{center}

\begin{equation}\eqlabel{8 lambda_C implies}
\gbeg{3}{5}
\got{1}{F} \got{3}{F} \gnl
\gwmuf{3} \gnl
\gvac{1} \gcl{1} \gnl
\gwcmf{3} \gnl
\gob{1}{F} \gob{3}{F}
\gend=
\gbeg{5}{8}
\got{1}{F} \got{3}{F} \gnl
\gcl{1} \gvac{1} \gcl{1} \gnl
\gcl{2} \gwcmf{3} \gnl
\gvac{1} \gcl{1} \glcm \gnl
\glmpt \gnot{\lambda_C} \gcmptb \grmptb \gcl{1} \gnl
\gvac{1} \gcl{2} \gmuf{1} \gnl
\gvac{2} \gcn{1}{1}{2}{2} \gnl
\gvac{1} \gob{1}{F} \gob{2}{F} \gnl
\gend
\end{equation}


\subsubsection{Applying pre-units and pre-counits of $F$}

Finally, we will apply the (co)monadic (co)unity of $F$, more precisely: 
$\gbeg{1}{2}
\got{1}{F} \gnl
\gcu{1} \gnl
\gend \hspace{0,1cm},
\gbeg{1}{2}
\got{1}{F} \gnl
\gcuf{1} \gnl
\gend \hspace{0,1cm},
\gbeg{1}{2}
\gu{1} \gnl
\gob{1}{F} \gnl
\gend$ 
and 
$\gbeg{1}{2}
\guf{1} \gnl
\gob{1}{F} \gnl
\gend
$ \hspace{0,2cm} to the (majority of) the axioms of a biwreath at appropriate places to obtain further information. Before we do so, we introduce the following notation: 
\begin{center} \hspace{-1,4cm} 
\begin{tabular}{p{3.6cm}p{0cm}p{3.6cm}p{0cm}p{3.6cm}p{0cm}p{3.6cm}}  
\begin{equation} \eqlabel{right F-mod}
\gbeg{2}{3}
\got{1}{B} \got{1}{F} \gnl
\grmo \gcl{1} \gnl 
\gob{1}{B}
\gend:=
\gbeg{2}{5}
\got{1}{B} \got{1}{F} \gnl
\gcl{1} \gcl{1} \gnl
\glmptb \gnot{\hspace{-0,34cm}\psi} \grmptb \gnl
\gcu{1} \gcl{1} \gnl
\gob{3}{B} 
\gend
\end{equation} & &
\hspace{-0,4cm} 
\begin{equation} \eqlabel{right F-comod}
\gbeg{2}{3}
\got{1}{B} \gnl
\gcl{1} \hspace{-0,42cm} \glmf \gnl  
\gvac{1} \gob{1}{B} \gob{1}{F}
\gend:=
\gbeg{2}{5}
\got{3}{B} \gnl
\guf{1} \gcl{1} \gnl
\glmptb \gnot{\hspace{-0,34cm}\phi} \grmptb \gnl
\gcl{1} \gcl{1} \gnl
\gob{1}{B} \gob{1}{F} 
\gend
\end{equation} & &
\hspace{-0,4cm} 
\begin{equation} \eqlabel{left F-comod}
\gbeg{2}{3}
\got{3}{B} \gnl
\grmo \gvac{1} \gcl{1} \gnl
\gob{1}{F} \gob{1}{B}
\gend:=
\gbeg{2}{5}
\got{1}{B} \gnl
\gcl{1} \gu{1} \gnl
\glmptb \gnot{\hspace{-0,34cm}\psi} \grmptb \gnl
\gcl{1} \gcl{1} \gnl
\gob{1}{F} \gob{1}{B} 
\gend
\end{equation} & &
\hspace{-0,4cm} 
\begin{equation} \eqlabel{left F-mod}
\gbeg{2}{3}
\got{1}{F} \got{1}{B} \gnl
\glmf \gcn{1}{1}{-1}{-1} \gnl
\gob{3}{B}
\gend:=
\gbeg{2}{5}
\got{1}{F} \got{1}{B} \gnl
\gcl{1} \gcl{1} \gnl
\glmptb \gnot{\hspace{-0,34cm}\phi} \grmptb \gnl
\gcl{1} \gcuf{1} \gnl
\gob{1}{B}  
\gend
\end{equation} 
\end{tabular}
\end{center}

\begin{center} \hspace{-1,3cm} 
\begin{tabular}{p{6cm}p{0cm}p{6cm}}
\begin{equation} \eqlabel{sigma}
\gbeg{2}{5}
\got{1}{F} \got{1}{F} \gnl
\gcl{1} \gcl{1} \gnl
\glmpt \gnot{\hspace{-0,34cm}\sigma} \grmptb \gnl
\gvac{1} \gcl{1} \gnl
\gob{3}{B}
\gend:=
\gbeg{2}{5}
\got{1}{F} \got{1}{F} \gnl
\gcl{1} \gcl{1} \gnl
\glmptb \gnot{\hspace{-0,34cm}\mu_M} \grmptb \gnl
\gcu{1} \gcl{1} \gnl
\gob{3}{B}
\gend
\end{equation} & &
\begin{equation} \eqlabel{ro}
\gbeg{2}{5}
\got{3}{B} \gnl
\gvac{1} \gcl{1} \gnl
\glmpb \gnot{\hspace{-0,34cm}\rho} \grmptb \gnl
\gcl{1} \gcl{1} \gnl
\gob{1}{F}\gob{1}{F}
\gend:=
\gbeg{2}{5}
\got{3}{B} \gnl
\guf{1} \gcl{1} \gnl
\glmptb \gnot{\hspace{-0,34cm}\Delta_C} \grmptb \gnl
\gcl{1} \gcl{1} \gnl
\gob{1}{F}\gob{1}{F}
\gend
\end{equation} 
\end{tabular}
\end{center}

\begin{center} \hspace{-1,3cm} 
\begin{tabular}{p{6cm}p{0cm}p{6cm}}
\begin{equation} \eqlabel{fi-lambda}
\gbeg{3}{3}
\gcn{1}{1}{2}{1} \gelt{\s\Phi_{\lambda}} \gcn{1}{1}{0}{1} \gnl  %
\gcl{1} \gcl{1} \gcl{1} \gnl
\gob{1}{F} \gob{1}{F} \gob{1}{B} \gnl
\gend:=
\gbeg{3}{4}
\gvac{1} \gu{1} \gnl
\glmpb \gnot{\Delta_M} \gcmptb \grmpb \gnl
\gcl{1} \gcl{1} \gcl{1} \gnl
\gob{1}{F} \gob{1}{F} \gob{1}{B} \gnl
\gend 
\end{equation} & &
\hspace{-0,4cm} 
\begin{equation} \eqlabel{omega}
\gbeg{3}{3}
\got{1}{F} \got{1}{F} \got{1}{B} \gnl
\gcl{1} \gcl{1} \gcl{1} \gnl
\gcn{1}{1}{1}{2} \gelt{\omega} \gcn{1}{1}{1}{0} \gnl   
\gend:=
\gbeg{2}{5}
\got{1}{F} \got{1}{B} \got{1}{B} \gnl
\gcl{1} \gcl{1} \gcl{1} \gnl
\glmpt \gnot{\mu_C} \gcmptb \grmpt \gnl
\gvac{1} \gcuf{1} \gnl
\gend
\end{equation} 
\end{tabular}
\end{center}

Now we do the following: to the 2 axioms for $\psi$ we apply 
$\gbeg{1}{2}
\got{1}{F} \gnl
\gcu{1} \gnl
\gend
$ \hspace{0,1cm} and 
$\gbeg{1}{2}
\gu{1} \gnl
\gob{1}{F} \gnl
\gend
$ \hspace{0,1cm} and to the 2 axioms for $\phi$ we apply 
$\gbeg{1}{2}
\got{1}{F} \gnl
\gcuf{1} \gnl
\gend
$ \hspace{0,1cm} and 
$\gbeg{1}{2}
\guf{1} \gnl
\gob{1}{F} \gnl
\gend
$. \hspace{0,1cm} 
Moreover, to the 2-cell conditions and the monad laws for $\mu_M, \eta_M$ we apply 
$\gbeg{1}{2}
\got{1}{F} \gnl
\gcu{1} \gnl
\gend
$ \hspace{0,1cm} 
and to those of $\mu_C, \eta_C$ we apply 
$\gbeg{1}{2}
\got{1}{F} \gnl
\gcuf{1} \gnl
\gend
$. \hspace{0,1cm} 
Dually, to the 2-cell conditions and the comonad laws of $\Delta_M, \Epsilon_M$ we apply 
$\gbeg{1}{2}
\gu{1} \gnl
\gob{1}{F} \gnl
\gend
$ \hspace{0,1cm} 
and to those of $\Delta_C, \Epsilon_C$ we apply 
$\gbeg{1}{2}
\guf{1} \gnl
\gob{1}{F} \gnl
\gend
$. \hspace{0,1cm} This is what we get: 
 \begin{center} \hspace{-1,5cm} 
\begin{tabular}{p{7.2cm}p{0cm}p{6cm}}
\begin{equation}\eqlabel{F mod alg}
\gbeg{3}{5}
\got{1}{B}\got{1}{B}\got{1}{F}\gnl
\gcl{1} \glmpt \gnot{\hspace{-0,34cm}\psi} \grmptb \gnl
\grmo \gcl{1} \gvac{1} \gcl{1} \gnl
\gwmu{3} \gnl
\gob{3}{B}
\gend=
\gbeg{3}{5}
\got{1}{B}\got{1}{B}\got{1}{F}\gnl
\gmu \gcn{1}{1}{1}{0} \gnl
\gvac{1} \hspace{-0,34cm} \grmo \gcl{1} \gnl
\gvac{1} \gcl{1} \gnl
\gvac{1} \gob{1}{B}
\gend
\end{equation} & &
\begin{equation}\eqlabel{F mod alg unit}
\gbeg{2}{5}
\got{3}{F} \gnl
\gu{1} \gcl{1} \gnl
\grmo \gcl{1} \gnl
\gcl{1} \gnl
\gob{1}{B}
\gend=
\gbeg{2}{4}
\got{1}{F} \gnl
\gcu{1} \gnl
\gu{1} \gnl
\gob{1}{B}
\gend
\end{equation}

\\  {\hspace{2cm} \footnotesize module monad} & &  { \hspace{1cm} \footnotesize module monad unity} \\
\multicolumn{2}{c}{{ \hspace{5cm} \footnotesize \quad\qquad from $\psi$ axioms \equref{psi laws for B}}} 
\end{tabular}
\end{center} 

 \begin{center} \hspace{-1,5cm} 
\begin{tabular}{p{7.2cm}p{0cm}p{6cm}}
\begin{equation}\eqlabel{F comod alg}
\gbeg{3}{5}
\got{1}{B}\got{3}{B}\gnl
\gcl{1} \grmo \gvac{1} \gcl{1} \gnl
\glmptb \gnot{\hspace{-0,34cm}\psi} \grmptb \gcl{1} \gnl
\gcl{1} \gmu \gnl
\gob{1}{F} \gob{2}{B}
\gend=
\gbeg{3}{5}
\got{1}{\hspace{0,22cm}B}\got{1}{\hspace{0,22cm}B}\gnl
\gvac{1} \hspace{-0,34cm} \gmu \gnl
\gvac{1} \hspace{-0,22cm} \grmo \gvac{1} \gcl{1} \gnl
\gvac{1} \gcl{1} \gcl{1} \gnl
\gvac{1} \gob{1}{F} \gob{1}{B}
\gend
\end{equation} & &  
\begin{equation}\eqlabel{F comod alg unit}
\gbeg{2}{4}
\got{3}{} \gnl
\gvac{1} \gu{1} \gnl
\grmo \gvac{1} \gcl{1} \gnl
\gob{1}{F} \gob{1}{B}
\gend=
\gbeg{2}{4}
\got{1}{} \gnl
\gu{1} \gu{1} \gnl
\gcl{1} \gcl{1} \gnl
\gob{1}{F} \gob{1}{B}
\gend
\end{equation}

\\  {\hspace{2cm} \footnotesize comodule monad} & &  { \hspace{1cm} \footnotesize comodule monad unity} \\
\multicolumn{2}{c}{{ \hspace{5cm} \footnotesize \quad\qquad from $\psi$ axioms \equref{psi laws for B}}} 
\end{tabular}
\end{center} 

 \begin{center} \hspace{-1,5cm} 
\begin{tabular}{p{7.2cm}p{0cm}p{6cm}}
\begin{equation}\eqlabel{F mod coalg}
\gbeg{3}{5}
\got{1}{F} \got{2}{B}\gnl
\gcl{1} \gcmu \gnl
\glmptb \gnot{\hspace{-0,34cm}\phi} \grmptb \gcl{1} \gnl
\gcl{1} \glmf \gcn{1}{1}{-1}{-1} \gnl
\gob{1}{B} \gob{3}{B}  
\gend=
\gbeg{3}{5}
\got{1}{F} \got{1}{B}\gnl
\gcl{1} \gcl{1} \gnl
\glmf \gcn{1}{1}{-1}{-1} \gnl
\gvac{1} \hspace{-0,34cm} \gcmu \gnl
\gvac{1} \gob{1}{B} \gob{1}{B} 
\gend
\end{equation} & &  \vspace{0,1cm}
\begin{equation}\eqlabel{F mod coalg counit}
\gbeg{2}{4}
\got{1}{F} \got{1}{B} \gnl
\glmf \gcn{1}{1}{-1}{-1} \gnl
\gvac{1} \gcu{1} \gnl
\gob{3}{} 
\gend=
\gbeg{2}{4}
\got{1}{F} \got{1}{B} \gnl
\gcl{1} \gcl{1} \gnl
\gcuf{1}  \gcu{1} \gnl
\gob{1}{} 
\gend
\end{equation}

\\  {\hspace{2cm} \footnotesize module comonad} & &  { \hspace{1cm} \footnotesize module comonad counity} \\
\multicolumn{2}{c}{{ \hspace{5cm} \footnotesize \quad\qquad from $\phi$ axioms \equref{phi laws for B}}} 
\end{tabular}
\end{center} 

\begin{center} \hspace{-1,5cm} 
\begin{tabular}{p{7.2cm}p{0cm}p{6cm}}
\begin{equation}\eqlabel{F comod coalg}
\gbeg{3}{5}
\got{3}{B}\gnl
\gwcm{3} \gnl
\gcl{1} \hspace{-0,42cm} \glmf \gcl{1} \gnl
\gvac{1} \gcl{1} \glmptb \gnot{\hspace{-0,34cm}\phi} \grmptb \gnl
\gvac{1} \gob{1}{B} \gob{1}{B} \gob{1}{F} 
\gend=
\gbeg{3}{4}
\got{1}{B}\gnl
\gcl{1} \hspace{-0,42cm} \glmf \gnl
\gvac{1} \hspace{-0,2cm} \gcmu \gcn{1}{1}{0}{1} \gnl
\gvac{1} \gob{1}{B} \gob{1}{B} \gob{1}{F} 
\gend
\end{equation} & & \vspace{0,2cm}
\begin{equation}\eqlabel{F comod coalg counit}
\gbeg{2}{4}
\got{1}{B} \gnl
\gcl{1} \hspace{-0,42cm} \glmf \gnl
\gvac{1} \gcu{1} \gcl{1} \gnl
\gob{5}{F} 
\gend=
\gbeg{2}{4}
\got{1}{B} \gnl
\gcu{1} \gnl
\guf{1} \gnl
\gob{1}{F} 
\gend
\end{equation}

\\  {\hspace{2cm} \footnotesize comodule comonad} & &  { \hspace{1cm} \footnotesize comodule comonad counity} \\
\multicolumn{2}{c}{{ \hspace{5cm} \footnotesize \quad\qquad from $\phi$ axioms \equref{psi laws for B}}} 
\end{tabular}
\end{center} 


\begin{center} 
\begin{tabular} {p{6cm}p{1cm}p{6cm}} 
\begin{equation} \eqlabel{weak action} 
\gbeg{3}{6}
\got{1}{B} \got{1}{F} \got{1}{F} \gnl
\glmptb \gnot{\hspace{-0,34cm}\psi} \grmptb \gcl{1} \gnl
\gcl{1} \glmptb \gnot{\hspace{-0,34cm}\psi} \grmptb \gnl
\glmptb \gnot{\hspace{-0,34cm}\sigma} \grmpt \gcl{1} \gnl
\gwmu{3} \gnl
\gob{3}{B}
\gend=
\gbeg{3}{5}
\got{1}{B} \got{1}{F} \got{1}{F}\gnl
\gcl{1} \glmptb \gnot{\hspace{-0,34cm}\mu_M} \grmptb \gnl
\grmo \gcl{1} \gvac{1} \gcl{1} \gnl
\gwmu{3} \gnl
\gob{3}{B}
\gend
\end{equation} & & 
\begin{equation} \eqlabel{weak action unity} 
\gbeg{3}{4}
\got{1}{} \got{3}{B}\gnl
\glmpb \gnot{\hspace{-0,34cm}\eta_M} \grmpb \gcl{1} \gnl
\gcu{1} \gmu \gnl
\gob{4}{B}
\gend=
\gbeg{3}{5}
\got{1}{B}\gnl
\gcl{1} \glmpb \gnot{\hspace{-0,34cm}\eta_M} \grmpb \gnl
\grmo \gcl{1} \gvac{1} \gcl{1} \gnl
\gwmu{3} \gnl
\gob{3}{B}
\gend
\end{equation} 
\end{tabular}
\end{center}
\vspace{-0,6cm}
$$ \textnormal{ \footnotesize twisted action}  \hspace{5,5cm}  \textnormal{\footnotesize twisted action unity} $$ \vspace{-0,7cm}
$$ \textnormal{\footnotesize from 2-cell cond. of $\mu_M$ \equref{2-cell mu_M}} \hspace{4cm} \textnormal{\footnotesize from 2-cell cond. of $\eta_M$ \equref{2-cell eta_M}} $$

\begin{center} \hspace{-1,4cm} 
\begin{tabular} {p{7cm}p{0.6cm}p{6cm}} 
\begin{equation} \eqlabel{quasi coaction} 
\gbeg{4}{5}
\gvac{3} \got{1}{B} \gnl
\gvac{2} \grmo \gvac{1} \gcl{1} \gnl
\glmpb \gnot{\Delta_M} \gcmpb \grmptb \gcl{1} \gnl
\gcl{1} \gcl{1} \gmu \gnl
\gob{1}{F} \gob{1}{F} \gob{2}{B}
\gend=
\gbeg{3}{6}
\got{1}{B} \gnl
\gcl{1} \gcn{1}{1}{2}{1} \gelt{\s\Phi_{\lambda}} \gcn{1}{1}{0}{1} \gnl
\glmptb \gnot{\hspace{-0,34cm}\psi} \grmptb \gcl{1} \gcl{2} \gnl
\gcl{1} \glmptb \gnot{\hspace{-0,34cm}\psi} \grmptb \gnl
\gcl{1} \gcl{1} \gmu \gnl
\gob{1}{F} \gob{1}{F} \gob{2}{B}
\gend
\end{equation} & & 
\begin{equation} \eqlabel{quasi coaction counity} 
\gbeg{2}{5}
\got{3}{B} \gnl
\grmo \gvac{1} \gcl{1} \gnl
\gbmp{\Epsilon_M} \gcl{1} \gnl
\gmu \gnl
\gob{2}{B} 
\gend=
\gbeg{2}{5}
\got{1}{B} \gnl
\gcl{1} \gu{1} \gnl
\gcl{1} \gbmp{\Epsilon_M} \gnl
\gmu \gnl
\gob{2}{B} 
\gend
\end{equation} 
\end{tabular}
\end{center} \vspace{-0,5cm}
$$ \textnormal{ \footnotesize quasi coaction}  \hspace{5,5cm}  \textnormal{\footnotesize quasi coaction counity} $$ \vspace{-0,7cm}
$$ \textnormal{\footnotesize from 2-cell cond. of $\Delta_M$ \equref{2-cell Delta_M}} \hspace{4cm} \textnormal{\footnotesize from 2-cell cond. of $\Epsilon_M$ \equref{2-cell Epsilon_M}} $$


\begin{center} \hspace{-1cm} 
\begin{tabular} {p{7cm}p{0cm}p{6cm}} 
\begin{equation}\eqlabel{dual quasi action}
\gbeg{4}{6}
\got{1}{F} \got{1}{F} \got{2}{B} \gnl
\gcl{1} \gcl{1} \gcmu \gnl
\gcl{1} \glmptb \gnot{\hspace{-0,34cm}\phi} \grmptb \gcl{2} \gnl
\glmptb \gnot{\hspace{-0,34cm}\phi} \grmptb \gcl{1} \gnl
\gcl{1} \gcn{1}{1}{1}{2} \gelt{\omega} \gcn{1}{1}{1}{0} \gnl
\gob{1}{B} 
\gend=
\gbeg{4}{6}
\got{1}{F} \got{1}{F} \got{2}{B} \gnl
\gcl{1} \gcl{1} \gcmu \gnl
\glmpt \gnot{\mu_C} \gcmpt \grmptb \gcl{1} \gnl
\gvac{2} \glmf \gcn{1}{1}{-1}{-1} \gnl
\gvac{3} \gcl{1} \gnl
\gvac{3} \gob{1}{B} 
\gend
\end{equation} & & 
\begin{equation}\eqlabel{dual quasi action unity}
\gbeg{2}{5}
\got{2}{B}\gnl
\gcmu \gnl
\gcl{1} \gbmp{\eta_C} \gnl
\gcl{1} \gcuf{1} \gnl
\gob{1}{B} 
\gend=
\gbeg{2}{5}
\got{2}{B}\gnl
\gcmu \gnl
\gbmp{\eta_C} \gcl{1} \gnl
\glmf \gcn{1}{1}{-1}{-1} \gnl
\gvac{1} \gob{1}{B} 
\gend
\end{equation}
\end{tabular}
\end{center} \vspace{-0,5cm}
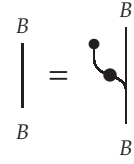
$$ \textnormal{ \footnotesize quasi action}  \hspace{5,5cm}  \textnormal{\footnotesize quasi action unity} $$ \vspace{-0,7cm}
$$ \textnormal{\footnotesize from 2-cell cond. of $\mu_C$ \equref{2-cell mu_C}} \hspace{4cm} \textnormal{\footnotesize from 2-cell cond. of $\eta_C$ \equref{2-cell eta_C}} $$

\begin{center} 
\begin{tabular} {p{6cm}p{1.5cm}p{6cm}} 
\begin{equation}\eqlabel{weak coaction}
\gbeg{3}{5}
\got{3}{B} \gnl
\gwcm{3} \gnl
\gcl{1} \hspace{-0,42cm} \glmf \gcl{1} \gnl 
\gvac{1} \gcl{1} \glmptb \gnot{\hspace{-0,34cm}\Delta_C} \grmptb \gnl
\gvac{1} \gob{1}{B} \gob{1}{F} \gob{1}{F} 
\gend=
\gbeg{3}{6}
\got{4}{B} \gnl
\gvac{1} \gcmu \gnl
\glmpb \gnot{\hspace{-0,34cm}\rho} \grmptb \gcl{1} \gnl
\gcl{1} \glmptb \gnot{\hspace{-0,34cm}\phi} \grmptb \gnl
\glmptb \gnot{\hspace{-0,34cm}\phi} \grmptb \gcl{1} \gnl
\gob{1}{B} \gob{1}{F} \gob{1}{F} 
\gend
\end{equation} & & \vspace{0,1cm}
\begin{equation}\eqlabel{weak coaction counity}
\gbeg{3}{5}
\got{3}{B}\gnl
\gwcm{3} \gnl
\gcl{1} \hspace{-0,42cm} \glmf \gcl{1} \gnl
\gvac{1} \gcl{1} \glmpt \gnot{\hspace{-0,34cm}\Epsilon_C} \grmpt \gnl
\gob{3}{B}
\gend=
\gbeg{3}{5}
\got{4}{B}\gnl
\guf{1} \gcmu \gnl
\glmpt \gnot{\hspace{-0,34cm}\Epsilon_C} \grmpt \gcl{2} \gnl
\gob{5}{B}
\gend
\end{equation} 
\end{tabular}
\end{center} \vspace{-0,5cm}
$$ \textnormal{ \footnotesize twisted coaction}  \hspace{5,5cm}  \textnormal{\footnotesize twisted coaction counity} $$ \vspace{-0,7cm}
$$ \textnormal{\footnotesize from 2-cell cond. of $\Delta_C$ \equref{2-cell Delta_C}} \hspace{4cm} \textnormal{\footnotesize from 2-cell cond. of $\Epsilon_C$ \equref{2-cell Epsilon_C}} $$

\begin{center} \hspace{4cm} 
\begin{tabular} {p{6cm}p{1cm}p{6.8cm}} 
\begin{equation}\eqlabel{2-cocycle condition}
\gbeg{3}{5}
\got{1}{F} \got{1}{F} \got{1}{F}\gnl
\gcl{1} \glmptb \gnot{\hspace{-0,34cm}\mu_M} \grmptb \gnl
\glmpt \gnot{\hspace{-0,34cm}\sigma} \grmptb \gcl{1} \gnl
\gvac{1} \gmu \gnl
\gob{4}{B}
\gend=
\gbeg{3}{6}
\got{1}{F} \got{1}{F} \got{1}{F} \gnl
\glmptb \gnot{\hspace{-0,34cm}\mu_M} \grmptb \gcl{1} \gnl
\gcl{1} \glmptb \gnot{\hspace{-0,34cm}\psi} \grmptb \gnl
\glmpt \gnot{\hspace{-0,34cm}\sigma} \grmptb \gcl{1} \gnl
\gvac{1} \gmu \gnl
\gob{4}{B}
\gend
\end{equation} & & 
\begin{equation}\eqlabel{normalized 2-cocycle}
\gbeg{3}{6}
\got{1}{} \got{3}{F}\gnl
\glmpb \gnot{\hspace{-0,34cm}\eta_M} \grmpb \gcl{1} \gnl
\gcl{1} \glmptb \gnot{\hspace{-0,34cm}\psi} \grmptb \gnl
\glmpt \gnot{\hspace{-0,34cm}\sigma} \grmptb \gcl{1} \gnl
\gvac{1} \gmu \gnl
\gob{4}{B}
\gend=
\gbeg{1}{4}
\got{1}{F}\gnl
\gcu{1} \gnl
\gu{1} \gnl
\gob{1}{B}
\gend=
\gbeg{3}{5}
\got{1}{F}\gnl
\gcl{1} \glmpb \gnot{\hspace{-0,34cm}\eta_M} \grmpb \gnl
\glmpt \gnot{\hspace{-0,34cm}\sigma} \grmptb \gcl{1} \gnl
\gvac{1} \gmu \gnl
\gob{4}{B}
\gend
\end{equation} 
\end{tabular}
\end{center} \vspace{-0,5cm}
$$ \textnormal{ \footnotesize 2-cocycle condition}  \hspace{5,5cm}  \textnormal{\footnotesize normalized 2-cocycle} $$ \vspace{-0,7cm}
$$ \textnormal{\footnotesize from monad law for $\mu_M$ \equref{monad law mu_M}} \hspace{4cm} \textnormal{\footnotesize from monad law for $\eta_M$ \equref{monad law eta_M}} $$

\begin{center} 
\begin{tabular}{p{7.2cm}p{1cm}p{6.8cm}}
\begin{equation} \eqlabel{3-cycle omega} 
\gbeg{5}{4}
\got{1}{F} \got{1}{F} \got{1}{F} \got{2}{B} \gnl
\gcl{1} \gcl{1} \gcl{1} \gcmu \gnl
\gcl{1} \glmpt \gnot{\mu_C} \gcmptb \grmpt \gcl{1} \gnl
\gcn{2}{2}{1}{4} \gcl{1} \gcn{2}{2}{3}{0}  \gnl
\gvac{2} \gelt{\omega} \gnl
\gend=
\gbeg{4}{4}
\got{1}{F} \got{1}{F} \got{1}{F} \got{2}{B} \gnl
\gcl{2} \gcl{2} \gcl{1} \gcmu \gnl
\gvac{2} \glmptb \gnot{\hspace{-0,34cm}\phi} \grmptb \gcl{2} \gnl
\glmpt \gnot{\mu_C} \gcmpt \grmptb \gcl{1} \gnl
\gvac{2} \gcn{1}{1}{1}{2} \gelt{\omega} \gcn{1}{1}{1}{0} \gnl
\gend
\end{equation} &  & \vspace{-0,9cm}
\begin{equation}\eqlabel{normalized 3-cycle omega}
\gbeg{3}{4}
\got{1}{F} \got{2}{B} \gnl
\gcl{1}\gcmu \gnl
\gcl{1} \gbmp{\s\eta_{C}} \gcl{1} \gnl
\gcn{1}{1}{1}{2} \gelt{\omega} \gcn{1}{1}{1}{0} \gnl
\gend=
\gbeg{2}{4}
\got{1}{F} \got{1}{B} \gnl
\gcl{1} \gcl{1} \gnl
\gcuf{1} \gcu{1} \gnl
\gend=
\gbeg{4}{5}
\got{1}{F} \got{2}{B} \gnl
\gcl{1} \gcmu \gnl
\glmptb \gnot{\hspace{-0,34cm}\phi} \grmptb \gcl{2} \gnl
\gbmp{\s\eta_{C}} \gcl{1} \gnl
\gcn{1}{1}{1}{2} \gelt{\omega} \gcn{1}{1}{1}{0} \gnl
\gend
\end{equation}
\end{tabular}
\end{center} \vspace{-0,5cm}
$$ \textnormal{ \footnotesize 3-cycle condition $\omega$}  \hspace{5,5cm}  \textnormal{\footnotesize normalized 3-cycle $\omega$} $$ \vspace{-0,7cm}
$$ \textnormal{\footnotesize from monad law for $\mu_C$ \equref{monad law mu_C}} \hspace{4cm} \textnormal{\footnotesize from monad law for $\eta_C$ \equref{monad law eta_C}} $$

\begin{center} 
\begin{tabular}{p{7.2cm}p{1cm}p{6.8cm}}
\begin{equation} \eqlabel{3-cocycle cond fi-lambda} 
\gbeg{5}{5}
\gvac{2} \gcn{1}{1}{2}{1} \gelt{\s\Phi_{\lambda}} \gcn{1}{1}{0}{1} \gnl
\glmpb \gnot{\Delta_M} \gcmpb \grmptb \gcl{1} \gcl{1} \gnl
\gcl{2} \gcl{2}  \glmptb \gnot{\hspace{-0,34cm}\psi} \grmptb \gcl{1} \gnl
\gvac{2} \gcl{1} \gmu \gnl
\gob{1}{F} \gob{1}{F} \gob{1}{F} \gob{2}{B} \gnl
\gend=
\gbeg{5}{5}
\gcn{1}{1}{2}{1} \gelt{\s\Phi_{\lambda}} \gcn{2}{2}{0}{5}  \gnl
\gcl{1} \gcl{1} \gnl
\gcl{1} \glmptb \gnot{\Delta_M} \gcmpb \grmpb \gcl{1} \gnl
\gcl{1} \gcl{1} \gcl{1} \gmu \gnl
\gob{1}{F} \gob{1}{F} \gob{1}{F} \gob{2}{B} \gnl
\gend
\end{equation} &  &
\begin{equation}\eqlabel{normalized 3-cocycle fi-lambda}
\gbeg{3}{5}
\gcn{1}{1}{2}{1} \gelt{\s\Phi_{\lambda}} \gcn{1}{1}{0}{1} \gnl 
\gbmp{\Epsilon_M} \gcl{1} \gcl{2} \gnl
\glmptb \gnot{\hspace{-0,34cm}\psi} \grmptb \gnl
\gcl{1} \gmu \gnl
\gob{1}{F} \gob{2}{B} \gnl
\gend=
\gbeg{2}{3}
\gu{1} \gu{1} \gnl
\gcl{1} \gcl{1} \gnl
\gob{1}{F} \gob{1}{B} \gnl
\gend=
\gbeg{3}{4}
\gcn{1}{1}{2}{1} \gelt{\s\Phi_{\lambda}} \gcn{1}{1}{0}{1} \gnl 
\gcl{1} \gbmp{\Epsilon_M} \gcl{1} \gnl
\gcl{1} \gmu \gnl
\gob{1}{F} \gob{2}{B} \gnl
\gend
\end{equation}
\end{tabular}
\end{center} \vspace{-0,5cm}
$$ \textnormal{ \footnotesize 3-cocycle condition for $\Phi_{\lambda}$}  \hspace{5cm}  \textnormal{\footnotesize normalized 3-cocycle $\Phi_{\lambda}$} $$ \vspace{-0,7cm}
$$ \textnormal{\footnotesize from comonad law for $\Delta_M$ \equref{comonad law Delta_M}} \hspace{4cm} \textnormal{\footnotesize from comonad law for $\Epsilon_M$ \equref{comonad law Epsilon_M}} $$

\begin{center} \hspace{3.5cm} 
\begin{tabular}{p{6cm}p{1.5cm}p{6.8cm}}
\begin{equation} \eqlabel{2-cycle ro} 
\gbeg{3}{6}
\got{4}{B}\gnl
\gvac{1} \gcmu \gnl
\glmpb \gnot{\hspace{-0,34cm}\rho} \grmptb \gcl{1} \gnl
\gcl{1} \glmptb \gnot{\hspace{-0,34cm}\phi} \grmptb \gnl
\glmptb \gnot{\hspace{-0,34cm}\Delta_C} \grmptb \gcl{1} \gnl
\gob{1}{F}\gob{1}{F}  \gob{1}{F} 
\gend=
\gbeg{3}{5}
\got{4}{B}\gnl
\gvac{1} \gcmu \gnl
\glmpb \gnot{\hspace{-0,34cm}\rho} \grmptb \gcl{1} \gnl
\gcl{1} \glmptb \gnot{\hspace{-0,34cm}\Delta_C} \grmptb \gnl
\gob{1}{F}\gob{1}{F}  \gob{1}{F} 
\gend
\end{equation} &  &
\begin{equation}\eqlabel{normalized 2-cycle ro}
\gbeg{3}{6}
\got{4}{B}\gnl
\gvac{1} \gcmu \gnl
\glmpb \gnot{\hspace{-0,34cm}\rho} \grmptb \gcl{1} \gnl
\gcl{1} \glmptb \gnot{\hspace{-0,34cm}\phi} \grmptb \gnl
\glmpt \gnot{\hspace{-0,34cm}\Epsilon_C} \grmpt \gcl{1} \gnl
\gob{5}{F} 
\gend=
\gbeg{1}{4}
\got{1}{B} \gnl
\gcu{1} \gnl
\guf{1} \gnl
\gob{1}{F} 
\gend=
\gbeg{3}{5}
\got{4}{B}\gnl
\gvac{1} \gcmu \gnl
\glmpb \gnot{\hspace{-0,34cm}\rho} \grmptb \gcl{1} \gnl
\gcl{1} \glmpt \gnot{\hspace{-0,34cm}\Epsilon_C} \grmpt \gnl
\gob{1}{F} \gnl
\gend
\end{equation}
\end{tabular}
\end{center} \vspace{-0,5cm}
$$ \textnormal{ \footnotesize 2-cycle condition for $\rho$}  \hspace{5,5cm}  \textnormal{\footnotesize normalized 2-cycle $\rho$} $$ \vspace{-0,7cm}
$$ \textnormal{\footnotesize from comonad law for $\Delta_C$ \equref{comonad law Delta_C}} \hspace{4cm} \textnormal{\footnotesize from comonad law for $\Epsilon_C$ \equref{comonad law Epsilon_C}} $$

\subsection{Canonical structures}

In \rmref{embed into complete} we defined the inclusion 2-functors $E_M: \Mnd(\K)\to\EM^M(\K)$ and $E_C: \Comnd(\K)\to\EM^C(\K)$. For a 2-cell $\rho_M: F\to G$ in $\Mnd(\K)$
it is $E_M(\rho_M)=\rho_M\times\eta_B$, and for a 2-cell $\rho_C: F\to G$ it is $E_C(\rho_C)=\rho_C\times\Epsilon_B$. We also have the projection 2-functors as in \ssref{Inc and proj}:
$\pi_M: \EM^M(\K)\to\Mnd(\K)$ and $\pi_C: \EM^C(\K)\to\Comnd(\K)$, given 2-cells $\rho_M: F\to G$ in $\EM^M(\K)$ and $\rho_C: F\to G$ in $\EM^C(\K)$,
with their determining 2-cells $\hat\rho_M: F\to GB$ and $\hat\rho_C: FB\to G$ in $\K$, it is
$\pi_M(\tau_M)=(G\times\Epsilon_B)\comp\tau_M$ and $\pi_C(\tau_C)=\tau_C\comp(F\times\eta_B)$. Clearly, it is: $\pi_M\comp E_M=\Id$ and $\pi_C\comp E_C=\Id$.


\begin{defn} \delabel{canonical}
We say that a 2-cell  $\rho_M$ in $\EM^M(\K)$ is {\em canonical} if $\rho_M=E_M\comp\pi_M(\rho_M)$. \\
We say that a 2-cell  $\rho_C$ in $\EM^C(\K)$ is {\em canonical} if $\rho_C=E_C\comp\pi_C(\rho_C)$. \\
\end{defn}

Pictorially, 2-cells $\rho_M: F\to G$ in $\EM^M(\K)$ and $\rho_C: F\to G$ in $\EM^C(\K)$, with their determining 2-cells $\hat\rho_M: F\to GB$ and $\hat\rho_C: FB\to G$ in $\K$,
are canonical if it holds:
\begin{equation} \eqlabel{canonical 2-cells}
\hat\rho_M=
\gbeg{2}{5}
\got{1}{F} \gnl
\glmptb \gnot{\hspace{-0,34cm}\hat\rho_M} \grmpb \gnl
\gcl{1} \gcu{1} \gnl
\gcl{1} \gu{1} \gnl
\gob{1}{G} \gob{1}{B} \gnl
\gend
\qquad\textnormal{and}\qquad
\hat\rho_C=
\gbeg{2}{5}
\got{1}{F} \got{1}{B} \gnl
\gcl{1} \gcu{1} \gnl
\gcl{1} \gu{1} \gnl
\glmptb \gnot{\hspace{-0,34cm}\hat\rho_C} \grmp \gnl
\gob{1}{G.} \gnl
\gend
\end{equation} 
In this case, we will indeed say that the 2-cells $\hat\rho_M$ and $\hat\rho_C$ in $\K$ are {\em canonical}.

For any 2-cells $\hat\rho_M,\hat\rho_C$ in $\K$ determining $\rho_M$ in $\EM^M(\K)$ and $\rho_C$ in $\EM^C(\K)$, we will call
$\pi_M(\hat\rho_M)$ and $\pi_C(\hat\rho_C)$ their {\em canonical restrictions}, we will denote them by $\widetilde{\hat\rho_M}$ and $\widetilde{\hat\rho_C}$.

\medskip

We say that a 2-cocycle $\sigma$, 2-cycle $\rho$, 3-cocycle $\Phi_{\lambda}$ and 3-cycle $\omega$, respectively, is trivial, if the following corresponding identity in:
$$\sigma=
\gbeg{2}{4}
\got{1}{F} \got{1}{F} \gnl
\gcu{1} \gcu{1} \gnl
\gu{1} \gnl
\gob{1}{B}
\gend \hspace{0,1cm}, \quad
\rho=
\gbeg{2}{4}
\got{1}{B} \gnl
\gcu{1} \gnl
\gu{1} \gu{1} \gnl
\gob{1}{F} \gob{1}{F}
\gend \hspace{0,1cm}, \quad
\Phi_{\lambda}=
\gbeg{3}{4}
\got{1}{} \gnl
\gu{1} \gu{1} \gu{1} \gnl
\gcl{1} \gcl{1} \gcl{1} \gnl
\gob{1}{F} \gob{1}{F} \gob{1}{B}
\gend \hspace{0,1cm}, \quad
\omega=
\gbeg{3}{4}
\got{1}{F} \got{1}{F} \got{1}{B} \gnl
\gcl{1} \gcl{1} \gcl{1} \gnl
\gcuf{1} \gcuf{1} \gcu{1} \gnl
\gob{1}{}
\gend
$$
holds. We have the next three obvious results, we state them for the sake of completeness. Having the definitions \equref{right F-mod} -- \equref{omega} in mind, 
in view of the identities \equref{2-cell eta_M}-\equref{2-cell eta_C}, \equref{2-cell Epsilon_M}-\equref{2-cell Epsilon_C}, 
\equref{1 lambda_M}, \equref{3 lambda_M}, \equref{1 lambda_C}, \equref{2 lambda_C} we get:

\begin{prop} \prlabel{trivial structures}
\begin{enumerate}
\item If $\eta_M$ is canonical, then $B$ is a trivial left $F$-comodule and $\Phi_{\lambda}$ is trivial. 
\item If $\Epsilon_M$ is canonical, then $B$ is a trivial right $F$-module and $\sigma$ is trivial. 
\item If $\eta_C$ is canonical, then $B$ is a trivial right $F$-comodule and $\rho$ is trivial. 
\item If $\Epsilon_C$ is canonical, then $B$ is a trivial left $F$-module and $\omega$ is trivial. 
\end{enumerate}
\end{prop}

\begin{prop} \prlabel{can implies assoc}
\begin{enumerate}
\item If the 2-cells $\mu_M, \eta_M$ are canonical, then $(F, \widetilde{\mu_M}, \widetilde{\eta_M})$ is a monad in $\K$.
\item If the 2-cells $\mu_C, \eta_C$ are canonical, then $(F, \widetilde{\mu_C}, \widetilde{\eta_C})$ is a monad in $\K$.
\item If the 2-cells $\Delta_M, \Epsilon_M$ are canonical, then $(F, \widetilde{\Delta_M}, \widetilde{\Epsilon_M})$ is a comonad in $\K$.
\item If the 2-cells $\Delta_C, \Epsilon_C$ are canonical, then $(F, \widetilde{\Delta_C}, \widetilde{\Epsilon_C})$ is a comonad in $\K$.
\item If the 2-cells $\mu_M, \eta_M, \Delta_M, \Epsilon_M, \lambda_M$ are canonical, then $(F, \widetilde{\mu_M}, \widetilde{\eta_M},
      \widetilde{\Delta_M}, \widetilde{\Epsilon_M}, \widetilde{\lambda_M})$ is a bimonad in $\K$.
\item If the 2-cells $\mu_C, \eta_C, \Delta_C, \Epsilon_C, \lambda_C$ are canonical, then $(F, \widetilde{\mu_C}, \widetilde{\eta_C},
\widetilde{\Delta_C}, \widetilde{\Epsilon_C}, \widetilde{\lambda_C})$ is a bimonad in $\K$.
\end{enumerate}
\end{prop}

\begin{proof}
The first 4 parts follow from the identities \equref{weak assoc. mu_M} to \equref{weak counit Epsilon_C}. The part (5) follows from the identities \equref{1 lambda_M} to
\equref{8 lambda_M} using that $\eta_B$ is a monomorphism, {\em i.e.} left cancelable. The part (6) follows from the identities \equref{1 lambda_C} to
\equref{8 lambda_C} using that $\Epsilon_B$ is an epimorphism, {\em i.e.} right cancelable. That $\eta_B$ is a monomorphism and $\Epsilon_B$ an epimorphism it follows
from \equref{epsilon-eta B}, that is $\Epsilon_B\comp\eta_B=\Id_{\id_{\A}}$.
\qed\end{proof}

The following result follows from the identities  \equref{weak action} - \equref{weak coaction counity}: 

\begin{prop}
\begin{enumerate}
\item If a 2-cocycle $\sigma$ is trivial and $\mu_M, \eta_M$ are canonical, then $B$ is a right $F$-module.
\item If a 3-cocycle $\Phi_{\lambda}$ is trivial and $\Delta_M, \Epsilon_M$ are canonical, then $B$ is a left $F$-comodule.
\item If a 3-cycle $\omega$ is trivial and $\mu_C, \eta_C$ are canonical, $B$ is a left $F$-module.
\item If a 2-cycle $\rho$ is trivial and $\Delta_C, \Epsilon_C$ are canonical, then $B$ is a right $F$-comodule.
\end{enumerate}
\end{prop}

\subsection{Combining the above to obtain some new identities} \sslabel{secondary}

\textbf{Counit and $\sigma$.} 
When we apply $\Epsilon_B$ to the 2-cocycle condition \equref{2-cocycle condition}, we get: 
$$
\gbeg{2}{5}
\got{1}{F} \got{1}{\hspace{-0,34cm}F} \got{1}{\hspace{-0,34cm}F} \gnl
\gcl{1} \hspace{-0,22cm} \gmu \gnl
\gvac{1} \hspace{-0,22cm} \glmpt \gnot{\hspace{-0,34cm}\sigma} \grmptb \gnl
\gvac{2} \gcu{1} \gnl
\gob{3}{}
\gend=
\gbeg{2}{5}
\got{1}{F} \got{1}{F} \got{1}{F} \gnl
\glmptb \gnot{\hspace{-0,34cm}\mu_M} \grmptb \gcl{1} \gnl
\gcl{1} \glm \gnl
\glmpt \gnot{\sigma} \gcmp \grmptb \gnl
\gvac{2} \gcu{1} \gnl
\gend
\qquad
\stackrel{FF
\gbeg{2}{2}
\gu{1} \gnl
\gob{1}{F}
\gend
}{\Rightarrow}
\gbeg{2}{6}
\got{1}{F} \got{1}{\hspace{-0,34cm}F}  \gnl
\gcl{2} \gcn{1}{1}{0}{0} \hspace{-0,22cm} \gu{1} \gnl
\gvac{1} \gmu \gnl
\gvac{1} \hspace{-0,22cm} \glmpt \gnot{\hspace{-0,34cm}\sigma} \grmptb \gnl
\gvac{2} \gcu{1} \gnl
\gob{3}{}
\gend=
\gbeg{2}{5}
\got{1}{F} \got{1}{F} \gnl
\glmptb \gnot{\hspace{-0,34cm}\mu_M} \grmptb \gu{1} \gnl
\gcl{1} \glm \gnl
\glmpt \gnot{\sigma} \gcmp \grmptb \gnl
\gvac{2} \gcu{1} \gnl
\gend
$$
which by \equref{weak unity eta_M} and \equref{mod alg unity} is further equivalent to: 
\begin{equation} \eqlabel{epsilon to sigma}
\gbeg{2}{4}
\got{1}{F} \got{1}{F} \gnl
\glmpt \gnot{\hspace{-0,34cm}\sigma} \grmptb \gnl
\gvac{1} \gcu{1} \gnl
\gob{1}{}
\gend=
\gbeg{2}{7}
\got{1}{F} \got{1}{F} \gnl
\glmptb \gnot{\hspace{-0,34cm}\mu_M} \grmptb \gnl
\gcl{1} \gcu{1} \gnl
\gcl{1} \gu{1} \gnl
\glmpt \gnot{\hspace{-0,34cm}\sigma} \grmptb \gnl
\gvac{1} \gcu{1} \gnl
\gend\stackrel{*}{\stackrel{\equref{normalized 2-cocycle}}{=}}
\gbeg{2}{5}
\got{1}{F} \got{1}{F} \gnl
\gmu \gnl
\gvac{1} \hspace{-0,22cm} \gcu{1} \gnl
\gvac{1} \gu{1} \gnl
\gvac{1} \gcu{1} \gnl
\gend\stackrel{\equref{3 lambda_M implies}}{\stackrel{\equref{1 lambda_M implies}}{=}   }
\gbeg{2}{5}
\got{1}{F} \got{1}{F} \gnl
\gcl{1} \gcl{1} \gnl
\gcu{1} \gcu{1} \gnl
\gend
\end{equation}
where at the place * the equality holds {\em if $\eta_M$ is canonical}. 

\textbf{Unit and $\rho$.} 
Dually to the above, from the 2-cycle condition \equref{2-cycle ro} and if $\Epsilon_C$ is canonical one gets: 
\begin{equation} \eqlabel{unit to ro}
\gbeg{2}{4}
\got{1}{} \gnl
\gvac{1} \gu{1} \gnl
\glmpb \gnot{\hspace{-0,34cm}\rho} \grmptb \gnl
\gob{1}{F} \gob{1}{F} \gnl
\gend=
\gbeg{2}{4}
\got{1}{} \gnl
\guf{1} \guf{1} \gnl
\gcl{1} \gcl{1} \gnl
\gob{1}{F} \gob{1}{F} \gnl
\gend
\end{equation}

\textbf{Counit and $\Phi_{\lambda}$.} 
\begin{equation} \eqlabel{fi-lambda normalizado}
\gbeg{3}{3}
\gcn{1}{1}{2}{1} \gelt{\s\Phi_{\lambda}} \gcn{1}{1}{0}{1} \gnl  %
\gcl{1} \gcl{1} \gcu{1} \gnl
\gob{1}{F} \gob{1}{F} \gnl
\gend=
\gbeg{3}{4}
\gvac{1} \gu{1} \gnl
\glmpb \gnot{\Delta_M} \gcmptb \grmpb \gnl
\gcl{1} \gcl{1} \gcu{1} \gnl
\gob{1}{F} \gob{1}{F} \gnl
\gend=
\gbeg{1}{4}
\got{2}{}  \gnl
 \gu{1} \gnl
\hspace{-0,34cm} \gcmu \gnl
\gob{1}{F} \gob{1}{F} \gnl
\gend\stackrel{\equref{2 lambda_M implies}}{=}
\gbeg{2}{4}
\got{1}{} \gnl
\gu{1} \gu{1} \gnl
\gcl{1} \gcl{1} \gnl
\gob{1}{F} \gob{1}{F}
\gend
\end{equation}
\textbf{Unit and $\omega$.} 
Dually: 
\begin{equation} \eqlabel{omega normalizado}
\gbeg{3}{3}
\got{1}{F} \got{1}{F} \gnl
\gcl{1} \gcl{1} \gu{1} \gnl
\gcn{1}{1}{1}{2} \gelt{\omega} \gcn{1}{1}{1}{0} \gnl   
\gend=
\gbeg{3}{5}
\got{1}{F} \got{1}{F} \gnl
\gcl{1} \gcl{1} \gu{1} \gnl
\glmpt \gnot{\mu_C} \gcmptb \grmpt \gnl
\gvac{1} \gcuf{1} \gnl
\gend=
\gbeg{2}{4}
\got{1}{F} \got{1}{F} \gnl
\gcl{1} \gcl{1} \gnl
\gcuf{1} \gcuf{1} \gnl
\gob{1}{}
\gend
\end{equation}

\textbf{(Co)module-(co)unit relations for $B$.} 
From the $F$-(co)module (co)monad identities \equref{mod alg}-\equref{comod coalg counity} for $B$ we have four relations between the (co)module structure and the (co)unit of $B$. 
The resting four of such relations we get like follows: 
\begin{equation} \eqlabel{left F-comod. coalg.}
\gbeg{2}{4}
\got{3}{B} \gnl
\grmo \gvac{1} \gcl{1} \gnl
\gcl{1} \gcu{1} \gnl
\gob{1}{F}
\gend\stackrel{\equref{left F-comod}}{=}
\gbeg{2}{5}
\got{1}{B} \gnl
\gcl{1} \gu{1} \gnl
\glmptb \gnot{\hspace{-0,34cm}\psi} \grmptb \gnl
\gcl{1} \gcu{1} \gnl
\gob{1}{F}  
\gend\stackrel{\equref{left B-mod}}{=}
\gbeg{2}{4}
\got{1}{B}  \gnl
\gcl{1} \gu{1} \gnl
\glm \gnl
\gob{3}{F} 
\gend\stackrel{\equref{mod alg unity}}{=}
\gbeg{2}{6}
\got{1}{B} \gnl
\gcl{1} \gnl
\gcu{1} \gnl
\gu{1} \gnl
\gcl{1} \gnl
\gob{1}{F} 
\gend
\end{equation}

\begin{equation} \eqlabel{right F-mod. coalg.}
\gbeg{2}{4}
\got{1}{B} \got{1}{F} \gnl
\grmo \gcl{1} \gnl 
\gcu{1} \gnl
\gob{1}{}
\gend\stackrel{\equref{right F-mod}}{=}
\gbeg{2}{5}
\got{1}{B} \got{1}{F} \gnl
\gcl{1} \gcl{1} \gnl
\glmptb \gnot{\hspace{-0,34cm}\psi} \grmptb \gnl
\gcu{1} \gcu{1} \gnl
\gob{3}{} 
\gend\stackrel{\equref{left B-mod}}{=}
\gbeg{2}{4}
\got{1}{B} \got{1}{F} \gnl
\glm \gnl
\gvac{1} \gcu{1} \gnl
\gob{1}{} 
\gend\stackrel{\equref{mod coalg counity}}{=}
\gbeg{2}{4}
\got{1}{B} \got{1}{F} \gnl
\gcl{1} \gcl{1} \gnl
\gcu{1} \gcu{1} \gnl
\gob{1}{} 
\gend
\end{equation}

\begin{equation} \eqlabel{right F-comod. alg.}
\gbeg{2}{4}
\got{1}{} \gnl
\gu{1} \gnl
\gcl{1} \hspace{-0,42cm} \glmf \gnl  
\gvac{1} \gob{1}{B} \gob{1}{F}
\gend\stackrel{\equref{right F-comod}}{=}
\gbeg{2}{5}
\got{3}{} \gnl
\guf{1} \gu{1} \gnl
\glmptb \gnot{\hspace{-0,34cm}\phi} \grmptb \gnl
\gcl{1} \gcl{1} \gnl
\gob{1}{B} \gob{1}{F} 
\gend\stackrel{\equref{left B-comod}}{=}
\gbeg{2}{4}
\got{1}{}  \gnl
\gvac{1} \guf{1} \gnl
\glcm \gnl
\gob{1}{B} \gob{1}{F} 
\gend\stackrel{\equref{comod alg unity}}{=}
\gbeg{2}{4}
\got{1}{} \gnl
\gu{1} \guf{1} \gnl
\gcl{1} \gcl{1} \gnl
\gob{1}{B} \gob{1}{F} 
\gend
\end{equation}

\begin{equation} \eqlabel{left F-mod. alg.}
\gbeg{2}{4}
\got{1}{F} \gnl
\gcl{1} \gu{1} \gnl
\glmf \gcn{1}{1}{-1}{-1} \gnl
\gob{3}{B}
\gend\stackrel{\equref{left F-mod}}{=}
\gbeg{2}{5}
\got{1}{F} \gnl
\gcl{1} \gu{1} \gnl
\glmptb \gnot{\hspace{-0,34cm}\phi} \grmptb \gnl
\gcl{1} \gcuf{1} \gnl
\gob{1}{B}  
\gend\stackrel{\equref{left B-comod}}{=}
\gbeg{2}{4}
\got{3}{F}  \gnl
\glcm \gnl
\gcl{1} \gcuf{1} \gnl
\gob{1}{B} 
\gend\stackrel{\equref{comod coalg counity}}{=}
\gbeg{2}{4}
\got{1}{F} \gnl
\gcuf{1} \gnl
\gu{1} \gnl
\gob{1}{B} 
\gend
\end{equation}

\textbf{Expressions for $\mu$'s and $\Delta$'s.} 
The $\Delta\x\eta$ compatibility conditions \equref{8 lambda_M} and \equref{8 lambda_C} yield expressions for some structures when some other structures are canonical. We will 
differentiate between the following cases:
\begin{enumerate}
\item \underline{If $\Delta_M, \eta_M$ are canonical.} For $\Delta_M$ canonical, \equref{8 lambda_M} becomes:
$$
\gbeg{2}{4}
\got{1}{F} \got{1}{F} \gnl
\glmptb \gnot{\hspace{-0,34cm}\mu_M} \grmptb \gnl
\hspace{-0,22cm} \gcmu \gcn{1}{1}{0}{0} \gnl
\gob{1}{F} \gob{1}{F} \gob{1}{\hspace{-0,34cm}B} \gnl
\gend=
\gbeg{5}{7}
\gvac{1} \got{1}{F} \got{2}{F} \gnl
\gvac{1} \gcl{1} \gcmu \gnl
\glmpb \gnot{\lambda_M} \gcmptb \grmptb \gcl{1} \gnl
\gcl{1} \gcl{1} \glmptb \gnot{\hspace{-0,34cm}\psi} \grmptb \gnl
\gcl{2} \glmptb \gnot{\hspace{-0,34cm}\mu_M} \grmptb \gcl{1} \gnl
\gcl{1} \gcl{1} \gmu \gnl
\gob{1}{F} \gob{1}{F} \gob{2}{B} \gnl
\gend 
\qquad
\stackrel{F
\gbeg{1}{2}
\got{1}{F} \gnl
\gcu{1} \gnl
\gob{1}{} \gnl
\gend
B}{\Rightarrow}
\mu_M=
\gbeg{5}{7}
\gvac{1} \got{1}{F} \got{2}{F} \gnl
\gvac{1} \gcl{1} \gcmu \gnl
\glmpb \gnot{\lambda_M} \gcmptb \grmptb \gcl{1} \gnl
\gcl{1} \gcl{1} \glmptb \gnot{\hspace{-0,34cm}\psi} \grmptb \gnl
\gcl{2} \glmptb \gnot{\hspace{-0,34cm}\sigma} \grmptb \gcl{1} \gnl
\gcl{1} \gcl{1} \gmu \gnl
\gob{1}{F} \gob{1}{F} \gob{2}{B} \gnl
\gend 
$$
We apply $\Epsilon_B$ to this and if $\eta_M$ is canonical, by \equref{epsilon to sigma} this is further equal to:
$$
\gbeg{2}{3}
\got{1}{F} \got{1}{F} \gnl
\gmu \gnl
\gob{2}{F}
\gend=
\gbeg{4}{6}
\gvac{1} \got{1}{F} \got{2}{F} \gnl
\gvac{1} \gcl{1} \gcmu \gnl
\glmpb \gnot{\lambda_M} \gcmptb \grmptb \gcl{1} \gnl
\gcl{1} \gcu{1} \glmptb \gnot{\hspace{-0,34cm}\psi} \grmptb \gnl
\gcl{1} \gvac{1} \gcu{1} \gcu{1} \gnl
\gob{1}{F} \gnl
\gend\stackrel{\equref{left B-mod}}{=}
\gbeg{4}{6}
\gvac{1} \got{1}{F} \got{2}{F} \gnl
\gvac{1} \gcl{1} \gcmu \gnl
\glmpb \gnot{\lambda_M} \gcmptb \grmptb \gcl{1} \gnl
\gcl{1} \gcu{1} \glm \gnl
\gcl{1} \gvac{2} \gcu{1} \gnl
\gob{1}{F} \gnl
\gend\stackrel{\equref{mod alg unity}}{=}
\gbeg{4}{5}
\gvac{1} \got{1}{F} \got{2}{F} \gnl
\gvac{1} \gcl{1} \gcmu \gnl
\glmpb \gnot{\lambda_M} \gcmptb \grmptb \gcu{1} \gnl
\gcl{1} \gcu{1} \gcu{1} \gnl
\gob{1}{F} \gnl
\gend\stackrel{\equref{quasi counity Epsilon_M}}{=}
\gbeg{3}{4}
\got{1}{F} \got{1}{F} \gnl
\glmptb \gnot{\lambda_M} \gcmptb \grmpb \gnl
\gcl{1} \gcu{1} \gcu{1} \gnl
\gob{1}{F} \gnl
\gend
$$
\item \underline{If $\Delta_M, \lambda_M$ are canonical.} In this case \equref{8 lambda_M} equals:
\begin{equation} \eqlabel{mu_m con sigma}
\gbeg{2}{4}
\got{1}{F} \got{1}{F} \gnl
\glmptb \gnot{\hspace{-0,34cm}\mu_M} \grmptb \gnl
\hspace{-0,22cm} \gcmu \gcn{1}{1}{0}{0} \gnl
\gob{1}{F} \gob{1}{F} \gob{1}{\hspace{-0,34cm}B} \gnl
\gend=
\gbeg{3}{5}
\got{1}{F} \got{2}{F} \gnl
\gcl{1} \gcmu \gnl
\glmptb \gnot{\hspace{-0,34cm}\widetilde{\lambda_M}} \grmptb \gcl{1} \gnl
\gcl{1} \glmptb \gnot{\hspace{-0,34cm}\mu_M} \grmptb \gnl
\gob{1}{F} \gob{1}{F} \gob{1}{B} \gnl
\gend\quad
\stackrel{F
\gbeg{1}{2}
\got{1}{F} \gnl
\gcu{1} \gnl
\gob{1}{} \gnl
\gend
B}{\stackrel{\equref{quasi counity Epsilon_M}}{\Rightarrow}} \quad
\mu_M=
\gbeg{3}{5}
\got{1}{F} \got{2}{F} \gnl
\gcl{1} \gcmu \gnl
\glmptb \gnot{\hspace{-0,34cm}\widetilde{\lambda_M}} \grmptb \gcl{1} \gnl
\gcl{1} \glmpt \gnot{\hspace{-0,34cm}\sigma} \grmptb \gnl
\gob{1}{F} \gob{3}{B} \gnl
\gend
\end{equation}
\item \underline{If $\mu_M, \lambda_M$ are canonical.} In this case \equref{8 lambda_M} equals:
\begin{equation} \eqlabel{Delta_m con Fi-lambda}
\gbeg{3}{5}
\gvac{2} \got{1}{\hspace{-0,4cm}F} \got{1}{\hspace{-0,4cm}F} \gnl
\gvac{2} \hspace{-0,34cm} \gmu \gnl
\gvac{1} \hspace{-0,22cm} \glmpb \gnot{\Delta_M} \gcmpb \grmptb \gnl
\gvac{1} \gcl{1} \gcl{1} \gcl{1} \gnl
\gvac{1} \gob{1}{F} \gob{1}{F} \gob{1}{B} \gnl
\gend=
\gbeg{4}{5}
\got{1}{F} \got{1}{F} \gnl
\gcl{1} \glmptb \gnot{\Delta_M} \gcmpb \grmpb \gnl
\glmptb \gnot{\hspace{-0,34cm}\widetilde{\lambda_M}} \grmptb \gcl{1} \gcl{2} \gnl
\gcl{1} \gmu \gnl
\gob{1}{F} \gob{2}{F} \gob{1}{B} \gnl
\gend \quad
\stackrel{F
\gbeg{1}{2}
\gu{1} \gnl
\gob{1}{F} \gnl
\gend}{\stackrel{\equref{weak unity eta_M}}{\Rightarrow}} \quad
\Delta_M=
\gbeg{4}{6}
\got{1}{F} \gnl
\gcl{1} \gu{1} \gnl
\gcl{1} \glmptb \gnot{\Delta_M} \gcmpb \grmpb \gnl
\glmptb \gnot{\hspace{-0,34cm}\widetilde{\lambda_M}} \grmptb \gcl{1} \gcl{2} \gnl
\gcl{1} \gmu \gnl
\gob{1}{F} \gob{2}{F} \gob{1}{B} \gnl
\gend
\end{equation}
\item \underline{If $\mu_C, \Epsilon_C$ are canonical.} 
Dually to the part 1), because of \equref{unit to ro}, we have:
$$
\gbeg{2}{4}
\got{2}{F} \gnl
\gcn{1}{1}{2}{2} \gnl
\gcmuf \gnl
\gob{1}{F} \gob{1}{F} \gnl
\gend=
\gbeg{3}{4}
\got{1}{F} \gnl
\gcl{1} \guf{1} \guf{1} \gnl
\glmpb \gnot{\lambda_C} \gcmptb \grmpt \gnl
\gob{1}{F} \gob{1}{F} \gnl
\gend
$$
\item \underline{If $\mu_C, \lambda_C$ are canonical.} 
Dually to the part 2), it is:
\begin{equation} \eqlabel{Delta_C con ro}
\Delta_C=
\gbeg{3}{5}
\got{1}{F} \got{3}{B} \gnl
\gcl{1} \glmpb \gnot{\hspace{-0,34cm}\rho} \grmptb \gnl
\glmptb \gnot{\hspace{-0,34cm}\widetilde{\lambda_C}} \grmptb \gcl{1} \gnl
\gcl{1} \gmuf \gnl \gnl
\gob{1}{F} \gob{2}{F} \gnl
\gend
\end{equation}
\item \underline{If $\Delta_C, \lambda_C$ are canonical.} 
Dually to the part 3), one has:
\begin{equation} \eqlabel{mu_C con omega}
\mu_C=
\gbeg{4}{7}
\got{1}{F} \got{2}{F} \got{1}{B} \gnl
\gcl{1} \gcn{1}{1}{2}{2} \gvac{1} \gcl{3} \gnl
\gcl{1} \gcmuf{1} \gnl
\glmptb \gnot{\hspace{-0,34cm}\widetilde{\lambda_M}} \grmptb \gcl{1} \gnl
\gcl{1} \glmptb \gnot{\mu_C} \gcmpt \grmpt \gnl
\gcl{1} \gcuf{1} \gnl
\gob{1}{F} \gnl
\gend
\end{equation}
\end{enumerate}

\section{Examples: when $\K$ is given by a braided monoidal category} \selabel{example1}

It is a well-known fact that a 2-category with a single 0-cell is a monoidal category. On the other hand, braidings in braided monoidal categories, 
because of their hexagon axioms and naturality, are distributive laws (in the sense of \equref{psi laws for B} and \equref{phi laws for B}). 
We fix a braided monoidal category $\C$ and set $\K=\hat\C$, the 2-category 
arising from $\C$, that is: $\K$ has a single 0-cell and its 1- and 2-cells are given by the objects and the morphisms of $\C$, respectively. 
In this and the following section we will study different cases of biwreaths and biwreath-like objects in $\hat\C$ giving 2-cells $\psi, \phi$ and $\lambda$'s. It will turn out that 
the latter 2-cells (and the corresponding biwreaths) have their meaning in a specific pre-braided monoidal category $\D$ built upon $\C$.  Namely, these will be the categories 
of Yetter-Drinfel'd modules over a certain bialgebra in $\C$, studied in \cite{Besp, Femic1}. 
The braiding in $\C$ and its inverse we will denote by 
$
\gbeg{2}{1}
\gbr \gnl
\gend
$ \hspace{0,1cm} and 
$
\gbeg{2}{1}
\gibr \gnl
\gend
$, \hspace{0,1cm} respectively. 

\begin{rem} \rmlabel{twist sides}
Strictly speaking, the identities involving compositions of 1-cells and horizontal compositions of 2-cells in $\K$ should be written in the reversed order when 
interpreted in $\C$. This is because the mentioned compositions correspond to tensor products in $\C$, and compositions are read from the right to the left, 
while the corresponding tensor products in $\C$ are read from the left to the right. In order to avoid complications we will not make this reversal of identities 
when working in $\C$ throughout. 
\end{rem}

In all the examples that we will study throughout the paper we will assume the following: 
\begin{enumerate}
\item a bimonad $B$ in $\hat\C$ will be a bialgebra $B$ in $\C$; 
\item the 2-cells $\eta_M, \eta_C, \Epsilon_M$ and $\Epsilon_C$ are canonical. 
\end{enumerate}

The bialgebra axiom \equref{bialg} in $\C$, in view of \equref{bialg-lambda} 
yields that $\lambda_B:BB\to BB$ is given by \equref{lambda_B}:
\begin{center} 
\begin{tabular}{p{6cm}p{1cm}p{4.3cm}}
\begin{equation} \eqlabel{bialg}
\scalebox{0.86}{
\gbeg{3}{5}
\got{1}{B} \got{3}{B} \gnl
\gwmu{3} \gnl
\gvac{1} \gcl{1} \gnl
\gwcm{3} \gnl
\gob{1}{B}\gvac{1}\gob{1}{B}
\gend}=
\scalebox{0.86}{
\gbeg{4}{5}
\got{2}{B} \got{2}{B} \gnl
\gcmu \gcmu \gnl
\gcl{1} \gbr \gcl{1} \gnl
\gmu \gmu \gnl
\gvac{1} \hspace{-0,2cm} \gob{1}{B}\gvac{1}\gob{1}{B}
\gend}
\end{equation} & & \vspace{-0,6cm}
\begin{equation} \eqlabel{lambda_B}
\lambda_B=
\gbeg{4}{5}
\got{2}{B} \got{1}{B} \gnl
\gcmu \gcl{1} \gnl
\gcl{1} \gbr  \gnl
\gmu \gcl{1} \gnl
\gob{2}{B} \gob{1}{B.}
\gend
\end{equation}
\end{tabular}
\end{center}
A biwreath in $\hat\C$ is given by a bialgebra $B$, an object $F$, morphisms $\psi:BF\to FB, \phi: FB\to BF$ and 10 morphisms from \equref{structure 2-cells in K} in $\C$ 
satisfying the corresponding identities (axioms).

\subsection{Case one: $\D_1={}_B ^B\YD(\C)$ }  \sslabel{ex psi_1}


In this particular case we will assume additionally: 
\begin{enumerate}
\item the 2-cells $\lambda_M, \lambda_C$ are canonical and $\widetilde{\lambda_M}= \widetilde{\lambda_C}=:\lambda$;
\item the projections of the 2-cells in $\bEM(\K)$ to the monadic and the comonadic components coincide, that is: \vspace{-0,24cm}
\begin{equation} \eqlabel{canonical restrictions coincide} 
\gbeg{2}{4}
\got{1}{F} \gnl
\glmptb \gnot{\hspace{-0,34cm}\hat\rho_M} \grmpb \gnl
\gcl{1} \gcu{1} \gnl
\gob{1}{G} \gnl
\gend=
\gbeg{2}{4}
\got{1}{F} \gnl
\gcl{1} \gu{1} \gnl
\glmptb \gnot{\hspace{-0,34cm}\hat\rho_C} \grmp \gnl
\gob{1}{G} \gnl
\gend
\end{equation} 
with notations as in \equref{canonical 2-cells}; 
in this case we will also say that the canonical restrictions of the monadic and the comonadic components of the 2-cells in $\bEM(\K)$ coincide. 
\end{enumerate}

Let  $\D_1$ be the category ${}_B ^B\YD(\C)$ of left-left Yetter-Drinfel'd modules over a bialgebra $B$ in $\C$. It is a pre-braided category with 
the braiding 
$$d^1_{X,Y}=
\gbeg{3}{5}
\got{1}{} \got{1}{X} \got{1}{Y} \gnl
\glcm \gcl{1} \gnl
\gcl{1} \gbr \gnl
\glm \gcl{1} \gnl
\gob{1}{} \gob{1}{Y} \gob{1}{X} 
\gend
$$
for any $X,Y\in {}_B ^B\YD(\C)$. Observe that $d^1$ can be evaluated at any left $B$-module and left $B$-comodule. We clearly have that $B$ is a left module and comodule 
over itself, and by \prref{left B-structures of F} we know that $F$ has these structures, too.  The latter structures on $F$ are given by \equref{left B-mod} and \equref{left B-comod}. 
So we may set $\psi=d^1_{B,F}$ and $\phi=d^1_{F,B}$ to obtain: 
$$
\psi=
\gbeg{3}{5}
\got{2}{B} \got{1}{F} \gnl
\gcmu \gcl{1} \gnl
\gcl{1} \gbr \gnl
\glm \gcl{1} \gnl
\gob{1}{} \gob{1}{F} \gob{1}{B} 
\gend \hspace{2cm}
\phi= 
\gbeg{3}{5}
\got{1}{} \got{1}{F} \got{1}{B} \gnl
\glcm \gcl{1} \gnl
\gcl{1} \gbr \gnl
\gmu \gcl{1} \gnl
\gob{2}{B} \gob{1}{F} 
\gend
$$
once we make sure that the axioms \equref{psi laws for B} and \equref{phi laws for B} are fulfilled. 
If both $B$ and $F$ were Yetter-Drinfel'd modules these axioms would be satisfied by the braiding properties. It is known that $B$ is such a module with a regular action and an adjoint coaction 
(or the other way around) if $B$ is a Hopf algebra with a bijective antipode. So far we do not study Hopf wreaths, they will be studied elsewhere, so we need to check the axioms by hand. 
The compatibility with the (co)unity is clear, we prove the one with the multiplication, the one with the comultiplication holds by duality. We find: 
$$
\gbeg{3}{5}
\got{1}{B}\got{1}{B}\got{1}{F}\gnl
\gmu \gcn{1}{1}{1}{0} \gnl
\gvac{1} \hspace{-0,34cm} \glmptb \gnot{\hspace{-0,34cm}\psi} \grmptb  \gnl
\gvac{1} \gcl{1} \gcl{1} \gnl
\gvac{1} \gob{1}{F} \gob{1}{B}
\gend=
\gbeg{3}{6}
\got{1}{B} \got{1}{B} \got{1}{F} \gnl
\gmu \gcl{1} \gnl
\gcmu \gcl{1} \gnl
\gcl{1} \gbr \gnl
\glm \gcl{1} \gnl
\gob{1}{} \gob{1}{F} \gob{1}{B} 
\gend\stackrel{bialg.}{=}
\gbeg{5}{7}
\got{2}{B} \got{2}{B} \got{1}{F} \gnl
\gcmu \gcmu \gcl{2} \gnl
\gcl{1} \gbr \gcl{1} \gnl
\gmu \gmu \gcn{1}{1}{1}{0} \gnl
\gcn{1}{1}{2}{4} \gvac{2} \hspace{-0,34cm} \gbr \gnl
\gvac{2} \glm \gcl{1} \gnl
\gvac{3} \gob{1}{F} \gob{1}{B} 
\gend\stackrel{mod.}{\stackrel{nat.}{=}}
\gbeg{5}{7}
\got{2}{B} \got{2}{B} \got{1}{F} \gnl
\gcmu \gcmu \gcl{1} \gnl
\gcl{1} \gbr \gbr \gnl
\gcl{1} \gcl{1} \gbr \gcl{1} \gnl
\gcn{1}{1}{1}{3} \glm \gmu \gnl
\gvac{1} \glm \gcn{1}{1}{2}{2} \gnl
\gvac{2} \gob{1}{F} \gob{2}{B} 
\gend\stackrel{nat.}{=}
\gbeg{5}{8}
\got{2}{B} \got{2}{B} \got{1}{F} \gnl
\gcn{1}{3}{2}{2} \gvac{1} \gcmu \gcl{1} \gnl
\gvac{2} \gcl{1} \gbr \gnl
\gvac{2} \glm \gcl{3} \gnl
\gcmu \gcn{1}{1}{3}{1} \gnl
\gcl{1} \gbr \gnl
\glm \gwmu{3} \gnl
\gvac{1} \gob{1}{F} \gob{3}{B} 
\gend=
\gbeg{3}{5}
\got{1}{B}\got{1}{B}\got{1}{F}\gnl
\gcl{1} \glmpt \gnot{\hspace{-0,34cm}\psi} \grmptb \gnl
\glmptb \gnot{\hspace{-0,34cm}\psi} \grmptb \gcl{1} \gnl
\gcl{1} \gmu \gnl
\gob{1}{F} \gob{2}{B}
\gend
$$
we have used the intrinsic properties of a biwreath and naturality of the braiding in $\C$ (with respect to the multiplication in $B$ and the left $B$-action on $F$). 
The relation \equref{YD condition} now becomes: 
\begin{equation} \eqlabel{YD}
\gbeg{4}{8}
\got{2}{B} \got{1}{F} \gnl
\gcmu \gcl{1} \gnl
\gcl{1} \gbr \gnl
\glm \gcl{2} \gnl
\glcm \gnl
\gcl{1} \gbr \gnl
\gmu \gcl{1} \gnl
\gob{2}{B} \gob{1}{F} 
\gend=
\gbeg{4}{5}
\got{2}{B} \got{3}{F} \gnl
\gcmu \glcm \gnl
\gcl{1} \gbr \gcl{1} \gnl
\gmu \glm \gnl
\gob{2}{B} \gob{3}{F} 
\gend
\end{equation}
which means that $F$ is a left-left Yetter-Drinfel'd module in $\C$, that is $F\in{}_B ^B\YD(\C)$.

Let us consider 
\begin{equation} \eqlabel{lambda Radford}
\lambda=
\gbeg{4}{7}
\got{1}{} \got{1}{F} \got{3}{F} \gnl
\gwcm{3} \gcl{2} \gnl
\gcl{1} \glcm \gnl
\gcl{1} \gcl{1} \gbr \gnl
\gcl{1} \glm \gcl{2} \gnl
\gwmu{3} \gnl
\gvac{1} \gob{1}{F} \gob{3}{F.}
\gend
\end{equation}

\begin{lma}
The above $\lambda_M$ satisfies the 2-cell condition. Moreover, if $\mu_M, \mu_C, \Delta_M, \Delta_C$ are canonical, $\lambda_M$ satisfies the four distributive law rules.
\end{lma}

\begin{proof}
First we have:
\begin{equation} \eqlabel{psi-psi}
\gbeg{3}{4}
\got{1}{B}\got{1}{B}\got{1}{F}\gnl
\glmptb \gnot{\hspace{-0,34cm}\psi} \grmptb \gcl{1} \gnl
\gcl{1} \glmptb \gnot{\hspace{-0,34cm}\psi} \grmptb \gnl
\gob{1}{F} \gob{1}{B} \gob{1}{B} 
\gend=
\gbeg{5}{8}
\got{2}{B}\got{1}{B}\got{3}{F}\gnl
\gcmu \gcl{1} \gvac{1} \gcl{4} \gnl
\gcl{1} \gbr \gnl
\glm \gcn{1}{1}{1}{2} \gnl
\gvac{1} \gcl{1} \gcmu \gcl{1} \gnl
\gvac{1} \gcl{1} \gcl{1} \gbr \gnl
\gvac{1} \gcl{1} \glm \gcl{1} \gnl
\gvac{1} \gob{1}{F} \gvac{1} \gob{1}{B} \gob{1}{B} 
\gend\stackrel{comod.}{=}
\gbeg{5}{7}
\gvac{2} \got{1}{B}\got{1}{B}\got{1}{F}\gnl
\gvac{1} \glcm \gcl{1} \gcl{2} \gnl
\gcn{1}{1}{3}{2} \gvac{1} \gbr \gnl
\gcmu \gcl{1} \gbr \gnl
\gcl{1} \gbr \gcl{1} \gcl{2} \gnl
\glm \glm \gnl
\gvac{1} \gob{1}{F} \gvac{1} \gob{1}{B} \gob{1}{B} 
\gend
\end{equation}
then, if we prove that $\lambda$ is left $B$-linear, we would have: 
$$
\gbeg{3}{5}
\got{1}{B} \got{1}{F} \got{1}{F} \gnl
\glmptb \gnot{\hspace{-0,34cm}\psi} \grmptb \gcl{1} \gnl
\gcl{1} \glmptb \gnot{\hspace{-0,34cm}\psi} \grmptb \gnl
\glmptb \gnot{\hspace{-0,34cm}\lambda} \grmptb \gcl{1} \gnl
\gob{1}{F} \gob{1}{B} \gob{1}{B} 
\gend\stackrel{\equref{psi-psi}}{=}
\gbeg{5}{8}
\gvac{2} \got{1}{B}\got{1}{B}\got{1}{F}\gnl
\gvac{1} \glcm \gcl{1} \gcl{2} \gnl
\gcn{1}{1}{3}{2} \gvac{1} \gbr \gnl
\gcmu \gcl{1} \gbr \gnl
\gcl{1} \gbr \gcl{1} \gcl{3} \gnl
\glm \glm \gnl
\gvac{1} \glmptb \gnot{\lambda} \gcmp \grmptb \gnl
\gvac{1} \gob{1}{F} \gvac{1} \gob{1}{B} \gob{1}{B} 
\gend\stackrel{B-lin.}{=}
\gbeg{5}{8}
\gvac{2} \got{1}{B}\got{1}{B}\got{1}{F}\gnl
\gvac{1} \glcm \gcl{1} \gcl{2} \gnl
\gcn{1}{1}{3}{2} \gvac{1} \gbr \gnl
\gcn{1}{1}{2}{2} \gvac{1} \gcl{1} \gbr \gnl
\gcmu \glmptb \gnot{\hspace{-0,34cm}\lambda} \grmptb \gcl{3} \gnl
\gcl{1} \gbr \gcl{1} \gnl
\glm \glm \gnl
\gvac{1} \gob{1}{F} \gvac{1} \gob{1}{B} \gob{1}{B} 
\gend\stackrel{nat.}{=}
\gbeg{5}{7}
\gvac{2} \got{1}{B}\got{1}{B}\got{1}{F}\gnl
\gvac{1} \glcm \glmptb \gnot{\hspace{-0,34cm}\lambda} \grmptb \gnl
\gcn{1}{1}{3}{2} \gvac{1} \gbr \gcl{1} \gnl
\gcmu \gcl{1} \gbr \gnl
\gcl{1} \gbr \gcl{1} \gcl{2} \gnl
\glm \glm \gnl
\gvac{1} \gob{1}{F} \gvac{1} \gob{1}{B} \gob{1}{B} 
\gend\stackrel{\equref{psi-psi}}{=}
\gbeg{4}{5}
\got{1}{B} \got{1}{F} \got{1}{F} \gnl
\gcl{1} \glmptb \gnot{\hspace{-0,34cm}\lambda} \grmptb \gnl
\glmptb \gnot{\hspace{-0,34cm}\psi} \grmptb \gcl{1} \gnl
\gcl{1} \glmptb \gnot{\hspace{-0,34cm}\psi} \grmptb \gnl
\gob{1}{F} \gob{1}{F} \gob{1}{B.} 
\gend
$$
So, we prove that $\lambda$ is left $B$-linear: 
$$\hspace{-0,4cm}
\scalebox{0.84}[0.84]{
\gbeg{4}{10}
\got{2}{B}\got{1}{F}\got{1}{F}\gnl
\gcmu \gcl{1} \gcl{2} \gnl
\gcl{1} \gbr \gnl
\glm \glm \gnl
\gwcm{3} \gcl{2} \gnl
\gcl{1} \glcm \gnl
\gcl{1} \gcl{1} \gbr \gnl
\gcl{1} \glm \gcl{2} \gnl
\gwmu{3} \gnl
\gvac{1} \gob{1}{F} \gvac{1} \gob{1}{B} 
\gend}\stackrel{\equref{mod coalg}}{=}
\scalebox{0.84}[0.84]{
\gbeg{5}{11}
\gvac{1} \got{1}{B} \gvac{1} \got{1}{F}\got{1}{F}\gnl
\gwcm{3} \gcl{1} \gcl{2} \gnl
\gcl{1} \gvac{1} \gbr \gnl  
\hspace{-0,2cm} \gcmu \gcmu \hspace{-0,22cm} \glm \gnl
\gvac{1} \hspace{-0,2cm} \gcl{1} \gbr \gcl{1} \gcn{1}{1}{2}{1} \gnl
\gvac{1} \glm \glm \gcl{2} \gnl
\gvac{2} \gcl{1} \glcm \gnl
\gvac{2} \gcl{1} \gcl{1} \gbr\gnl
\gvac{2} \gcl{1} \glm \gcl{2} \gnl
\gvac{2} \gwmu{3} \gnl
\gvac{3} \gob{1}{F} \gvac{1} \gob{1}{B} 
\gend}\stackrel{coass.}{\stackrel{nat.}{\stackrel{mod.}{=}}}
\scalebox{0.84}[0.84]{
\gbeg{6}{14}
\gvac{1} \got{1}{B} \gvac{1} \got{2}{F}\got{1}{F}\gnl
\gwcm{3} \gvac{1} \gcn{1}{2}{0}{0} \gcl{8} \gnl
\gcl{4}  \gvac{1} \hspace{-0,34cm} \gcmu \gnl
\gvac{2} \gcn{1}{1}{1}{0} \gbr \gnl 
\gvac{1} \gcn{1}{1}{2}{2} \gvac{1} \hspace{-0,34cm} \gcmu \gcn{1}{1}{0}{1} \gnl
\gvac{2} \gbr \gcl{1} \gcl{3} \gnl
\gvac{1} \glm \glm \gnl
\gvac{2} \gcl{5} \glcm \gnl
\gvac{3} \gcl{1} \gbr \gnl
\gvac{3} \gmu \gbr \gnl
\gvac{4} \gcn{1}{1}{0}{1} \gcl{1} \gcl{3} \gnl
\gvac{4} \glm \gnl 
\gvac{2} \gwmu{4} \gnl
\gvac{3} \gob{2}{F} \gvac{1} \gob{1}{B} 
\gend}\stackrel{\equref{YD}}{\stackrel{nat.}{=}}
\scalebox{0.84}[0.84]{
\gbeg{6}{10}
\got{2}{B} \gvac{1} \got{1}{F} \got{4}{F}\gnl
\gcmu \gwcm{3} \gvac{1} \gcn{1}{5}{0}{0}  \gnl
\gcl{1} \gbr \gvac{1} \gcn{1}{1}{1}{2} \gnl 
\glm \hspace{-0,22cm} \gcmu \glcm \gnl  
\gvac{1} \gcn{1}{2}{2}{2} \gcl{1} \gbr \gcl{1} \gnl
\gvac{2} \gmu \glm \gnl
\gvac{1} \gcn{1}{2}{2}{3} \gcn{1}{1}{2}{5} \gvac{2} \gbr \gnl
\gvac{4} \glm \gcl{2} \gnl 
\gvac{2} \gwmu{4} \gnl
\gvac{3} \gob{2}{F} \gvac{1} \gob{1}{B} 
\gend}\stackrel{coass.}{\stackrel{mod.}{\stackrel{nat.}{=}}}
\scalebox{0.84}[0.84]{
\gbeg{7}{9}
\gvac{1} \got{2}{B} \gvac{1} \got{1}{F} \got{3}{F}\gnl
\gvac{1} \gcmu \gwcm{3} \gcl{2}  \gnl
\gcn{1}{1}{3}{2} \gvac{1} \gbr \glcm \gnl 
\gcmu \gcl{1} \gcl{1} \gcl{1} \gbr \gnl  
\gcl{1} \gbr \gcn{1}{1}{1}{3} \glm \gcl{2} \gnl
\glm \gcn{1}{1}{1}{3} \gvac{1} \gbr \gnl
\gvac{1} \gcl{1} \gvac{1} \glm \glm \gnl
\gvac{1} \gwmu{4} \gvac{1}  \gcl{1} \gnl
\gvac{2} \gob{2}{F} \gob{5}{B} 
\gend}\stackrel{nat.}{\stackrel{\equref{mod coalg}}{=}}
\scalebox{0.84}[0.84]{
\gbeg{4}{10}
\got{1}{B} \gvac{1} \got{1}{F}\got{3}{F}\gnl
\gcl{4} \gwcm{3} \gcl{2} \gnl
\gvac{1} \gcl{1} \glcm \gnl
\gvac{1} \gcl{1} \gcl{1} \gbr \gnl
\gvac{1} \gcl{1} \glm \gcl{1} \gnl
\gcn{1}{1}{1}{2} \gwmu{3} \gcn{1}{2}{1}{-1} \gnl
\gcmu \gcl{1} \gnl
\gcl{1} \gbr \gcl{1} \gnl
\glm \glm \gnl
\gvac{1} \gob{1}{F} \gvac{1} \gob{1}{B} 
\gend}
$$
It is easy to see that $\lambda_M$ satisfies the distributive laws with respect to the (co)unit. We prove the one for the multiplication, the one for the comultiplication 
is then valid by the auto-duality inside of the biwreath structure. 
Putting $\lambda$ from \equref{lambda Radford} in \equref{8 lambda_M implies} we get: 
\begin{equation} \eqlabel{lambda 8 Radford}
\gbeg{3}{5}
\got{1}{F} \got{3}{F} \gnl
\gwmu{3} \gnl
\gvac{1} \gcl{1} \gnl
\gwcm{3} \gnl
\gob{1}{F}\gvac{1}\gob{1}{F}
\gend=
\gbeg{3}{7}
\got{1}{} \got{1}{F} \got{5}{F} \gnl
\gwcm{3} \gwcm{3} \gnl
\gcl{1} \glcm \gcl{1} \gvac{1} \gcl{3} \gnl
\gcl{1} \gcl{1} \gbr \gnl
\gcl{1} \glm \gcl{1} \gnl
\gwmu{3} \gwmu{3} \gnl
\gvac{1} \gob{1}{F} \gob{5}{F}
\gend
\end{equation}
then:
$$\hspace{-0,2cm}
\gbeg{3}{5}
\got{1}{F} \got{1}{F} \got{1}{F} \gnl
\gmu \gcn{1}{1}{1}{0} \gnl
\gvac{1} \hspace{-0,34cm} \glmptb \gnot{\hspace{-0,34cm}\lambda} \grmptb \gnl
\gvac{1} \gcl{1} \gcl{1} \gnl
\gvac{1} \gob{1}{F} \gob{1}{F} 
\gend=
\scalebox{0.84}[0.84]{
\gbeg{5}{9}
\got{1}{F} \gvac{1} \got{1}{F}\got{1}{F}\gnl
\gwmu{3} \gcl{4} \gnl
\gvac{1} \gcl{1} \gnl
\gwcm{3} \gnl
\gcl{1} \glcm \gnl
\gcl{1} \gcl{1} \gbr \gnl
\gcl{1} \glm \gcl{2} \gnl
\gwmu{3} \gnl
\gvac{1} \gob{1}{F} \gvac{1} \gob{1}{F} 
\gend}
\stackrel{\equref{lambda 8 Radford}}{=}
\scalebox{0.84}[0.84]{
\gbeg{7}{10}
\gvac{1} \got{1}{F} \gvac{1} \got{2}{F} \got{1}{F}\gnl
\gwcm{3} \gcmu \gcl{4} \gnl
\gcl{1} \glcm \gcl{1} \gcl{2} \gnl
\gcl{1} \gcl{1} \gbr \gnl
\gcl{1} \glm \gmu \gnl
\gwmu{3} \hspace{-0,22cm} \glcm \gcn{1}{1}{2}{1} \gnl
\gvac{2} \gcn{1}{1}{0}{1} \gcl{1} \gbr \gnl
\gvac{2} \gcl{1} \glm \gcl{2} \gnl
\gvac{2} \gwmu{3} \gnl
\gvac{3} \gob{1}{F} \gob{3}{F} 
\gend}\stackrel{ass. F}{\stackrel{\equref{comod alg}}{=}}
\scalebox{0.84}[0.84]{
\gbeg{8}{13}
\gvac{1} \got{1}{F} \gvac{2} \got{1}{F} \got{3}{F}\gnl
\gwcm{3} \gwcm{3} \gcl{7} \gnl
\gcl{1} \glcm \gcl{1} \gvac{1} \gcl{2} \gnl
\gcl{1} \gcl{1} \gbr \gnl
\gcl{6} \glm \gcn{1}{1}{1}{2} \gcn{1}{1}{3}{4} \gnl
\gvac{2} \gcl{5} \hspace{-0,22cm} \glcm \glcm \gnl
\gvac{3} \gcl{1} \gbr \gcl{1} \gnl
\gvac{3} \gmu \gmu \gnl
\gvac{4} \gcn{1}{1}{0}{2} \gvac{1} \hspace{-0,34cm} \gbr \gnl
\gvac{5} \glm \gcl{3} \gnl
\gcn{1}{1}{3}{4} \gvac{2} \gwmu{4} \gnl
\gvac{2} \hspace{-0,34cm} \gwmu{4} \gnl
\gvac{3} \gob{2}{F} \gob{6}{F} 
\gend}$$

$$\stackrel{comod.}{\stackrel{mod.}{=}}
\scalebox{0.84}[0.84]{
\gbeg{8}{13}
\gvac{1} \got{1}{F} \gvac{3} \got{1}{F} \got{3}{F}\gnl
\gwcm{4} \gwcm{3} \gcl{4} \gnl
\gcl{1} \gvac{1} \glcm \gcl{1} \glcm \gnl
\gcl{1} \gcn{1}{1}{3}{2} \gvac{1} \gbr \gcl{1} \gcl{2} \gnl
\gcl{7} \gcmu \gcl{1} \gbr \gnl
\gvac{1} \gcl{1} \gbr \gcl{1} \gmu \gcn{1}{1}{1}{0} \gnl
\gvac{1} \glm \gcl{1} \gcn{1}{1}{1}{2} \gvac{1} \hspace{-0,34cm} \gbr \gnl
\gvac{2} \gcn{1}{2}{2}{2} \gcn{1}{2}{2}{5} \gvac{1} \glm \gcl{5} \gnl
\gvac{6} \gcl{1} \gnl
\gvac{2} \gcn{1}{1}{2}{3} \gvac{2} \glm \gcl{3} \gnl
\gvac{3} \gwmu{4} \gnl
\gvac{1} \hspace{-0,34cm} \gwmu{5} \gnl
\gvac{3} \gob{2}{F} \gob{6}{F} 
\gend}
\stackrel{nat.}{\stackrel{\equref{mod coalg}}{=}}
\scalebox{0.84}[0.84]{
\gbeg{8}{10}
\gvac{1} \got{1}{F} \gvac{2} \got{1}{F} \got{3}{F}\gnl
\gwcm{3} \gwcm{3} \gcl{2} \gnl
\gcl{1} \glcm \gcl{1} \glcm \gnl
\gcl{1} \gcl{3} \gbr \gcl{1} \gbr \gnl
\gcl{4} \gcl{1} \gcl{2} \gcn{1}{1}{1}{3} \glm \gcl{1} \gnl
\gvac{4} \gbr \gcl{1} \gnl
\gvac{1} \gcn{1}{1}{1}{3} \gwmu{3} \gmu \gnl
\gvac{2} \glm \gcn{1}{2}{4}{4} \gnl
\gwmu{4} \gnl
\gvac{1} \gob{2}{F} \gob{6}{F} 
\gend}
\stackrel{nat.}{=}
\scalebox{0.84}[0.84]{
\gbeg{7}{12}
\gvac{1} \got{1}{F} \gvac{1} \got{1}{F} \got{3}{F}\gnl
\gvac{1} \gcl{5} \gwcm{3} \gcl{2} \gnl
\gvac{2} \gcl{1} \glcm \gnl
\gvac{2} \gcl{1} \gcl{1} \gbr \gnl
\gvac{2} \gcl{1} \glm \gcl{6} \gnl
\gvac{2} \gwmu{3} \gnl
\gwcm{3} \gcl{2} \gnl
\gcl{1} \glcm \gnl
\gcl{1} \gcl{1} \gbr \gnl
\gcl{1} \glm \gcl{1} \gnl
\gwmu{3} \gwmu{3} \gnl
\gvac{1} \gob{1}{F} \gob{5}{F} 
\gend}=
\gbeg{4}{5}
\got{1}{F} \got{1}{F} \got{1}{F} \gnl
\gcl{1} \glmptb \gnot{\hspace{-0,34cm}\lambda} \grmptb \gnl
\glmptb \gnot{\hspace{-0,34cm}\lambda} \grmptb \gcl{1} \gnl
\gcl{1} \gmu \gnl
\gob{1}{F} \gob{2}{F} 
\gend
$$
We have used that the pre-multiplication of $F$ is associative. By duality then, also the pre-comultiplication of $F$ should be coassoaciative. This is why we 
assumed that the 2-cells $\mu_M, \mu_C, \Delta_M, \Delta_C$ are canonical. 
\qed\end{proof}

For the rest of the $F$-(co)module structures on $B$, with the above $\psi$ and $\phi$ and by \equref{right F-mod} -- \equref{left F-mod} we obtain: 
$$
\gbeg{2}{3}
\got{1}{B} \got{1}{F} \gnl
\grmo \gcl{1} \gnl 
\gob{1}{B}
\gend=
\gbeg{3}{6}
\got{2}{B} \got{1}{F} \gnl
\gcmu \gcl{1} \gnl
\gcl{1} \gbr \gnl
\glm \gcl{2} \gnl
\gvac{1} \gcu{1} \gnl
\gob{1}{} \gob{1}{} \gob{1}{B} 
\gend\stackrel{\equref{mod coalg counity}}{=}
\gbeg{2}{4}
\got{1}{B} \got{1}{F} \gnl
\gcl{1} \gcl{1} \gnl 
\gcl{1} \gcu{1} \gnl 
\gob{1}{F}
\gend
\hspace{2cm}
\gbeg{2}{3}
\got{1}{B} \gnl
\gcl{1} \hspace{-0,42cm} \glmf \gnl  
\gvac{1} \gob{1}{B} \gob{1}{F}
\gend=
\gbeg{3}{6}
\got{1}{} \got{1}{} \got{1}{B} \gnl
\gvac{1} \guf{1} \gcl{2} \gnl
\glcm \gnl
\gcl{1} \gbr \gnl
\gmu \gcl{1} \gnl
\gob{2}{B} \gob{1}{F} 
\gend\stackrel{\equref{comod alg unity}}{=}
\gbeg{2}{4}
\got{1}{B} \gnl
\gcl{1} \guf{1} \gnl 
\gcl{1} \gcl{1} \gnl 
\gob{1}{B} \gob{1}{F} 
\gend
$$

$$
\gbeg{2}{3}
\got{3}{B} \gnl
\grmo \gvac{1} \gcl{1} \gnl
\gob{1}{F} \gob{1}{B}
\gend=
\gbeg{3}{5}
\got{2}{B} \got{1}{} \gnl
\gcmu \gu{1} \gnl
\gcl{1} \gbr \gnl
\glm \gcl{1} \gnl
\gob{1}{} \gob{1}{F} \gob{1}{B} 
\gend\stackrel{\equref{mod alg unity}}{=}
\gbeg{2}{4}
\got{3}{B} \gnl
\gu{1} \gcl{1} \gnl
\gcl{1} \gcl{1} \gnl
\gob{1}{F} \gob{1}{B}
\gend
\hspace{2cm}
\gbeg{2}{3}
\got{1}{F} \got{1}{B} \gnl
\glmf \gcn{1}{1}{-1}{-1} \gnl
\gob{3}{B}
\gend=
\gbeg{3}{5}
\got{1}{} \got{1}{F} \got{1}{B} \gnl
\glcm \gcl{1} \gnl
\gcl{1} \gbr \gnl
\gmu \gcuf{1} \gnl
\gob{2}{B} \gob{1}{} 
\gend\stackrel{\equref{comod coalg counity}}{=}
\gbeg{2}{4}
\got{1}{F} \got{1}{B} \gnl
\gcl{1} \gcl{1} \gnl 
\gcuf{1} \gcl{1} \gnl 
\gob{1}{} \gob{1}{B} 
\gend
$$
meaning that $B$ is a trivial right $F$-module, right $F$-comodule and a trivial left $F$-comodule and left $F$-module. (As a matter of fact, the right hand-side identities  
above are dual versions of the left hand-side ones, so it is enough to prove the latter ones, the former follow by duality.)


\underline{When $\mu_M, \mu_C, \Delta_M, \Delta_C$ are canonical.} As we saw in \prref{can implies assoc}, in this case 
by relations \equref{weak assoc. mu_M} -- \equref{weak counit Epsilon_C} the pre-(co)multiplications 
$
\gbeg{2}{1}
\gmu \gnl
\gend,
\gbeg{2}{1}
\gmuf{1} \gnl
\gend, 
\gbeg{2}{1}
\gcmu \gnl
\gend, 
\gbeg{2}{1}
\gcmuf{1} \gnl
\gend
$ 
are (co)associative. Now, with $\psi$ and $\phi$ defined as above, the relations \equref{mod alg} -- \equref{comod coalg} yield that $F$ is: 
\begin{itemize}
\item a left $B$-module algebra;
\item a left $B$-module coalgebra;
\item a left $B$-comodule algebra and
\item a left $B$-comodule coalgebra
\end{itemize}
in the usual sense. 

By \prref{can implies assoc}, part 5, relations \equref{1 lambda_M} -- \equref{8 lambda_M} yield that $F$ is a bimonad in $\hat\C$. 
Observe that in view of \equref{lambda 8 Radford} we have that $F$ is a bialgebra in the braided monoidal category ${}_B ^B\YD(\C)$. 

\medskip

The identities \equref{mod alg} -- \equref{comod coalg counity} are trivially satisfied, so they do not bring any new information, because the (co)module structures involved are trivial, 
as we commented above, and because of the (co)unity identities from \equref{mod alg unity} -- \equref{comod coalg counity}. 
The rest of the identities, \equref{weak action} -- \equref{normalized 2-cycle ro}, neither bring new information because $\mu_M, \mu_C, \Delta_M, \Delta_C$ are canonical 
(hence $\sigma, \rho, \Phi_{\lambda}, \omega$ are trivial) and because of the trivial (co)module structures. 

Collecting the above data we see that we recovered the necessary and sufficient conditions for the Radford biproduct to be a bialgebra \cite{Rad1} but here in the setting of 
an arbitrary braided monoidal category $\C$. Indeed, we have: 

\begin{thm} \thlabel{Radford}
Given a braided monoidal category $\C$, a bialgebra $B$ and a left $B$-module algebra and $B$-comodule coalgebra $F$ in $\C$. The following are equivalent: 
\begin{enumerate}
\item the biproduct $F\times B$ is a bialgebra in $\C$; 
\item $F$ is a 
bialgebra in ${}_B ^B\YD(\C)$. 
\end{enumerate}
\end{thm}

(The proof of the above theorem is lengthy, we do not type it here because of the extent of the article.)

In the biproduct $F\times B$, the product is given by the smash product and the coproduct by the smash coproduct, the (co)unity is componentwise and all these structures come out from the 
wreath and the cowreath structures \equref{wreath (co)product}, namely: 
$$
\nabla_{FB}=
\gbeg{3}{6}
\got{1}{F} \got{2}{B} \got{1}{F} \got{1}{B}  \gnl
\gcl{1} \gcmu \gcl{1} \gcl{1} \gnl
\gcl{1} \gcl{1} \gbr \gcl{1} \gnl
\gcl{1} \glm \gmu \gnl
\gwmu{3}\gvac{1} \hspace{-0,22cm} \gcl{1} \gnl
\gob{4}{F} \gob{1}{B}
\gend\hspace{1,5cm}
\eta_{FB}=
\gbeg{2}{3}
\gu{1} \gu{1} \gnl
\gcl{1} \gcl{1} \gnl
\gob{1}{F} \gob{1}{B}
\gend; \hspace{1,5cm}
\Delta_{FB}=
\gbeg{3}{6}
\gvac{1} \got{1}{F} \got{4}{B} \gnl
\gwcm{3} \gcn{1}{1}{2}{2} \gnl
\gcl{1} \glcm \gcmu \gnl
\gcl{1} \gcl{1} \gbr \gcl{1} \gnl
\gcl{1} \gmu \gcl{1} \gcl{1} \gnl
\gob{1}{F} \gob{2}{B} \gob{1}{F} \gob{1}{B}  
\gend\hspace{1,5cm}
\Epsilon_{FB}=
\gbeg{2}{4}
\got{1}{F} \got{1}{B} \gnl
\gcl{1} \gcl{1} \gnl
\gcu{1} \gcu{1} \gnl
\gend
$$

\begin{ex}
For $\C=Vec$, the category of vector spaces over a field $k$, we recover \cite[Theorem 2.1]{Rad1}. 
The above example reveals that a biwreath encodes Radford biproduct: $F$ is a bialgebra in ${}_B ^B\YD(\C)$.
\end{ex}

\section{Biwreath-like objects and examples in $\K=\hat\C$} \selabel{bl objects}  

In all the examples of (mixed) (co)wreaths the 2-cells $\eta_M , \Epsilon_M, \eta_C, \Epsilon_C$ are canonical, at least as to the knowledge of the author. 
In what follows we are going to study some examples in which $B$ will have non-trivial $F$-(co)module structures. 
Then in view of \prref{trivial structures} some of the 2-cells $\eta_M , \Epsilon_M, \eta_C, \Epsilon_C$ should not appear in the site. When any of them appears 
we will consider it canonical. Henceforth, our next examples will partially have a structure of a biwreath.

\begin{defn} \delabel{bl object}
A biwreath-like object in $\K$ is a monad $(F, \mu_M, \eta_M)$ in $\EM^M(\K)$ and a comonad $(F, \Delta_C, \Epsilon_C)$ in $\EM^C(\K)$ over the same bimonad $B$ in $\K$ 
with the canonical restrictions: 
$$
\gbeg{2}{3}
\got{1}{F} \got{1}{F} \gnl
\gmu \gnl
\gob{2}{F} \gnl
\gend:=
\gbeg{2}{4}
\got{1}{F} \got{1}{F} \gnl
\glmptb \gnot{\hspace{-0,34cm}\mu_M} \grmptb \gnl
\gcl{1} \gcu{1} \gnl
\gob{1}{F} \gnl
\gend \hspace{2cm}
\gbeg{1}{4}
\got{1}{} \gnl
\gu{1} \gnl
\gcl{1} \gnl
\gob{1}{F} \gnl
\gend:=
\gbeg{1}{4}
\got{1}{} \gnl
\glmpb \gnot{\hspace{-0,34cm}\eta_M} \grmpb \gnl
\gcl{1} \gcu{1} \gnl
\gob{1}{F} \gnl
\gend \hspace{2cm}
\gbeg{2}{4}
\got{2}{F} \gnl
\gcn{1}{1}{2}{2} \gnl
\gcmuf{1} \gnl
\gob{1}{F} \gob{1}{F} \gnl
\gend:=
\gbeg{2}{4}
\got{1}{F} \gnl
\gcl{1} \gu{1} \gnl
\glmptb \gnot{\hspace{-0,34cm}\Delta_C} \grmptb \gnl
\gob{1}{F} \gob{1}{F} \gnl
\gend \hspace{2cm}
\gbeg{1}{3}
\got{1}{F} \gnl
\gcl{1} \gnl
\gcuf{1} \gnl
\gob{1}{} \gnl
\gend:=
\gbeg{1}{3}
\got{1}{F} \gnl
\gcl{1} \gu{1} \gnl
\glmpt \gnot{\hspace{-0,34cm}\Epsilon_C} \grmpt \gnl
\gob{1}{} \gnl
\gend
$$
equipped with a 2-cell $\lambda:FF\to FF$ in $\K$ so that: 
\begin{enumerate}[(a)]
\item $\lambda$ is a distributive law in the sense of \deref{bimonad} b); 
\item the following compatibility conditions are fulfilled: 
\begin{equation} \eqlabel{lambda mixed 1-3}
\gbeg{2}{3}
\got{1}{F} \got{1}{F} \gnl
\gcl{1} \gcl{1} \gnl
\gcuf{1} \gcuf{1} \gnl
\gend=
\gbeg{2}{3}
\got{1}{F} \got{1}{F} \gnl
\gmu \gnl
\gvac{1} \hspace{-0,22cm} \gcuf{1} \gnl
\gend \hspace{2cm}
\gbeg{1}{5}
\got{2}{}  \gnl
 \gu{1} \gnl
 \gcl{1} \gnl
\hspace{-0,34cm} \gcmuf{1} \gnl
\gob{1}{F} \gob{1}{F} \gnl
\gend=
\gbeg{2}{4}
\got{1}{} \gnl
\gu{1} \gu{1} \gnl
\gcl{1} \gcl{1} \gnl
\gob{1}{F} \gob{1}{F}
\gend  \hspace{2cm}
\gbeg{1}{2}
\gu{1} \gnl
\gcuf{1} \gnl
\gob{2}{} \gnl
\gend=
\Id_{id_{\A}}
\end{equation}
and
\begin{center} \hspace{-0,2cm}
\begin{tabular}{p{5.6cm}p{0cm}p{6cm}}
\begin{equation} \eqlabel{lambda_M 8}
\gbeg{2}{5}
\got{1}{F} \got{1}{F} \gnl
\gcl{1} \gcl{1} \gnl
\glmptb \gnot{\hspace{-0,34cm}\mu_M} \grmptb \gnl
\hspace{-0,22cm} \gcmuf{1} \gvac{2} \gcn{1}{1}{0}{1} \gnl
\gob{1}{F} \gob{1}{F} \gob{1}{B} \gnl
\gend=
\gbeg{3}{6}
\got{1}{F} \got{2}{F} \gnl
\gcl{1} \gcn{1}{1}{2}{2} \gnl
\gcl{1} \gcmuf{1} \gnl
\glmptb \gnot{\hspace{-0,34cm}\lambda} \grmptb \gcl{1} \gnl
\gcl{1} \glmptb \gnot{\hspace{-0,34cm}\mu_M} \grmptb \gnl
\gob{1}{F} \gob{1}{F} \gob{1}{B} \gnl
\gend
\end{equation} & & \vspace{0,1cm}
\begin{equation}\eqlabel{lambda_C 8}
\gbeg{3}{4}
\got{1}{F} \got{1}{F} \got{1}{B} \gnl
\gmu \gcn{1}{1}{1}{0} \gnl
\gvac{1} \hspace{-0,34cm} \glmptb \gnot{\hspace{-0,34cm}\Delta_C} \grmptb \gnl
\gvac{1} \gob{1}{F} \gob{1}{F} \gnl
\gend=
\gbeg{3}{5}
\got{1}{F} \got{1}{F} \got{1}{B} \gnl
\gcl{1} \glmptb \gnot{\hspace{-0,34cm}\Delta_C} \grmptb \gnl
\glmptb \gnot{\hspace{-0,34cm}\lambda} \grmptb \gcl{1} \gnl
\gcl{1} \gmu \gnl
\gob{1}{F} \gob{2}{F.} \gnl
\gend
\end{equation}
\end{tabular}
\end{center}
\end{enumerate}
\end{defn}

Applying $\Epsilon_B$ to \equref{lambda_M 8} (or applying $\eta_B$ to \equref{lambda_C 8}), one gets \equref{bimonad in K}, 
then $F$ is a bimonad in $\K$ with possibly non-(co)associative (co)multiplication. On the other hand, applying 
$F
\gbeg{1}{2}
\got{1}{F}\gnl
\gcuf{1} \gnl
\gend B$ to \equref{lambda_M 8} and 
$F
\gbeg{1}{2}
\gu{1} \gnl
\gob{1}{F}\gnl
\gend B$ to \equref{lambda_C 8}, similarly as in \equref{mu_m con sigma} and \equref{Delta_C con ro}, we obtain $\mu_M$ and $\Delta_C$ below:  
\begin{center} \hspace{-1,3cm} 
\begin{tabular}{p{4.5cm}p{0,7cm}p{5cm}p{2cm}p{5cm}}
\begin{equation} \eqlabel{bimonad in K}
\gbeg{2}{5}
\got{1}{F} \got{1}{F} \gnl
\gmu \gnl
\gcn{1}{1}{2}{2} \gnl
\gcmuf{1} \gnl
\gob{1}{F}\gob{1}{F} 
\gend=
\gbeg{3}{6}
\got{1}{F} \got{2}{F} \gnl
\gcl{1} \gcn{1}{1}{2}{2} \gnl
\gcl{1} \gcmuf{1} \gnl
\glmptb \gnot{\hspace{-0,34cm}\lambda} \grmptb  \gcl{1} \gnl
\gcl{1} \gmu \gnl
\gob{1}{F}  \gob{2}{F}
\gend
\end{equation} & & 
\begin{equation*}
\mu_M=
\gbeg{3}{6}
\got{1}{F} \got{2}{F} \gnl
\gcl{1} \gcn{1}{1}{2}{2} \gnl
\gcl{1} \gcmuf{1} \gnl
\glmptb \gnot{\hspace{-0,34cm}\lambda} \grmptb \gcl{1} \gnl
\gcl{1} \glmpt \gnot{\hspace{-0,34cm}\sigma} \grmptb \gnl
\gob{1}{F} \gob{3}{B} \gnl
\gend
\hspace{2cm}
\Delta_C=
\gbeg{3}{5}
\got{1}{F} \got{3}{B} \gnl
\gcl{1} \glmpb \gnot{\hspace{-0,34cm}\rho} \grmptb \gnl
\glmptb \gnot{\hspace{-0,34cm}\lambda} \grmptb \gcl{1} \gnl
\gcl{1} \gmu \gnl \gnl
\gob{1}{F} \gob{2}{F} \gnl
\gend
\end{equation*}
\end{tabular}
\end{center}
where
\begin{center} \hspace{1,2cm}
\begin{tabular}{p{4cm}p{1cm}p{4cm}}
\gbeg{2}{4}
\got{1}{F} \got{1}{F} \gnl
\glmpt \gnot{\hspace{-0,34cm}\sigma} \grmptb \gnl
\gvac{1} \gcl{1} \gnl
\gob{3}{B}
\gend:=
\gbeg{2}{4}
\got{1}{F} \got{1}{F} \gnl
\glmptb \gnot{\hspace{-0,34cm}\mu_M} \grmptb \gnl
\gcuf{1} \gcl{1} \gnl
\gob{3}{B}
\gend
& &
\gbeg{2}{4}
\got{3}{B} \gnl
\gvac{1} \gcl{1} \gnl
\glmpb \gnot{\hspace{-0,34cm}\rho} \grmptb \gnl
\gob{1}{F}\gob{1}{F}
\gend:=
\gbeg{2}{4}
\got{3}{B} \gnl
\gu{1} \gcl{1} \gnl
\glmptb \gnot{\hspace{-0,34cm}\Delta_C} \grmptb \gnl
\gob{1}{F}\gob{1}{F.}
\gend
\end{tabular}
\end{center}

\subsection{Case two: $\D_2=\YD(\C)_F ^F$ } \sslabel{ex psi_2}

We proceed now with $\K=\hat\C$ and we consider a biwreath-like object $F$ with $\lambda$ given by \equref{left lambda}.
Then $F$ is a bialgebra in $\C$ (with possibly non-(co)associative (co)multiplication). Then by the naturality 
of the braiding in $\C$ it is not difficult to see that $\lambda$ satisfies the distributive law conditions in the part a) of \deref{bl object}. 
Now $\mu_M$ and $\Delta_C$ from above become: 
\begin{center} \hspace{-2cm} 
\begin{tabular}{p{4.5cm}p{0,7cm}p{5cm}p{2cm}p{5cm}}
\begin{equation} \eqlabel{left lambda}
\lambda=
\gbeg{3}{6}
\got{2}{F} \got{1}{F} \gnl
\gcn{2}{1}{2}{2} \gcl{2} \gnl
\gcmuf{1}  \gnl  
\gcl{1} \gbr  \gnl
\gmu \gcl{1} \gnl
\gob{2}{F} \gob{1}{F}
\gend
\end{equation} & & 
\begin{equation*}
\mu_M=
\gbeg{3}{6}
\got{2}{F} \got{2}{F} \gnl
\gcn{2}{1}{2}{2} \gcn{2}{1}{2}{2} \gnl
\gcmuf{1} \gvac{2} \gcmuf{1} \gnl
\gcl{1} \gbr \gcl{1} \gnl
\gmu \glmpt \gnot{\hspace{-0,34cm}\sigma} \grmptb  \gnl
\gob{2}{F} \gob{3}{B} 
\gend
\hspace{2cm}
\Delta_C=
\gbeg{3}{6}
\got{2}{F} \got{3}{B} \gnl
\gcn{2}{1}{2}{2} \gvac{1} \gcl{1} \gnl
\gcmuf{1} \gvac{2} \glmpb \gnot{\hspace{-0,34cm}\rho} \grmptb  \gnl
\gcl{1} \gbr \gcl{1} \gnl
\gmu \gmu  \gnl
\gob{2}{F} \gob{2}{F} 
\gend
\end{equation*}
\end{tabular}
\end{center}

The category $\D_2=\YD(\C)_F ^F$ of right-right Yetter-Drinfel'd modules over 
a bialgebra $F$ in $\C$ has a pre-braiding 
$$d^2_{X,Y}=
\gbeg{3}{5}
\got{1}{X} \got{1}{Y} \gnl
\gcl{1} \grcm \gnl
\gbr \gcl{1} \gnl
\gcl{1} \grm \gnl
\gob{1}{Y} \gob{1}{X} 
\gend
$$
for any $X,Y\in\YD(\C)_F ^F$. 
In our setting, where $F$ is a biwreath-like object, it is a bialgebra in $\C$ with possibly non-(co)associative (co)multiplication and it (co)acts from the right on 
$B$ in some broader sense, which we will analyze below. We set $\psi=d^2_{B,F}$ and $\phi=d^2_{F,B}$ to obtain: 
\begin{center} 
\begin{tabular}{p{4.2cm}p{2cm}p{4.2cm}}
\begin{equation} \eqlabel{psi_2}
\psi=
\gbeg{3}{6}
\got{1}{B} \got{2}{F} \gnl
\gcl{1} \gcn{1}{1}{2}{2} \gnl 
\gcl{1} \gcmuf{1} \gnl
\gbr \gcl{1} \gnl
\gcl{1} \grmo \gcl{1} \gnl
\gob{1}{F} \gob{1}{B} 
\gend
\end{equation}  &  &  
\begin{equation} \eqlabel{phi_2}
\phi= 
\gbeg{3}{6}
\got{1}{F} \got{1}{B} \gnl
\gcl{1} \gcl{1}  \gnl
\gcl{1} \gcl{1} \hspace{-0,42cm} \glmf \gnl
\gvac{1} \gbr \gcl{1} \gnl
\gvac{1} \gcl{1} \gmu \gnl
\gvac{1} \gob{1}{B} \gob{2}{F} 
\gend
\end{equation}
\end{tabular} 
\end{center} 
then similarly as in \equref{right F-mod} and \equref{right F-comod} we have: 
$$\gbeg{2}{3}
\got{1}{B} \got{1}{F} \gnl
\grmo \gcl{1} \gnl 
\gob{1}{B}
\gend=
\gbeg{2}{4}
\got{1}{B} \got{1}{F} \gnl
\glmptb \gnot{\hspace{-0,34cm}\psi} \grmptb \gnl
\gcuf{1} \gcl{1} \gnl
\gob{3}{B} 
\gend
\hspace{3cm} 
\gbeg{2}{3}
\got{1}{B} \gnl
\gcl{1} \hspace{-0,42cm} \glmf \gnl  
\gvac{1} \gob{1}{B} \gob{1}{F}
\gend=
\gbeg{2}{4}
\got{3}{B} \gnl
\gu{1} \gcl{1} \gnl
\glmptb \gnot{\hspace{-0,34cm}\phi} \grmptb \gnl
\gob{1}{B} \gob{1}{F.} 
\gend
$$
The definitions \equref{left F-mod} and \equref{left F-comod} make sense here and by \equref{lambda mixed 1-3}, \equref{weak action unity} and \equref{weak coaction counity} 
we have that $B$ is a trivial left $F$-(co)module (the same we obtain by \prref{trivial structures}, 1) and 4), as we consider $\eta_M$ and $\Epsilon_C$ canonical). 

We are going to study the structure of the biwreath-like object on $F$ with the above $\psi, \phi$ and $\lambda$ applying the same arguments as in \ssref{ex psi_1}, but this time 
considering only those axioms of a biwreath which involve the wreath and the cowreath structures. 
We check that the axioms \equref{psi laws for B} and \equref{phi laws for B} are fulfilled. For the same reasons as before we prove only the axiom with the multiplication: 
$$
\gbeg{3}{5}
\got{1}{B}\got{1}{B}\got{1}{F}\gnl
\gmu \gcn{1}{1}{1}{0} \gnl
\gvac{1} \hspace{-0,34cm} \glmptb \gnot{\hspace{-0,34cm}\psi} \grmptb  \gnl
\gvac{1} \gcl{1} \gcl{1} \gnl
\gvac{1} \gob{1}{F} \gob{1}{B}
\gend=
\gbeg{4}{6}
\got{1}{B} \got{1}{B} \got{1}{F} \gnl
\gmu \gcl{1} \gnl
\gvac{1} \hspace{-0,32cm} \gcl{1} \gcmuf{1} \gnl
\gvac{1} \gbr \gcl{1} \gnl
\gvac{1} \gcl{1} \grmo \gcl{1} \gnl
\gvac{1} \gob{1}{\s F} \gob{1}{B} 
\gend\stackrel{nat.}{=}
\gbeg{4}{8}
\got{1}{B} \got{1}{B} \got{2}{F} \gnl
\gcl{3} \gcl{2} \gcn{1}{1}{2}{2} \gnl
\gvac{2} \gcmuf{1} \gnl
\gvac{1} \gbr \gcl{2} \gnl
\gbr \gcl{1} \gnl
\gcl{2} \gmu \gcn{1}{1}{1}{0} \gnl
\gvac{2} \hspace{-0,34cm} \grmo \gcl{1} \gnl
\gob{1}{\hspace{0,34cm} F} \gob{3}{B} 
\gend\stackrel{\equref{F mod alg}}{=}
\gbeg{5}{9}
\got{1}{B} \got{1}{B} \got{2}{F} \gnl
\gcl{3} \gcl{2} \gcn{1}{1}{2}{2} \gnl
\gvac{2} \gcmuf{1} \gnl
\gvac{1} \gbr \gcn{1}{1}{1}{2} \gnl
\gbr \gcl{1} \gcmuf{1} \gnl
\gcl{2} \gcl{1} \gbr \gcl{1} \gnl
\gvac{1} \grmo \gcl{1} \gvac{1} \grmo \gcl{1} \gnl
\gcl{1} \gwmu{3} \gnl
\gob{1}{F} \gob{3}{B} 
\gend\stackrel{coass. F}{\stackrel{nat.}{=}}
\gbeg{5}{10}
\got{1}{B} \gvac{1} \got{1}{B} \got{2}{F} \gnl
\gcl{1} \gvac{1} \gcl{1} \gcn{1}{1}{2}{2} \gnl
\gcl{1} \gvac{1} \gcl{1} \gcmuf{1} \gnl
\gcl{1} \gvac{1} \gbr \gcl{1} \gnl
\gcl{1} \gvac{1} \gcn{1}{1}{1}{0} \grmo \gcl{1} \gnl
\gcl{1} \gcmuf{1} \gvac{2} \gcl{3} \gnl
\gbr \gcl{1} \gnl
\gcl{1} \grmo \gcl{1} \gnl
\gcl{1} \gwmu{3} \gnl
\gob{1}{F} \gob{3}{B} 
\gend=
\gbeg{3}{5}
\got{1}{B}\got{1}{B}\got{1}{F}\gnl
\gcl{1} \glmpt \gnot{\hspace{-0,34cm}\psi} \grmptb \gnl
\glmptb \gnot{\hspace{-0,34cm}\psi} \grmptb \gcl{1} \gnl
\gcl{1} \gmu \gnl
\gob{1}{F} \gob{2}{B.}
\gend
$$
Since we used that 
$\gbeg{2}{2}
\gcmuf{1} \gnl
\gend$ of $F$ is coassociative, we need to consider $F$ a trivial left $B$-comodule by \equref{quasi coass. Delta_M}. 
Dually, for the compatibility of $\phi$ with the comultiplication of $B$ we need 
$\gbeg{2}{1}
\gmu \gnl
\gend$ associative and accordingly $F$ to be a trivial left $B$-module, by \equref{weak assoc. mu_M}. 
Observe that again we have used the intrinsic properties of a biwreath-like object and naturality of the braiding in $\C$ (with respect to the multiplication of $B$). 
Then we obtain that $F$ is a usual bialgebra in $\C$.

With $\psi$ and $\phi$ defined as above, because of the trivial left $B$-(co)module structures on $F$ and left $F$-(co)module structures on $B$, the relations 
\equref{mod alg}-\equref{mod alg unity}, \equref{comod coalg}-\equref{comod coalg counity} and \equref{F comod alg} -- \equref{F mod coalg counit} hold trivially 
and they do not provide any new information.

Though, by \equref{F mod alg} -- \equref{F mod alg unit} and \equref{F comod coalg} -- \equref{F comod coalg counit} $B$ is a right $F$-module algebra and 
a right $F$-comodule coalgebra in $\C$ (in the broader sense, without knowing if its right $F$-(co)module actions are proper). This was called ``$F$ measures $B$'' 
in \cite[Section 4.1]{Mont}, in the category of modules over a commutative ring. 
As a matter of fact, due to 
\equref{weak action} -- \equref{weak action unity}, \equref{weak coaction} -- \equref{normalized 2-cocycle} and \equref{2-cycle ro} -- \equref{normalized 2-cycle ro}, 
we have that there is a right {\em twisted action} of $F$ on $B$ by a normalized 2-cocycle $\sigma$ (and dually: 
there is a right {\em twisted coaction} of $F$ on $B$ by a normalized 2-cycle $\rho$):  
\begin{equation*} 
\gbeg{4}{11}
\got{1}{B} \got{2}{F} \got{1}{F} \gnl
\gcl{1} \gcn{2}{1}{2}{2} \gcl{3} \gnl
\gcl{1} \gcmuf{1} \gnl
\gbr \gcl{1} \gnl
\gcl{1} \grmo \gcl{1} \gcn{1}{1}{3}{2} \gnl
\gcl{1} \gcl{1} \gcmuf{1} \gnl
\gcl{1} \gbr \gcl{1} \gnl
\gcl{1} \gcl{1} \grmo \gcl{1} \gnl
\glmpt \gnot{\hspace{-0,34cm}\sigma} \grmptb \gcl{1} \gnl
\gvac{1} \gmu \gnl
\gob{4}{B}
\gend= 
\gbeg{3}{8}
\got{1}{B} \got{2}{F} \got{2}{F}\gnl
\gcl{1} \gcn{2}{1}{2}{2} \gcn{1}{1}{2}{2} \gnl
\gcl{1} \gcmuf{1} \gvac{2} \gcmuf{1}  \gnl
\gcl{1} \gcl{1} \gbr \gcl{1} \gnl
\gcn{1}{1}{1}{2} \gmu \glmpt \gnot{\hspace{-0,34cm}\sigma} \grmpt  \gnl
\gvac{1} \hspace{-0,22cm} \grmo \gcl{1} \gvac{2} \gcl{1} \gnl
\gvac{1} \gwmu{4} \gnl
\gob{6}{B}
\gend
\hspace{4cm}
\gbeg{2}{5}
\got{1}{B}  \gnl
\gcl{1} \gu{1}  \gnl
\grmo \gcl{1} \gnl
\gcl{1} \gnl
\gob{1}{B}
\gend=
\gbeg{3}{4}
\got{1}{B}  \gnl
\gcl{2} \gnl
\gob{1}{B}
\gend
\end{equation*}
$$\qquad\textnormal{twisted action} \hspace{4cm} \textnormal{twisted action unity} \vspace{-0,2cm} $$ 
$$\textnormal{ from \equref{weak action} and \equref{weak action unity}} $$

\begin{equation} \eqlabel{2-cocycle condition psi_1 caso 2}
\gbeg{6}{8}
\got{1}{F} \got{2}{F} \got{2}{F}\gnl
\gcl{1} \gcn{2}{1}{2}{2} \gcn{1}{1}{2}{2} \gnl
\gcl{1} \gcmuf{1} \gvac{2} \gcmuf{1} \gnl
\gcl{1} \gcl{1} \gbr \gcl{1} \gnl
\gcn{1}{1}{1}{2} \gmu \glmpt \gnot{\hspace{-0,34cm}\sigma} \grmptb  \gnl
\gvac{1} \hspace{-0,22cm} \glmpt \gnot{\hspace{-0,34cm}\sigma} \grmptb \gvac{1} \gcn{1}{1}{2}{1} \gnl
\gvac{2} \gwmu{3} \gnl
\gob{7}{B}
\gend=
\gbeg{6}{9}
\got{2}{F} \got{2}{F} \got{2}{F} \gnl
\gcn{2}{1}{2}{2} \gcn{2}{1}{2}{2} \gcn{1}{3}{2}{2} \gnl
\gcmuf{1} \gvac{2} \gcmuf{1} \gnl
\gcl{1} \gbr \gcl{1} \gnl
\gmu \glmpt \gnot{\hspace{-0,34cm}\sigma} \grmptb \gcmuf{1} \gnl
\gcn{1}{1}{2}{3} \gvac{2} \gbr \gcl{1} \gnl
\gvac{1} \glmpt \gnot{\sigma} \gcmpb \grmpt \grmo \gcl{1} \gnl
\gvac{2} \gwmu{3} \gnl
\gob{7}{B}
\gend
\hspace{2cm}
\gbeg{2}{5}
\got{3}{F}\gnl
\gu{1} \gcl{1} \gnl
\glmpt \gnot{\hspace{-0,34cm}\sigma} \grmptb \gnl
\gvac{1} \gcl{1} \gnl
\gob{3}{B}
\gend=
\gbeg{1}{4}
\got{1}{F}\gnl
\gcuf{1} \gnl
\gu{1} \gnl
\gob{1}{B}
\gend=
\gbeg{3}{5}
\got{1}{F}\gnl
\gcl{1} \gu{1} \gnl
\glmpt \gnot{\hspace{-0,34cm}\sigma} \grmptb \gnl
\gvac{1} \gcl{1} \gnl
\gob{3}{B}
\gend
\end{equation}
$$\textnormal{normalized Sweedler 2-cocycle} \vspace{-0,2cm} $$ 
$$\textnormal{ from \equref{2-cocycle condition} and \equref{normalized 2-cocycle}} $$

\begin{equation*} 
\gbeg{5}{7}
\got{4}{B} \gnl
\gwcm{4} \gnl 
\hspace{-0,12cm} \gcl{1} \hspace{-0,42cm} \glmf \gvac{1} \gcl{1} \gnl
\gcn{1}{1}{3}{2} \gvac{1} \hspace{-0,32cm} \gcmuf{1} \gvac{2} \glmpb \gnot{\hspace{-0,34cm}\rho} \grmpb  \gnl
\gvac{1} \gcl{1} \gcl{1} \gbr \gcl{1} \gnl
\gvac{1} \gcl{1} \gmu \gmu \gnl
\gvac{1} \gob{1}{B} \gob{2}{F} \gob{2}{F}\gnl
\gend
= 
\gbeg{5}{10}
\got{3}{B} \gnl
\gwcm{3} \gnl 
\glmptb \gnot{\hspace{-0,34cm}\rho} \grmpb \gcl{1} \gnl
\gcl{1} \gcl{1} \gcl{1} \hspace{-0,42cm} \glmf \gnl
\gvac{1} \gcl{1} \gbr \gcl{1} \gnl
\gvac{1} \gcl{1} \gcl{1} \gmu \gnl 
\gvac{1} \gcl{1} \gcl{1} \hspace{-0,42cm} \glmf \gcn{1}{1}{0}{1} \gnl
\gvac{2} \gbr \gcl{1} \gcl{2} \gnl
\gvac{2} \gcl{1} \gmu \gnl
\gvac{2} \gob{1}{B} \gob{2}{F} \gob{1}{F}\gnl
\gend \hspace{3cm}
\gbeg{2}{4}
\got{1}{B}  \gnl
\gcl{1} \hspace{-0,42cm} \glmf \gnl
\gvac{1} \gcl{1} \gcuf{1} \gnl
\gob{3}{B}
\gend=
\gbeg{1}{4}
\got{1}{B}  \gnl
\gcl{2} \gnl
\gob{1}{B}
\gend
\end{equation*}
$$\qquad\qquad\textnormal{twisted coaction} \hspace{4cm} \textnormal{twisted coaction counity} \vspace{-0,2cm} $$ 
$$\textnormal{ from \equref{weak coaction} and \equref{weak coaction counity}} $$

\begin{equation*} 
\gbeg{5}{8}
\got{5}{B} \gnl
\gvac{1} \gwcm{3} \gnl
\glmpb \gnot{\rho} \gcmpt \grmpb \gcl{1} \hspace{-0,42cm} \glmf \gnl
\gcn{1}{1}{3}{2} \gvac{2} \gbr \gcl{1} \gnl
\gcmuf{1} \gvac{2} \glmpb \gnot{\hspace{-0,34cm}\rho} \grmptb \gmu \gnl
\gcl{1} \gbr \gcl{1} \gcn{1}{2}{2}{2} \gnl
\gmu \gmu \gnl
\gob{2}{F} \gob{2}{F} \gob{2}{F} 
\gend=
\gbeg{5}{7}
\got{5}{B} \gnl
\gvac{1} \gwcm{3} \gnl
\glmpb \gnot{\hspace{-0,34cm}\rho} \grmpb \gvac{1} \gcn{1}{1}{1}{2} \gnl
\gcn{1}{1}{1}{0} \hspace{-0,22cm} \gcmuf{1} \gvac{2} \glmpb \gnot{\hspace{-0,34cm}\rho} \grmptb \gnl
\gcl{1} \gcl{1} \gbr \gcl{1} \gnl
\gcl{1} \gmu \gmu \gnl
\gob{1}{F} \gob{2}{F} \gob{2}{F} 
\gend \hspace{3cm}
\gbeg{2}{5}
\got{1}{B}  \gnl
\gcl{1} \gnl
\glmptb \gnot{\hspace{-0,34cm}\rho} \grmpb \gnl
\gcuf{1} \gcl{1} \gnl
\gob{3}{F}
\gend=
\gbeg{1}{4}
\got{1}{B}\gnl
\gcu{1} \gnl
\gu{1} \gnl
\gob{1}{F}
\gend=
\gbeg{3}{5}
\got{1}{B}\gnl
\gcl{1} \gnl
\glmptb \gnot{\hspace{-0,34cm}\rho} \grmpb \gnl
\gcl{1} \gcuf{1} \gnl
\gob{1}{F}
\gend
\end{equation*}
$$\textnormal{normalized 2-cycle} \vspace{-0,2cm} $$ 
$$\textnormal{ from \equref{2-cycle ro} and \equref{normalized 2-cycle ro}} $$
In the biproduct $F\times B$ the product is given by the wreath product and the coproduct by the cowreath coproduct \equref{wreath (co)product}, which in this case turn out 
to be Sweedler's crossed product and its dual construction: the crossed coproduct, the (co)unity is componentwise:
$$
\nabla_{FB}=
\gbeg{6}{10}
\gvac{1} \got{1}{F} \gvac{1} \got{1}{B} \got{2}{F} \got{1}{B}  \gnl
\gvac{1} \gcl{3} \gvac{1} \gcl{1} \gcn{2}{1}{2}{2} \gcl{4} \gnl
\gvac{3} \gcl{1} \gcmuf{1} \gnl
\gvac{3} \gbr \gcl{1} \gnl
\gvac{1} \gcn{1}{1}{1}{0} \gvac{1} \gcn{1}{1}{1}{0} \grmo \gcl{1} \gnl
\gcmuf{1} \gvac{2} \gcmuf{1} \gvac{2} \gwmu{3} \gnl
\gcl{1} \gbr \gcl{1} \gvac{1} \gcl{2} \gnl
\gmu \glmpt \gnot{\hspace{-0,34cm}\sigma} \grmptb  \gnl
\gcn{1}{1}{2}{2} \gvac{2} \gwmu{3} \gnl
\gob{2}{F} \gob{5}{B}
\gend\hspace{1cm}
\eta_{FB}=
\gbeg{2}{3}
\gu{1} \gu{1} \gnl
\gcl{1} \gcl{1} \gnl
\gob{1}{F} \gob{1}{B}
\gend; \hspace{1,5cm}
\Delta_{FB}=
\gbeg{6}{9}
\got{2}{F} \got{5}{B} \gnl
\gcn{1}{1}{2}{2} \gvac{2} \gwcm{3} \gnl
\gcmuf{1} \gvac{2} \glmpb \gnot{\hspace{-0,34cm}\rho} \grmptb \gvac{1} \gcl{2}  \gnl
\gcl{1} \gbr \gcl{1} \gvac{1} \gnl
\gmu \gmu \gwcm{3} \gnl
\gvac{1} \gcn{1}{1}{0}{1} \gvac{1} \gcn{1}{1}{0}{1} \gcl{1} \hspace{-0,42cm} \glmf \gcl{3} \gnl
\gvac{2} \gcl{1} \gvac{1} \gbr \gcl{1} \gnl
\gvac{2} \gcl{1} \gvac{1} \gcl{1} \gmu \gnl
\gvac{2} \gob{1}{F} \gvac{1} \gob{1}{B} \gob{2}{F} \gob{1}{B}  
\gend\hspace{1cm}
\Epsilon_{FB}=
\gbeg{2}{4}
\got{1}{F} \got{1}{B} \gnl
\gcl{1} \gcl{1} \gnl
\gcuf{1} \gcu{1} \gnl
\gend
$$

Observe that in our setting the 2-cocycle $\sigma$ is not invertible. 

\begin{ex} \exlabel{Sw}
For $\C=\R\x\Mod$, the category of modules over a commutative ring $R$, we recover the definition of a twisted $F$-action on $B$, ``$F$ measures $B$'' and of a (non-invertible!) 
normalized 2-cocycle $\sigma: F\ot F\to B$ from \cite[Definition 7.1.1, Lemma 7.1.2]{Mont}. These concepts were introduced in \cite{DoiTak, BCM}, generalizing \cite{Sw1}. 
Namely, if $\sigma$ is invertible, multiply the above formula for twisted action with $\sigma^{-1}$ from the left in the convolution algebra to obtain the usual formula. We also 
obtain the dual notions. In this sense, our construction generalizes the above notions to non-invertible 2-cocycles and to the setting of any braided monoidal category. 
\end{ex}

\subsection{Particular case: when $B=I$, or $\D=\C$} \sslabel{B=I}

If we set $B=I$ in the definition of a biwreath-like object, rather than obtaining the results of the previous subsection with $B=I$, that is, Sweedler's 
crossed product in $\C$ with a 2-cocycle $\sigma$ ``with trivial coefficients'', the 2-cocycle would actually turn out to be trivial. Namely, 
since $\eta_I=\Epsilon_I=id_I$, from the definition of 
$\gbeg{2}{1}
\gmu \gnl
\gend
$ it would follow: $\mu_M=
\gbeg{2}{1}
\gmu \gnl
\gend
$ and the 2-cocycle $\sigma$ would be trivial. For this reason, and because we can not defer $\mu_M$ from 
$\gbeg{2}{1}
\gmu \gnl
\gend
$ as we did before, so to deduce an expression for the former out of the identity \equref{lambda_M 8}, we are motivated to introduce the following:

\begin{defn} 
A {\em biwreath-like hybrid} is a 1-cell $F:\A\to\A$ in $\K$ equipped with the following structures: 
\begin{enumerate}
\item $(F, \gbeg{2}{1}
\gmu \gnl 
\gend, 
\gbeg{1}{1}
\gu{1} \gnl 
\gend, 
\gbeg{2}{1}
\gcmuf{1} \gnl 
\gend, 
\gbeg{1}{1}
\gcuf{1} \gnl 
\gend, 
\gbeg{2}{1}
\glmptb \gnot{\hspace{-0,34cm}\lambda} \grmptb \gnl 
\gend\hspace{0,12cm})
$ is a left bimonad in $\K$; 
\item $(F, \mu_M, \eta_M)$  is a monad in $\EM^M(\K)$ and $(F, \Delta_C, \Epsilon_C)$ is a comonad in $\EM^C(\K)$ both over the trivial bimonad $I$ in $\K$; 
\item the following compatibility conditions hold: 
$$\gbeg{1}{2}
\gbmp{\s\eta_M} \gnl
\gcl{1} \gnl
\gend=
\gbeg{1}{2}
\gu{1} \gnl
\gcl{1} \gnl
\gend, \qquad \qquad
\gbeg{1}{2}
\gcl{1} \gnl
\gbmp{\s\Epsilon_C} \gnl
\gend=
\gbeg{1}{2}
\gcl{1} \gnl
\gcuf{1} \gnl
\gend
$$ 

\begin{equation}\eqlabel{Majid lambda_M 8}
\gbeg{2}{4}
\got{1}{F} \got{1}{F} \gnl
\glmptb \gnot{\hspace{-0,34cm}\mu_M} \grmpt \gnl
\glmptb \gnot{\hspace{-0,34cm}\Delta_C} \grmpb \gnl
\gob{1}{F} \gob{1}{F} \gnl
\gend=
\gbeg{5}{5}
\got{1}{F} \got{1}{F} \gnl
\gcl{1} \glmptb \gnot{\hspace{-0,34cm}\Delta_C} \grmpb \gnl
\glmptb \gnot{\hspace{-0,34cm}\lambda} \grmptb \gcl{1} \gnl
\gcl{1} \glmptb \gnot{\hspace{-0,34cm}\mu_M} \grmpt \gnl
\gob{1}{F} \gob{1}{F.} \gnl
\gend 
\end{equation}
\end{enumerate}
\end{defn}

Observe that in the point 2) in the above definition we may take: $\psi=\phi=\id_F$, then by the identities \equref{monad law mu_M}-\equref{monad law eta_M} and 
\equref{comonad law Delta_C}-\equref{comonad law Epsilon_C} $F$ has associative and unital multiplication 
$
\gbeg{2}{1}
\glmptb \gnot{\hspace{-0,34cm}\mu_M} \grmpt \gnl
\gend$ with 
$\gbeg{1}{2}
\gbmp{\s\eta_M} \gnl
\gcl{1} \gnl
\gend$ and coassociative and counital comultiplication 
$
\gbeg{2}{1}
\glmptb \gnot{\hspace{-0,34cm}\Delta_C} \grmpb \gnl
\gend$ with 
$\gbeg{1}{2}
\gcl{1} \gnl
\gbmp{\s\Epsilon_C} \gnl
\gend\hspace{0,1cm}$. Set 
$
\gbeg{2}{2}
\got{1}{F} \got{1}{F} \gnl
\glmpt \gnot{\hspace{-0,34cm}\sigma} \grmpt \gnl
\gend:=
\gbeg{2}{3}
\got{1}{F} \got{1}{F} \gnl
\glmptb \gnot{\hspace{-0,34cm}\mu_M} \grmpt \gnl
\gcuf{1} \gnl
\gend \hspace{0,1cm}$ and 
$\gbeg{2}{2}
\glmpb \gnot{\hspace{-0,34cm}\rho} \grmpb \gnl
\gob{1}{F}\gob{1}{F}
\gend:=
\gbeg{2}{3}
\gu{1} \gnl
\glmptb \gnot{\hspace{-0,34cm}\Delta_C} \grmpb \gnl
\gob{1}{F}\gob{1}{F}
\gend$\hspace{0,1cm}. We further have the identities \equref{2-cocycle condition} -- \equref{normalized 2-cycle ro} (of which the first two and the last two are 
obtained in a different way than there), which come down to:

\pagebreak

\begin{equation}\eqlabel{2-cocycle condition I}
\gbeg{3}{5}
\got{1}{F} \got{1}{F} \got{1}{F}\gnl
\gcl{1} \glmptb \gnot{\hspace{-0,34cm}\mu_M} \grmpt \gnl
\glmpt \gnot{\hspace{-0,34cm}\sigma} \grmpt \gnl
\gend=
\gbeg{3}{5}
\got{1}{F} \got{1}{F} \got{1}{F} \gnl
\glmpt \gnot{\hspace{-0,34cm}\mu_M} \grmptb \gcl{1} \gnl
\gvac{1} \glmpt \gnot{\hspace{-0,34cm}\sigma} \grmpt \gnl
\gend \hspace{4cm}
\gbeg{2}{4}
\got{3}{F}\gnl
\gu{1}  \gcl{1} \gnl
\glmpt \gnot{\hspace{-0,34cm}\sigma} \grmpt \gnl
\gend=
\gbeg{1}{4}
\got{1}{F}\gnl
\gcl{1} \gnl
\gcuf{1} \gnl
\gend=
\gbeg{3}{4}
\got{1}{F}\gnl
\gcl{1} \gu{1} \gnl
\glmpt \gnot{\hspace{-0,34cm}\sigma} \grmpt \gnl
\gend
\end{equation} \vspace{-1cm}
$$\textnormal{  2-cocycle condition \hspace{3cm}  normalized 2-cocycle} $$
$$\textnormal{  \quad from monad law for $\mu_M$ \equref{monad law mu_M} \hspace{2cm} from monad law for $\eta_M$ \equref{monad law eta_M}}$$


\begin{equation} \eqlabel{2-cycle ro I}
\gbeg{3}{3}
\gvac{1} \glmpb \gnot{\hspace{-0,34cm}\rho} \grmpb \gnl
\glmpb \gnot{\hspace{-0,34cm}\Delta_C} \grmptb \gcl{1} \gnl
\gob{1}{F}\gob{1}{F}  \gob{1}{F}
\gend=
\gbeg{3}{3}
\glmpb \gnot{\hspace{-0,34cm}\rho} \grmpb \gnl
\gcl{1} \glmptb \gnot{\hspace{-0,34cm}\Delta_C} \grmpb \gnl
\gob{1}{F}\gob{1}{F}  \gob{1}{F}
\gend \hspace{4cm}
\gbeg{2}{3}
\glmpb \gnot{\hspace{-0,34cm}\rho} \grmpb \gnl
\gcuf{1} \gcl{1} \gnl
\gob{3}{F}
\gend=
\gbeg{1}{3}
\gu{1} \gnl
\gcl{1} \gnl
\gob{1}{F}
\gend=
\gbeg{3}{3}
\glmpb \gnot{\hspace{-0,34cm}\rho} \grmpb \gnl
\gcl{1} \gcuf{1} \gnl
\gob{1}{F} \gnl
\gend
\end{equation}
$$\textnormal{ 2-cycle condition for $\rho$ \hspace{3cm}  normalized 2-cycle $\rho$} $$
$$\textnormal{ \qquad from comonad law for $\Delta_C$ \equref{comonad law Delta_C}  \qquad\qquad from comonad law for $\Epsilon_C$ \equref{comonad law Epsilon_C}}$$


\bigskip

\bigskip

Let us now consider $\K=\hat\C$ with $\lambda$ defined by \equref{left lambda}. 
If we apply to \equref{Majid lambda_M 8} on the one hand 
$F\gbeg{1}{2}
\got{1}{F} \gnl
\gcuf{1} \gnl
\gend$ from below, and on the other  
$F \gbeg{1}{2}
\gu{1} \gnl
\gob{1}{F} \gnl
\gend$ from above, 
we get: 
$$\mu_M=
\gbeg{3}{6}
\got{2}{F} \got{1}{F} \gnl
\gcn{2}{1}{2}{2} \gcl{1} \gnl
\gcmuf{1} \gvac{2} \glmptb \gnot{\hspace{-0,34cm}\Delta_C} \grmpb  \gnl
\gcl{1} \gbr \gcl{1} \gnl
\gmu \glmpt \gnot{\hspace{-0,34cm}\sigma} \grmpt  \gnl
\gob{2}{F}
\gend \hspace{2cm}\textnormal{and} \hspace{2cm}
\Delta_C=
\gbeg{3}{6}
\got{2}{F} \gnl
\gcn{1}{1}{2}{2} \gnl
\gcmuf{1} \gvac{2}  \glmpb \gnot{\hspace{-0,34cm}\rho} \grmpb  \gnl
\gcl{1} \gbr \gcl{1} \gnl
\gmu \glmptb \gnot{\hspace{-0,34cm}\mu_M} \grmpt \gnl
\gob{2}{F} \gob{1}{F}
\gend
$$
respectively. We distinguish the following three cases: 
\begin{enumerate}
\item if 
$
\gbeg{2}{1}
\glmptb \gnot{\hspace{-0,34cm}\mu_M} \grmpt \gnl
\gend=
\gbeg{2}{1}
\gmu \gnl
\gend
$ and  
$
\gbeg{2}{1}
\glmptb \gnot{\hspace{-0,34cm}\Delta_C} \grmpb \gnl
\gend=
\gbeg{2}{1}
\gcmuf{1} \gnl
\gend,
$ we have that $\sigma$ and $\rho$ are trivial and the above data is merely a bialgebra $F$ in $\C$; 
\item if $
\gbeg{2}{1}
\glmptb \gnot{\hspace{-0,34cm}\Delta_C} \grmpb \gnl
\gend=
\gbeg{2}{1}
\gcmuf{1} \gnl
\gend,
$ but 
$\gbeg{2}{1}
\glmptb \gnot{\hspace{-0,34cm}\mu_M} \grmpt \gnl
\gend\not=
\gbeg{2}{1}
\gmu \gnl
\gend,
$ we have \equref{mu_m con sigma I}, which by \equref{2-cocycle condition I} delivers a normalized 2-cocycle \equref{right 2-coc*}: 
\begin{center} \hspace{-1,4cm}
\begin{tabular} {p{6cm}p{1cm}p{7cm}}
\begin{equation} \eqlabel{mu_m con sigma I}
\mu_M=
\gbeg{3}{6}
\got{2}{F} \got{2}{F} \gnl
\gcn{2}{1}{2}{2} \gcn{1}{1}{2}{2} \gnl
\gcmuf{1} \gvac{2} \gcmuf{1}  \gnl
\gcl{1} \gbr \gcl{1} \gnl
\gmu \glmpt \gnot{\hspace{-0,34cm}\sigma} \grmpt  \gnl
\gob{2}{F}
\gend
\end{equation} & &  \vspace{-1cm}
\begin{equation}\eqlabel{right 2-coc*}
\gbeg{5}{7}
\got{1}{F} \got{2}{F} \got{2}{F}\gnl
\gcl{1} \gcn{2}{1}{2}{2} \gcn{1}{1}{2}{2} \gnl
\gcl{1}  \gcmuf{1} \gvac{2} \gcmuf{1} \gnl
\gcl{1} \gcl{1} \gbr \gcl{1} \gnl
\gcn{1}{1}{1}{2} \gmu \glmpt \gnot{\hspace{-0,34cm}\sigma} \grmpt  \gnl
\gvac{1} \hspace{-0,24cm} \glmpt \gnot{\hspace{-0,34cm}\sigma} \grmpt \gnl
\gend=
\gbeg{5}{8}
\got{2}{F} \got{2}{F} \got{1}{F}\gnl
\gcn{2}{1}{2}{2} \gcn{2}{1}{2}{2} \gcl{4} \gnl
\gcmuf{1} \gvac{2} \gcmuf{1} \gvac{2} \gnl
\gcl{1} \gbr \gcl{1} \gnl
\gmu \glmpt \gnot{\hspace{-0,34cm}\sigma} \grmpt  \gnl
\gcn{1}{1}{2}{3} \gvac{2} \gcn{1}{1}{3}{1} \gnl
\gvac{1} \glmpt \gnot{\sigma} \gcmp \grmpt \gnl
\gend
\end{equation}
\end{tabular}
\end{center}
\item if $
\gbeg{2}{1}
\glmptb \gnot{\hspace{-0,34cm}\mu_M} \grmpt \gnl
\gend=
\gbeg{2}{1}
\gmu \gnl
\gend,
$ but 
$
\gbeg{2}{1}
\glmptb \gnot{\hspace{-0,34cm}\Delta_C} \grmpb \gnl
\gend\not=
\gbeg{2}{1}
\gcmuf{1} \gnl
\gend,
$ we have \equref{Delta_C con ro I}, which by \equref{2-cycle ro I} delivers a normalized 2-cycle \equref{Drinfeld's twist}: 
\begin{center} 
\begin{tabular} {p{5cm}p{1cm}p{7cm}}
\begin{equation} \eqlabel{Delta_C con ro I}
\Delta_C=
\gbeg{3}{6}
\got{2}{F} \gnl
\gcn{1}{1}{2}{2} \gnl
\gcmuf{1} \gvac{2} \glmpb \gnot{\hspace{-0,34cm}\rho} \grmpb  \gnl
\gcl{1} \gbr \gcl{1} \gnl
\gmu \gmu  \gnl
\gob{2}{F} \gob{2}{F}
\gend
\end{equation} & & \vspace{-1cm}
\begin{equation} \eqlabel{Drinfeld's twist}
\gbeg{6}{5}
\gvac{1}  \glmpb \gnot{\hspace{-0,34cm}\rho} \grmpb \gnl
\gvac{1} \hspace{-0,24cm} \gcn{1}{1}{2}{1} \gcmuf{1} \gvac{2} \glmpb \gnot{\hspace{-0,34cm}\rho} \grmpb  \gnl
\gvac{1} \gcl{1} \gcl{1} \gbr \gcl{1} \gnl
\gvac{1} \gcl{1}  \gmu \gmu \gnl
\gvac{1} \gob{1}{F} \gob{2}{F} \gob{2}{F}\gnl
\gend=
\gbeg{5}{6}
\gvac{1} \glmpb \gnot{\rho} \gcmp \grmpb \gnl
\gcn{1}{1}{3}{2} \gvac{2} \gcn{1}{1}{1}{3} \gnl
\gcmuf{1} \gvac{2} \glmpb \gnot{\hspace{-0,34cm}\rho} \grmpb \gcl{3}  \gnl
\gcl{1} \gbr \gcl{1} \gnl
\gmu \gmu \gnl
\gob{2}{F} \gob{2}{F} \gob{1}{F}\gnl
\gend
\end{equation}
\end{tabular}
\end{center}
\end{enumerate}

\begin{ex} \exlabel{Drinfeld}
The identity \equref{Drinfeld's twist} recovers the Drinfel'd twist, and \equref{right 2-coc*} its dual version, known in $\C=Vec$ from \cite{Maj5}. Moreover, 
\equref{right 2-coc*} coincides with Sweedler's 2-cocycle \equref{2-cocycle condition psi_1 caso 2} for $B=I$. We have studied them in \cite[Section 4.1]{th}. 
\end{ex}

Recall that $\mu_M$ is associative and $\Delta_C$ is coassociative. Then from the structure of a biwreath-like hybrid $F$ we deduce the following known fact (at least a) was known, \cite{th}):

\begin{cor}
Let $F$ be a bialgebra in $\C$. 
\begin{enumerate} [a)]
\item If $\sigma: FF\to I$ is a Sweedler's 2-cocycle in $\C$, then $\mu_M:FF\to F$ defined by \equref{mu_m con sigma I} is an associative multiplication on $F$. 
\item If $\rho: I\to FF$ is a Drinfel'd twist in $\C$, then $\Delta_C:F\to FF$ defined by \equref{Delta_C con ro I} is a coassociative comultiplication on $F$. 
\end{enumerate}
\end{cor}


\section{Mixed biwreath-like objects} \selabel{mixed bl objects}  

We study now the mixed (co)wreaths from the point of view of a biwreath.

\begin{defn} \delabel{bl mixed}
A mixed biwreath-like object in $\K$ is a monad $(F, \mu_C, \eta_C)$ in $\EM^C(\K)$ and a comonad $(F, \Delta_M, \Epsilon_M)$ in $\EM^M(\K)$ over the same bimonad $B$ in $\K$ 
with the canonical restrictions: 
$$
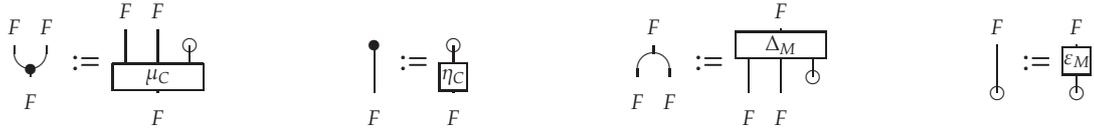
\gbeg{2}{3}
\got{1}{F} \got{1}{F} \gnl
\gmuf{1} \gnl
\gob{2}{F} \gnl
\gend:=
\gbeg{3}{4}
\got{1}{F} \got{1}{F} \gnl
\gcl{1} \gcl{1} \gu{1} \gnl
\glmpt \gnot{\mu_C} \gcmptb \grmpt \gnl
\gob{3}{F} \gnl
\gend \hspace{2cm}
\gbeg{1}{4}
\got{1}{} \gnl
\guf{1} \gnl
\gcl{1} \gnl
\gob{1}{F} \gnl
\gend:=
\gbeg{1}{4}
\got{1}{} \gnl
\gu{1} \gnl
\gbmp{\eta_C} \gnl
\gob{1}{F} \gnl
\gend \hspace{2cm}
\gbeg{2}{3}
\got{2}{F} \gnl
\gcmu \gnl
\gob{1}{F} \gob{1}{F} \gnl
\gend:=
\gbeg{3}{4}
\got{3}{F} \gnl
\glmpb \gnot{\Delta_M} \gcmptb \grmpb \gnl
\gcl{1} \gcl{1} \gcu{1} \gnl
\gob{1}{F} \gob{1}{F} \gnl
\gend \hspace{2cm}
\gbeg{1}{3}
\got{1}{F} \gnl
\gcl{1} \gnl
\gcu{1} \gnl
\gend:=
\gbeg{1}{3}
\got{1}{F} \gnl
\gbmp{\Epsilon_{M}} \gnl
\gcu{1} \gnl
\gob{1}{} \gnl
\gend
$$
equipped with a 2-cell $\lambda:FF\to FF$ in $\K$ so that: 
\begin{enumerate}[(a)]
\item $\lambda$ is a distributive law in the sense of \deref{bimonad} b); 
\item the following compatibility conditions are fulfilled: 
\begin{equation*} 
\gbeg{2}{3}
\got{1}{F} \got{1}{F} \gnl
\gcl{1} \gcl{1} \gnl
\gcu{1} \gcu{1} \gnl
\gend=
\gbeg{2}{3}
\got{1}{F} \got{1}{F} \gnl
\gmuf{1} \gnl
\gvac{1} \hspace{-0,34cm} \gcu{1} \gnl
\gend \hspace{2cm}
\gbeg{1}{4}
\got{2}{}  \gnl
 \guf{1} \gnl
\hspace{-0,34cm} \gcmu \gnl
\gob{1}{F} \gob{1}{F} \gnl
\gend=
\gbeg{2}{4}
\got{1}{} \gnl
\guf{1} \guf{1} \gnl
\gcl{1} \gcl{1} \gnl
\gob{1}{F} \gob{1}{F}
\gend  \hspace{2cm}
\gbeg{1}{2}
\guf{1} \gnl
\gcu{1} \gnl
\gob{2}{} \gnl
\gend=
\Id_{id_{\A}}
\end{equation*}
and
\begin{center} \hspace{-0,2cm}
\begin{tabular}{p{6cm}p{0cm}p{6cm}}
\begin{equation} \eqlabel{lambda_M 8*}
\gbeg{3}{5}
\gvac{2} \got{1}{\hspace{-0,4cm}F} \got{1}{\hspace{-0,4cm}F} \gnl
\gvac{2} \hspace{-0,34cm} \gmuf{1} \gnl
\gvac{1} \hspace{-0,34cm} \glmpb \gnot{\Delta_M} \gcmpb \grmptb \gnl
\gvac{1} \gcl{1} \gcl{1} \gcl{1} \gnl
\gvac{1} \gob{1}{F} \gob{1}{F} \gob{1}{B} \gnl
\gend=
\gbeg{4}{5}
\got{1}{F} \got{1}{F} \gnl
\gcl{1} \glmptb \gnot{\Delta_M} \gcmpb \grmpb \gnl
\glmptb \gnot{\hspace{-0,34cm}\lambda} \grmptb \gcl{1} \gcl{2} \gnl
\gcl{1} \gmuf{1} \gnl
\gob{1}{F} \gob{2}{F} \gob{1}{B} \gnl
\gend 
\end{equation} & & 
\begin{equation} \eqlabel{lambda_C 8*}
\gbeg{3}{5}
\got{1}{F} \got{1}{F} \got{1}{B} \gnl
\gcl{1} \gcl{1} \gcl{1} \gnl
\glmpt \gnot{\mu_C} \gcmpt \grmptb \gnl
\gvac{2} \hspace{-0,22cm} \gcmu \gnl
\gvac{2} \gob{1}{F} \gob{1}{F} \gnl
\gend
=
\gbeg{4}{5}
\got{1}{F} \got{2}{F} \got{1}{B} \gnl
\gcl{1} \gcmu \gcl{2} \gnl
\glmptb \gnot{\hspace{-0,34cm}\lambda} \grmptb \gcl{1} \gnl
\gcl{1} \glmptb \gnot{\mu_C} \gcmpt \grmpt \gnl
\gob{1}{F} \gob{1}{F} \gnl
\gend
\end{equation}
\end{tabular}
\end{center}
\end{enumerate}
\end{defn}

Applying $\Epsilon_B$ to \equref{lambda_M 8*} (or applying $\eta_B$ to \equref{lambda_C 8*}), one gets \equref{bimonad in K*}, 
then $F$ is a bimonad in $\K$ with possibly non-(co)associative (co)multiplication. Applying 
$F
\gbeg{1}{2}
\guf{1} \gnl
\gob{1}{F}
\gend$ to \equref{lambda_M 8*} and 
$F
\gbeg{1}{2}
\got{1}{F}\gnl
\gcu{1} \gnl
\gnl
\gend$ to \equref{lambda_C 8*}, similarly as in \equref{Delta_m con Fi-lambda} and \equref{mu_C con omega}, we obtain $\Delta_M$ and $\mu_C$ below: 
\begin{center} \hspace{-1,3cm} 
\begin{tabular}{p{4.5cm}p{0,7cm}p{5cm}p{2cm}p{5cm}}
\begin{equation} \eqlabel{bimonad in K*}
\gbeg{2}{4}
\got{1}{F} \got{1}{F} \gnl
\gmuf{1} \gnl
\gcmu \gnl
\gob{1}{F}\gob{1}{F} 
\gend=
\gbeg{3}{5}
\got{1}{F} \got{2}{F} \gnl
\gcl{1} \gcmu \gnl
\glmptb \gnot{\hspace{-0,34cm}\lambda} \grmptb  \gcl{1} \gnl
\gcl{1} \gmuf{1} \gnl
\gob{1}{F}  \gob{2}{F}
\gend
\end{equation} & & 
\begin{equation*}
\Delta_M=
\gbeg{4}{5}
\got{1}{F} \gnl
\gcl{1} \gcn{1}{1}{2}{1} \gelt{\s\Phi_{\lambda}} \gcn{1}{1}{0}{1} \gnl  
\glmptb \gnot{\hspace{-0,34cm}\lambda} \grmptb \gcl{1} \gcl{2} \gnl
\gcl{1} \gmuf{1} \gnl
\gob{1}{F} \gob{2}{F} \gob{1}{B} \gnl
\gend
\hspace{2cm}
\mu_C=
\gbeg{4}{5}
\got{1}{F} \got{2}{F} \got{1}{B} \gnl
\gcl{1} \gcmu \gcl{2} \gnl
\glmptb \gnot{\hspace{-0,34cm}\lambda} \grmptb \gcl{1} \gnl
\gcl{1} \gcn{1}{1}{1}{2} \gelt{\omega} \gcn{1}{1}{1}{0} \gnl   
\gnl
\gob{1}{F} \gnl
\gend
\end{equation*}
\end{tabular}
\end{center}
where
\begin{center} 
\begin{tabular}{p{5cm}p{0cm}p{5cm}}
\begin{equation} \eqlabel{fi-lambda 3}
\Phi_{\lambda}:=
\gbeg{3}{4}
\gvac{1} \guf{1} \gnl
\glmpb \gnot{\Delta_M} \gcmptb \grmpb \gnl
\gcl{1} \gcl{1} \gcl{1} \gnl
\gob{1}{F} \gob{1}{F} \gob{1}{B} \gnl
\gend 
\end{equation} & &
\begin{equation} \eqlabel{omega 3}
\omega:=
\gbeg{2}{4}
\got{1}{F} \got{1}{B} \got{1}{B} \gnl
\gcl{1} \gcl{1} \gcl{1} \gnl
\glmpt \gnot{\mu_C} \gcmptb \grmpt \gnl
\gvac{1} \gcu{1} \gnl
\gend
\end{equation} 
\end{tabular}
\end{center}

\subsection{Case three: $\D_3={}_F ^F\YD(\C)$ with ``alternative'' quasi (co)actions } \sslabel{Case 3 alt}

For a mixed biwreath-like object in $\K=\hat\C$ with $\lambda$ given by \equref{left lambda*} we have \equref{Delta_M alt} and \equref{mu_C alt}: 
\begin{center} \hspace{-1,4cm}
\begin{tabular}{p{3.6cm}p{0cm}p{5cm}p{0cm}p{5cm}}
\begin{equation} \eqlabel{left lambda*}
\lambda=
\gbeg{3}{5}
\got{2}{F} \got{1}{F} \gnl
\gcmu \gcl{1} \gnl  %
\gcl{1} \gbr  \gnl
\gmuf{1} \gcl{1} \gnl
\gob{2}{F} \gob{1}{F}
\gend
\end{equation} & & \vspace{-0,6cm}
\begin{equation} \eqlabel{Delta_M alt}
\Delta_M=
\gbeg{4}{6}
\got{2}{F} \gnl
\gcmu \gcn{1}{1}{2}{1} \gelt{\s\Phi_{\lambda}} \gcn{1}{1}{0}{1}  \gnl 
\gcl{1} \gbr \gcl{1} \gcl{3} \gnl
\gmuf{1} \gmuf{1} \gnl
\gcn{2}{1}{2}{2} \gcn{2}{1}{2}{2} \gnl
\gob{2}{F} \gob{2}{F} \gob{1}{B} \gnl
\gend
\end{equation} & & \vspace{-0,4cm}
\begin{equation} \eqlabel{mu_C alt}
\mu_C=
\gbeg{4}{5}
\got{2}{F} \got{2}{F} \got{1}{B} \gnl
\gcmu \gcmu \gcl{2} \gnl 
\gcl{1} \gbr \gcl{1}  \gnl
\gmuf{1} \gcn{1}{1}{1}{2} \gelt{\omega} \gcn{1}{1}{1}{0} \gnl
\gnl
\gob{2}{F.} \gnl
\gend
\end{equation}
\end{tabular}
\end{center}
Now consider the category $\D_3={}_F ^F\YD(\C)$ of left-left Yetter-Drinfel'd modules over $F$ in $\C$ with a pre-braiding, 
$\psi=d^3_{B,F}$ and $\phi=d^3_{F,B}$ given as below:
$$d^3_{X,Y}=
\gbeg{3}{5}
\got{1}{} \got{1}{X} \got{1}{Y} \gnl
\glcm \gcl{1} \gnl
\gcl{1} \gbr \gnl
\glm \gcl{1} \gnl
\gob{1}{} \gob{1}{Y} \gob{1}{X} 
\gend
\hspace{1,5cm}
\psi=
\gbeg{3}{5}
\got{3}{B} \got{-1}{F} \gnl
\grmo \gvac{1} \gcl{1} \gcl{1} \gnl
\gcl{1} \gbr \gnl
\gmuf{1} \gcl{1} \gnl
\gob{2}{F} \gob{1}{B} 
\gend 
\hspace{1,5cm}
\phi= 
\gbeg{3}{5}
\got{2}{F} \got{1}{B} \gnl
\gcmu \gcl{1} \gnl
\gcl{1} \gbr \gnl
\glmf \gcn{1}{1}{-1}{-1} \gcn{1}{1}{-1}{-1} \gnl
\gob{3}{B} \gob{-1}{F.} 
\gend
$$
Then similarly as in \equref{left F-mod} and \equref{left F-comod}, for the left $F$-(co)actions on $B$ we have: 
$$\gbeg{2}{3}
\got{3}{B} \gnl
\grmo \gvac{1} \gcl{1} \gnl
\gob{1}{F} \gob{1}{B}
\gend=
\gbeg{2}{4}
\got{1}{B} \gnl
\gcl{1} \guf{1} \gnl
\glmptb \gnot{\hspace{-0,34cm}\psi} \grmptb \gnl
\gob{1}{F} \gob{1}{B} 
\gend
\hspace{4cm} 
\gbeg{2}{3}
\got{1}{F} \got{1}{B} \gnl
\glmf \gcn{1}{1}{-1}{-1} \gnl
\gob{3}{B}
\gend=
\gbeg{2}{4}
\got{1}{F} \got{1}{B} \gnl
\glmptb \gnot{\hspace{-0,34cm}\phi} \grmptb \gnl
\gcl{1} \gcu{1} \gnl
\gob{1}{B}  
\gend
$$
The structure of the mixed biwreath-like object on $F$ with the above $\psi, \phi$ and $\lambda$ is governed by the axioms of a biwreath that involve the mixed wreath and the mixed cowreath. 
Similarly as in \ssref{ex psi_1} and \ssref{ex psi_2} we get that $F$ is a trivial left $B$-(co)module and that $B$ is a trivial right $F$-(co)module. Then from \equref{quasi assoc. mu_C} 
and \equref{quasi coass. Delta_M} we get that the involved pre-(co)multiplications of $F$ are (co)associative, so $F$ is a usual bialgebra in $\C$. 
Also the axiom \equref{psi laws for B} for the multiplication is fulfilled -- this computation is completely analogous (the diagrams in the computation 
have a similar form) to the one we did in \ssref{ex psi_1} with the following differences: where we used the bialgebra property of $B$ now we use the left $F$-comodule algebra 
property \equref{F comod alg} of $B$; where we used that $F$ is a left $B$-module now we use that 
$\gbeg{2}{1}
\gmuf{1} \gnl
\gend$ of $F$ is associative. The unit axiom in \equref{psi laws for B} is clearly fulfilled and the axioms \equref{phi laws for B} hold by the duality of the mixed biwreath-like construction.

Since the left $B$-(co)module structures on $F$ are trivial, the identities \equref{mod coalg} -- \equref {comod alg unity} are trivially satisfied (they do not deliver any new information). 
The same happens with \equref{F mod alg} -- \equref{F mod alg unit} and \equref{F comod coalg} -- \equref{F comod coalg counit} because the right $F$-(co)module structures of $B$ are trivial. 
Though, the equations \equref{F comod alg} -- \equref{F mod coalg counit} say that $B$ is a left $F$-comodule algebra and a left $F$-module coalgebra in $\C$, again in a broader sense, 
we still do not know if it has proper structures of a left $F$-(co)module. 

From \equref{quasi coaction} -- \equref{dual quasi action unity} we obtain that $B$ is a left {\em alternative quasi $F$-comodule} and a left {\em alternative quasi $F$-module}. 
In the next subsection we will analyze a similar structure without the adjective ``alternative'' and there we will explain the motivation for our terminology and notation. 
The coaction in \equref{quasi coaction} is twisted by $\Phi_{\lambda}$ which by \equref{3-cocycle cond fi-lambda} -- \equref{normalized 3-cocycle fi-lambda} is a normalized 3-cocycle, 
and the action in \equref{dual quasi action} is twisted by $\omega$ which by \equref{3-cycle omega} -- \equref{normalized 3-cycle omega} is a normalized 3-cycle. Here the name 3-(co)cycle 
is not completely correct, we will explain this later, too. Putting \equref{Delta_M alt} and \equref{mu_C alt} and the above $\psi$ and $\phi$ in \equref{quasi coaction} -- \equref{dual quasi action unity}, 
\equref{3-cycle omega} -- \equref{normalized 3-cocycle fi-lambda} we obtain: \vspace{-0,4cm}
\begin{center} \hspace{-1,4cm} 
\begin{tabular} {p{8.8cm}p{2cm}p{4cm}} 
\begin{equation} \eqlabel{quasi coaction case 3 alt} 
\gbeg{6}{8}
\got{11}{B} \gnl
\gvac{4} \grmo \gvac{1} \gcl{1} \gnl  
\gvac{2} \gcn{3}{2}{5}{-2}  \gcl{4} \gnl
\gvac{2}  \gnl  
\hspace{-0,22cm} \gcmu \gcn{1}{1}{2}{1} \gelt{\s\Phi_{\lambda}} \gcn{1}{1}{0}{1} \gnl 
\gcl{1} \gbr \gcl{1} \gcl{1} \gnl
\gmuf{1} \gmuf{1} \gmu \gnl
\gob{2}{F} \gob{2}{F} \gob{2}{B}
\gend=
\gbeg{3}{9}
\got{3}{B} \gnl
\gvac{1} \gcl{1} \gcn{1}{1}{2}{1} \gelt{\s\Phi_{\lambda}} \gcn{1}{1}{0}{1} \gnl
\grmo \gvac{1} \gcl{1} \gcl{1} \gcl{3} \gcl{5} \gnl
\gcl{1} \gbr \gnl
\gmuf{1} \gcl{1} \gnl
\gcn{1}{1}{2}{1}  \grmo \gvac{1} \gcl{1} \gcl{1} \gnl
\gcl{1} \gcl{1} \gbr \gnl
\gcl{1} \gmuf{1} \gmu \gnl
\gob{1}{F} \gob{2}{F} \gob{2}{B}
\gend 
\end{equation} & & \vspace{0,8cm}
\begin{equation*} 
\gbeg{2}{4}
\got{3}{B} \gnl
\grmo \gvac{1} \gcl{1} \gnl
\gcu{1} \gcl{1} \gnl
\gob{3}{B} 
\gend=
\gbeg{2}{4}
\got{1}{B} \gnl
\gcl{2} \gnl
\gob{1}{B} 
\gend
\end{equation*} 
\end{tabular}
\end{center} \vspace{-0,5cm}
$$ \textnormal{\hspace{1cm}\footnotesize alternative quasi coaction}  \hspace{5,6cm}  \textnormal{\footnotesize quasi coaction counity} $$ \vspace{-0,7cm}
$$ \textnormal{ \hspace{1cm} \footnotesize from \equref{quasi coaction}} \hspace{7cm} \textnormal{\footnotesize from \equref{quasi coaction counity}} $$

\begin{center} \hspace{-1,4cm} 
\begin{tabular}{p{10.6cm}p{0cm}p{5.3cm}}
\begin{equation} \eqlabel{3-cocycle cond fi-lambda case 3 alt} 
\gbeg{8}{7}
\gvac{2} \gcn{1}{2}{2}{-2} \gelt{\s\Phi_{\lambda}} \gcn{2}{2}{0}{7}  \gnl
\gvac{4} \gcn{2}{1}{0}{4} \gnl
\gcmu \gvac{1} \gcn{1}{1}{2}{-1} \gelt{\s\Phi_{\lambda}} \gcn{1}{1}{0}{1} \gcn{1}{1}{0}{1}  \gcl{1} \gnl
\gcl{1} \gbr \gcn{1}{1}{3}{1} \grmo \gvac{1} \gcl{1} \gcl{1} \gcl{2} \gnl
\gmuf{1} \gmuf{1}  \gcl{1} \gbr \gnl
\gcn{1}{1}{2}{2} \gvac{1} \gcn{1}{1}{2}{2} \gvac{1} \gmuf{1} \gmu \gnl
\gob{2}{F} \gob{2}{F} \gob{2}{F} \gob{2}{B} \gnl
\gend=
\gbeg{5}{6}
\gvac{1} \gcn{1}{2}{2}{-1} \gelt{\s\Phi_{\lambda}} \gcn{2}{2}{0}{7}  \gnl
\gvac{1} \gcn{1}{1}{3}{2} \gnl
\gcl{3} \gcmu \gcn{1}{1}{2}{1} \gelt{\s\Phi_{\lambda}} \gcn{1}{1}{0}{1} \gcl{2} \gnl
\gvac{1} \gcl{1} \gbr \gcl{1} \gcl{1} \gnl
\gvac{1} \gmuf{1} \gmuf{1}  \gmu \gnl
\gob{1}{F} \gob{2}{F} \gob{2}{F} \gob{2}{B} \gnl
\gend 
\end{equation} &  & \vspace{0,6cm}
\begin{equation*} 
\gbeg{3}{3}
\gcn{1}{1}{2}{1} \gelt{\s\Phi_{\lambda}} \gcn{1}{1}{0}{1} \gnl 
\gcu{1} \gcl{1} \gcl{1} \gnl
\gvac{1} \gob{1}{F} \gob{1}{B} \gnl
\gend=
\gbeg{2}{3}
\gu{1} \gu{1} \gnl
\gcl{1} \gcl{1} \gnl
\gob{1}{F} \gob{1}{B} \gnl
\gend=
\gbeg{3}{3}
\gcn{1}{1}{2}{1} \gelt{\s\Phi_{\lambda}} \gcn{1}{1}{0}{1} \gnl 
\gcl{1} \gcu{1} \gcl{1} \gnl
\gob{1}{F} \gob{3}{B} \gnl
\gend
\end{equation*}
\end{tabular}
\end{center} \vspace{-0,5cm}
$$ \textnormal{\hspace{1cm}\footnotesize alternative 3-cocycle condition for $\Phi_{\lambda}$}  \hspace{4cm}  \textnormal{\footnotesize normalized 3-cocycle $\Phi_{\lambda}$} $$ \vspace{-0,7cm}
$$ \textnormal{ \hspace{1cm} \footnotesize from \equref{3-cocycle cond fi-lambda}} \hspace{7cm} \textnormal{\footnotesize from \equref{normalized 3-cocycle fi-lambda}} $$

\pagebreak

$$
\gbeg{5}{9}
\got{1}{F} \got{2}{F} \got{2}{B} \gnl
\gcl{1} \gcmu \gcmu \gnl
\gcl{1} \gcl{1} \gbr \gcl{5} \gnl
\gcn{1}{1}{1}{2}  \glmf \gcn{1}{1}{-1}{-1} \gcn{1}{4}{-1}{-1} \gnl
\gcmu \gcl{1} \gnl
\gcl{1} \gbr \gnl
\glmf \gcn{1}{1}{-1}{-1} \gcn{1}{1}{-1}{-1} \gnl
\gvac{1} \gcl{1} \gcn{1}{1}{1}{2} \gelt{\s\omega} \gcn{1}{1}{1}{0} \gnl
\gob{3}{B} 
\gend=
\gbeg{4}{7}
\got{2}{F} \got{2}{F} \got{2}{B} \gnl
\gcmu \gcmu  \gcmu \gnl
\gcl{1} \gbr \gcl{1} \gcl{1} \gcl{3} \gnl
\gmuf{1} \gcn{1}{1}{1}{2} \gelt{\s\omega} \gcn{1}{1}{1}{0}  \gnl
\gcn{1}{1}{2}{9}  \gnl
\gvac{4}  \glmf \gcn{1}{1}{-1}{-1} \gnl
\gob{11}{B} 
\gend \hspace{4,5cm}
\gbeg{1}{4}
\got{1}{B}\gnl
\gcl{2}  \gnl
\gob{1}{B} 
\gend=
\gbeg{3}{5}
\got{3}{B}\gnl
\guf{1} \gcl{1} \gnl
\glmf \gcn{1}{1}{-1}{-1}  \gnl
\gvac{1} \gcl{1} \gnl
\gob{3}{B} 
\gend
$$
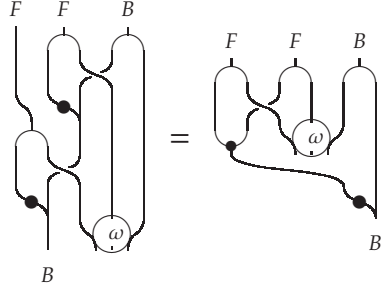
$$\qquad \textnormal{ \footnotesize alternative  quasi action}  \hspace{4cm}  \textnormal{\footnotesize quasi action unity} $$
$$\qquad \textnormal{\footnotesize from \equref{dual quasi action}} \hspace{6cm} \textnormal{\footnotesize from \equref{dual quasi action unity}} $$ 

$$
\gbeg{7}{7}
\got{1}{F} \got{2}{F} \got{2}{F} \got{2}{B} \gnl
\gcl{3} \gcmu \gcmu \gcmu \gnl
\gvac{1} \gcl{1} \gbr \gcl{1} \gcl{1} \gcl{2} \gnl
\gvac{1} \gmuf{1} \gcn{1}{1}{1}{2} \gelt{\s\omega} \gcn{1}{1}{1}{0}  \gnl
\gcn{1}{2}{1}{6} \gcn{1}{1}{2}{4} \gvac{4} \gcn{2}{2}{1}{-4} \gnl
\gvac{3} \gelt{\s\omega} \gnl
\gend=
\gbeg{4}{7}
\got{2}{F} \got{2}{F} \got{2}{F} \got{2}{B} \gnl
\gcn{1}{1}{2}{2} \gvac{1} \gcn{1}{1}{2}{2} \gvac{1} \gcmu \gcmu \gnl
\gcmu \gcmu  \gcl{1} \gbr \gcl{3} \gnl
\gcl{1} \gbr \gcl{1} \glmf \gcn{1}{1}{-1}{-1} \gcn{1}{1}{-1}{-1} \gnl
\gmuf{1} \gcn{1}{1}{1}{2} \gelt{\s\omega} \gcn{1}{1}{3}{0} \gcn{2}{2}{3}{0} \gnl
\gcn{1}{2}{2}{8} \gvac{5} \gcn{1}{2}{3}{-2} \gnl
\gvac{4} \gelt{\s\omega} \gnl
\gend \hspace{4cm}
\gbeg{3}{4}
\got{1}{F} \got{3}{B} \gnl
\gcl{1} \guf{1} \gcl{1} \gnl
\gcn{1}{1}{1}{2} \gelt{\omega} \gcn{1}{1}{1}{0} \gnl
\gend=
\gbeg{2}{4}
\got{1}{F} \got{1}{B} \gnl
\gcl{1} \gcl{1} \gnl
\gcu{1} \gcu{1} \gnl
\gend=
\gbeg{4}{4}
\gvac{1} \got{1}{F} \got{1}{B} \gnl
\guf{1} \gcl{1} \gcl{1} \gnl
\gcn{1}{1}{1}{2} \gelt{\omega} \gcn{1}{1}{1}{0} \gnl
\gend
$$ 

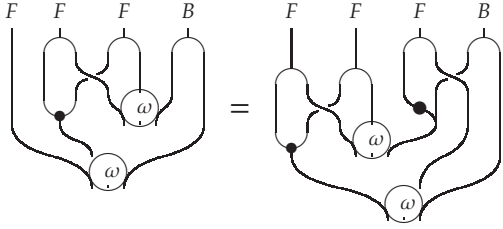
$$\quad  \textnormal{\footnotesize alternative 3-cycle condition $\omega$} \hspace{3,8cm}\textnormal{\footnotesize normalized 3-cycle $\omega$} $$
$$\qquad\textnormal{\footnotesize from \equref{3-cycle omega}} \hspace{6,8cm} \textnormal{\footnotesize from \equref{normalized 3-cycle omega}} $$

The above provides an alternative definition of a twisted (co)action in the context of braided monoidal categories. When $F$ is (co)commutative, the two definitions coincide.

\begin{ex} \exlabel{alternative}
When $\C={}_R\M$ the category of modules over a commutative ring $R$ and $B,F$ are $R$-bialgebras, the above identities for a left alternative quasi-coaction of $F$ on $B$ 
with a 3-cocycle $\Phi_{\lambda}$ take form: 
$$b_{[-1]_(1)}X^1\ot b_{[-1]_(2)}X^2\ot X^3b_{[0]}=b_{[-2]} X^1\ot b_{[-1]} X^2\ot b_{[0]}X^3$$
$$\Epsilon_F(b_{[-1]})b_{[0]}=b$$
$$X^1_{(1)}Y^1\ot X^1_{(2)}Y^2\ot Y^3_{[-1]}X^2\ot Y^3X^3=X^1\ot X^2_{(2)}Y^1\ot X^2_{(2)}Y^2\ot Y^3X^3$$
$$\Epsilon_F(X^1)X^2\ot X^3=1_F\ot 1_B= X^1\Epsilon_F(X^2)\ot X^3$$
where $b\in B$ and $\Phi_{\lambda}=X^1\ot X^2\ot X^3=Y^1\ot Y^2\ot Y^3$. For a left alternative quasi-action of $F$ on $B$ with a 3-cycle $\omega$ we have: 
$$f_{(1)}\cdot(g_{(1)}\cdot b_{(1)}) \omega(f_{(2)}, g_{(2)}, b_{(2)})=f_{(1)}g_{(1)}\omega(f_{(2)}, g_{(2)}, b_{(1)})\cdot b_{(2)}$$
$$1_F\cdot b=b$$
$$\omega(f, g_{(1)}h_{(1)}, b_{(2)})\omega(g_{(2)}, h_{(2)}, b_{(1)})=\omega(f_{(1)}g_{(1)}, h_{(2)}, b_{(2)})\omega(f_{(2)}, g_{(2)}, h_{(1)}b_{(1)})$$
$$\omega(f, 1_F, b)=\Epsilon_F(f)\Epsilon_B(b)=\omega(1_F, f, b)$$
for $f,g,h\in F$ and $b\in B$.
When $F$ is commutative, alternative quasi-coaction coincides with quasi-coaction \equref{quasi coaction case 3}, \equref{3-cocycle cond fi-lambda case 3}, 
and similarly when $F$ is cocommutative: alternative quasi-action coincides with quasi-action. 
\end{ex}

\subsection{Left-right mixed biwreath-like objects} \sslabel{left-right}

We  have defined a bimonad in $\K$ in \deref{bimonad}. Due to properties (b) and (c) we may call this a left bimonad in $\K$. Then it is clear how a right bimonad in $\K$ should be defined. 
In \seref{bEM cat} we have defined the 2-category $\bEM(\K)$ using left bimonads 
and so that the 2-cells $\psi, \phi$ satisfy \equref{psi laws for bimonads} and \equref{phi laws for bimonads}. We will say for these properties that 
$\psi$ and $\phi$ are left distributive laws, more precisely, that $\psi$ is left monadic and $\phi$ left comonadic.  
For a biwreath this means that its $\psi$ and $\phi$ have the analogous behaviour and that $\lambda_M$ and $\lambda_C$ satisfy \equref{8 lambda_M} and \equref{8 lambda_C}. 
Accordingly, let us denote the so far studied 2-category $\bEM(\K)$ by $\bEM(\K)_{left}$.

We may consider the right hand-side version of the latter 2-category in the obvious way, we denote it by $\bEM(\K)_{right}$, 
but we may also consider a mixed version: with right bimonads and where $\psi$ and $\phi$ are left distributive laws. 
We denote this 2-category by $\bEM(\K)_{l\x r}$. 
For a right bimonad in this 2-category and for which $\lambda_M, \lambda_C$ are canonical and $\widetilde{\lambda_M}=\widetilde{\lambda_C}=\lambda$,  
the identities corresponding to \equref{8 lambda_M} and \equref{8 lambda_C} take form: 
$$
\gbeg{4}{5}
\gvac{2} \got{1}{F} \got{1}{F} \gnl
\gvac{2} \glmptb \gnot{\hspace{-0,34cm}\mu_M} \grmptb \gnl
\glmpb \gnot{\Delta_M} \gcmpb \grmptb \gcl{1} \gnl
\gcl{1} \gcl{1} \gmu \gnl
\gob{1}{F} \gob{1}{F} \gob{2}{B} \gnl
\gend=
\gbeg{5}{8}
\got{1}{F} \got{5}{F} \gnl
\glmptb \gnot{\Delta_M} \gcmpb \grmpb \gcl{1} \gnl
\gcl{2} \gcl{1} \glmptb \gnot{\hspace{-0,34cm}\psi} \grmptb \gnl
\gcl{1} \glmptb \gnot{\hspace{-0,34cm}\lambda} \grmptb \gcl{3} \gnl
\glmptb \gnot{\hspace{-0,34cm}\mu_M} \grmptb \gcl{1} \gnl
\gcl{1} \glmptb \gnot{\hspace{-0,34cm}\psi} \grmptb \gnl
\gcl{1}  \gcl{1}  \gmu \gnl
\gob{1}{F} \gob{1}{F} \gob{2}{B} \gnl
\gend 
\hspace{2cm}
\gbeg{4}{5}
\got{1}{F} \got{1}{F} \got{2}{B} \gnl
\gcl{1} \gcl{1} \gcmu \gnl
\glmpt \gnot{\mu_C} \gcmpt \grmptb \gcl{1} \gnl
\gvac{2} \glmptb \gnot{\hspace{-0,34cm}\Delta_C} \grmptb \gnl
\gvac{2} \gob{1}{F} \gob{1}{F} \gnl
\gend=
\gbeg{5}{8}
\got{1}{F} \got{1}{F} \got{2}{B} \gnl
\gcl{2} \gcl{1} \gcmu \gnl
\gvac{1} \glmptb \gnot{\hspace{-0,34cm}\phi} \grmptb \gcl{3} \gnl
\glmptb \gnot{\hspace{-0,34cm}\Delta_C} \grmptb \gcl{1} \gnl
\gcl{2} \glmptb \gnot{\hspace{-0,34cm}\lambda} \grmptb \gcl{1} \gnl
\gvac{1} \gcl{1} \glmptb \gnot{\hspace{-0,34cm}\phi} \grmptb \gnl
\glmpt \gnot{\mu_C} \gcmptb \grmpt \gcl{1} \gnl
\gvac{1} \gob{1}{F} \gvac{1} \gob{1}{F} \gnl
\gend
$$
This motivates our next definition which differs from \deref{bl mixed} in the part a) and the identities \equref{lambda_M 8*} and \equref{lambda_C 8*} of the latter: 

\begin{defn} \delabel{lr mixed bl o}
A left-right mixed biwreath-like object in $\K$ is a monad $(F, \mu_C, \eta_C)$ in $\EM^C(\K)$ and a comonad $(F, \Delta_M, \Epsilon_M)$ in $\EM^M(\K)$ over the same right bimonad $B$ in $\K$ 
with the canonical restrictions: 
$$
\gbeg{2}{3}
\got{1}{F} \got{1}{F} \gnl
\gmuf{1} \gnl
\gob{2}{F} \gnl
\gend:=
\gbeg{3}{4}
\got{1}{F} \got{1}{F} \gnl
\gcl{1} \gcl{1} \gu{1} \gnl
\glmpt \gnot{\mu_C} \gcmptb \grmpt \gnl
\gob{3}{F} \gnl
\gend \hspace{2cm}
\gbeg{1}{4}
\got{1}{} \gnl
\guf{1} \gnl
\gcl{1} \gnl
\gob{1}{F} \gnl
\gend:=
\gbeg{1}{4}
\got{1}{} \gnl
\gu{1} \gnl
\gbmp{\eta_C} \gnl
\gob{1}{F} \gnl
\gend \hspace{2cm}
\gbeg{2}{3}
\got{2}{F} \gnl
\gcmu \gnl
\gob{1}{F} \gob{1}{F} \gnl
\gend:=
\gbeg{3}{4}
\got{3}{F} \gnl
\glmpb \gnot{\Delta_M} \gcmptb \grmpb \gnl
\gcl{1} \gcl{1} \gcu{1} \gnl
\gob{1}{F} \gob{1}{F} \gnl
\gend \hspace{2cm}
\gbeg{1}{3}
\got{1}{F} \gnl
\gcl{1} \gnl
\gcu{1} \gnl
\gend:=
\gbeg{1}{3}
\got{1}{F} \gnl
\gbmp{\Epsilon_{M}} \gnl
\gcu{1} \gnl
\gob{1}{} \gnl
\gend
$$
equipped with a 2-cell $\lambda: FF\to FF$ which is a right distributive law, that is: 
$$
\gbeg{3}{5}
\got{1}{F}\got{1}{F}\got{1}{F}\gnl
\glmptb \gnot{\hspace{-0,34cm}\lambda} \grmptb \gcl{1} \gnl
\gcl{1} \glmptb \gnot{\hspace{-0,34cm}\lambda} \grmptb \gnl
\gmuf{1} \gcl{1} \gnl
\gob{2}{F} \gob{1}{F}
\gend=
\gbeg{3}{5}
\got{0}{F}\got{2}{F}\got{0}{F}\gnl
\gcn{1}{1}{0}{1} \hspace{-0,24cm} \gmuf{1} \gnl 
\gvac{1} \hspace{-0,34cm} \glmptb \gnot{\hspace{-0,34cm}\lambda} \grmptb  \gnl
\gvac{1} \gcl{1} \gcl{1} \gnl
\gvac{1} \gob{1}{F} \gob{1}{F}
\gend, \hspace{0,5cm}
\quad
\gbeg{2}{5}
\got{1}{F} \gnl
\gcl{1} \guf{1} \gnl
\glmptb \gnot{\hspace{-0,34cm}\lambda} \grmptb \gnl
\gcl{1} \gcl{1} \gnl
\gob{1}{F} \gob{1}{F}
\gend=
\gbeg{2}{5}
\got{3}{F} \gnl
\guf{1} \gcl{1} \gnl
\gcl{2} \gcl{2} \gnl
\gob{1}{F} \gob{1}{F}
\gend; \hspace{1,5cm}
\gbeg{3}{5}
\got{2}{F} \got{1}{F}\gnl
\gcmu \gcl{1} \gnl
\gcl{1} \glmptb \gnot{\hspace{-0,34cm}\lambda} \grmptb \gnl
\glmptb \gnot{\hspace{-0,34cm}\lambda} \grmptb \gcl{1} \gnl
\gob{1}{F} \gob{1}{F} \gob{1}{F} 
\gend=
\gbeg{3}{5}
\got{1}{F} \got{1}{F} \gnl
\gcl{1} \gcl{1} \gnl
\glmptb \gnot{\hspace{-0,34cm}\lambda} \grmptb \gnl
\gcn{1}{1}{1}{0} \hspace{-0,22cm} \gcmu \gnl
\gob{1}{F} \gob{1}{F} \gob{1}{F} 
\gend, \hspace{0,5cm}
\quad
\gbeg{2}{5}
\got{1}{F} \got{1}{F} \gnl
\gcl{1} \gcl{1} \gnl
\glmptb \gnot{\hspace{-0,34cm}\lambda} \grmptb \gnl
\gcl{1} \gcu{1} \gnl
\gob{1}{F} 
\gend=
\gbeg{2}{5}
\got{1}{F} \got{1}{F} \gnl
\gcl{1} \gcl{3} \gnl
\hspace{-0,32cm} \gcu{2} \gnl
\gob{4}{F} 
\gend
$$
and 2-cells $\tau_1: BF\to FB$ and $\tau_2: FB\to BF$ which are both monadic and comonadic distributive laws at appropiate side 
(that is, $\tau_1$ is left monadic and right comonadic, and $\tau_2$ is left comonadic and right monadic), so that the following compatibility conditions hold: 
\begin{equation}\eqlabel{lambda mixed 1-3 lr}
\gbeg{2}{3}
\got{1}{F} \got{1}{F} \gnl
\gcl{1} \gcl{1} \gnl
\gcu{1} \gcu{1} \gnl
\gend=
\gbeg{2}{3}
\got{1}{F} \got{1}{F} \gnl
\gmuf{1} \gnl
\gvac{1} \hspace{-0,34cm} \gcu{1} \gnl
\gend \hspace{2cm}
\gbeg{1}{4}
\got{2}{}  \gnl
 \guf{1} \gnl
\hspace{-0,34cm} \gcmu \gnl
\gob{1}{F} \gob{1}{F} \gnl
\gend=
\gbeg{2}{4}
\got{1}{} \gnl
\guf{1} \guf{1} \gnl
\gcl{1} \gcl{1} \gnl
\gob{1}{F} \gob{1}{F}
\gend  \hspace{2cm}
\gbeg{1}{2}
\guf{1} \gnl
\gcu{1} \gnl
\gob{2}{} \gnl
\gend=
\Id_{id_{\A}}
\end{equation}
and: 
\begin{center} \hspace{-1,4cm} 
\begin{tabular}{p{7.2cm}p{1cm}p{6.8cm}}
\begin{equation}\eqlabel{mixed lambda_M 8}
\gbeg{3}{5}
\gvac{2} \got{1}{\hspace{-0,4cm}F} \got{1}{\hspace{-0,4cm}F} \gnl
\gvac{2} \hspace{-0,34cm} \gmuf{1} \gnl
\gvac{1} \hspace{-0,34cm} \glmpb \gnot{\Delta_M} \gcmpb \grmptb \gnl
\gvac{1} \gcl{1} \gcl{1} \gcl{1} \gnl
\gvac{1} \gob{1}{F} \gob{1}{F} \gob{1}{B} \gnl
\gend=
\gbeg{5}{6}
\got{1}{F} \got{5}{F} \gnl
\glmptb \gnot{\Delta_M} \gcmpb \grmpb \gcl{1} \gnl
\gcl{2} \gcl{1} \glmptb \gnot{\hspace{-0,34cm}\tau_1} \grmptb \gnl
\gcl{1} \glmptb \gnot{\hspace{-0,34cm}\lambda} \grmptb \gcl{2} \gnl
\gmuf{1} \gcl{1} \gnl
\gob{2}{F} \gob{1}{F} \gob{1}{B} \gnl
\gend 
\end{equation} &  & 
\begin{equation}\eqlabel{mixed lambda_C 8}
\gbeg{3}{5}
\got{1}{F} \got{1}{F} \got{1}{B} \gnl
\gcl{1} \gcl{1} \gcl{1} \gnl
\glmpt \gnot{\mu_C} \gcmpt \grmptb \gnl
\gvac{2} \hspace{-0,22cm} \gcmu \gnl
\gvac{2} \gob{1}{F} \gob{1}{F} \gnl
\gend=
\gbeg{5}{6}
\got{2}{F} \got{1}{F} \got{1}{B} \gnl
\gcmu \gcl{1} \gcl{2} \gnl
\gcl{2} \glmptb \gnot{\hspace{-0,34cm}\lambda} \grmptb \gnl
\gvac{1} \gcl{1} \glmptb \gnot{\hspace{-0,34cm}\tau_2} \grmptb \gnl
\glmpt \gnot{\mu_C} \gcmptb \grmpt \gcl{1} \gnl
\gvac{1} \gob{1}{F} \gvac{1} \gob{1}{F.} \gnl
\gend
\end{equation}
\end{tabular}
\end{center}
\end{defn}

\begin{rem}
In the above definition, for the 2-cells $\tau_1$ and $\tau_2$ one may also say that $(F, \tau_1, \tau_2)$ is a 1-cell in $\bEM(\K)_{left}$ and that 
$(F, \tau_2, \tau_1)$ is a 1-cell in $\bEM(\K)_{right}$. 
\end{rem}

Apply $\Epsilon_B$ to \equref{mixed lambda_M 8} (or apply $\eta_B$ to \equref{mixed lambda_C 8}) to get: 
\begin{equation} \eqlabel{lr bialgebra}
\gbeg{2}{4}
\got{1}{F} \got{1}{F} \gnl
\gmuf{1} \gnl
\gcmu \gnl
\gob{1}{F}\gob{1}{F}
\gend=
\gbeg{3}{5}
\got{2}{F} \got{1}{F} \gnl
\gcmu \gcl{1} \gnl
\gcl{1} \glmptb \gnot{\hspace{-0,34cm}\lambda} \grmptb \gnl
\gmuf{1} \gcl{1} \gnl
\gob{2}{F}  \gob{1}{F}
\gend=
\gbeg{4}{5}
\got{2}{F} \got{2}{F} \gnl
\gcmu \gcmu \gnl
\gcl{1} \gbr \gcl{1} \gnl
\gmuf{1} \gmuf{1} \gnl
\gob{2}{F}  \gob{2}{F}
\gend
\end{equation}
then $F$ is a right bimonad in $\K$ with possibly non-(co)associative (co)multiplication. Applying 
$\gbeg{1}{2}
\guf{1} \gnl
\gob{1}{F}
\gend F$ to \equref{mixed lambda_M 8} and 
$\gbeg{1}{2}
\got{1}{F}\gnl
\gcu{1} \gnl
\gnl
\gend F$ to \equref{mixed lambda_C 8}, similarly as in \equref{Delta_m con Fi-lambda} and \equref{mu_C con omega}, we obtain $\Delta_M$ and $\mu_C$ below, where  
$\Phi_{\lambda}$ and $\omega$ are as in \equref{fi-lambda 3} and \equref{omega 3}: 
$$\Delta_M=
\gbeg{5}{6}
\gvac{1} \got{5}{F} \gnl
\gcn{1}{1}{2}{1} \gelt{\s\Phi_{\lambda}} \gcn{1}{1}{0}{1} \gcl{1} \gnl
\gcl{2} \gcl{1} \glmptb \gnot{\hspace{-0,34cm}\tau_1} \grmptb \gnl
\gcl{1} \glmptb \gnot{\hspace{-0,34cm}\lambda} \grmptb \gcl{2} \gnl
\gmuf{1} \gcl{1} \gnl
\gob{2}{F} \gob{1}{F} \gob{1}{B} \gnl
\gend
\hspace{2cm}
\mu_C=
\gbeg{5}{6}
\got{2}{F} \got{1}{F} \got{1}{B} \gnl
\gcmu \gcl{1} \gcl{2} \gnl
\gcl{1} \glmptb \gnot{\hspace{-0,34cm}\lambda} \grmptb \gnl
\gcl{1} \gcl{1} \glmptb \gnot{\hspace{-0,34cm}\tau_2} \grmptb \gnl
\gcn{1}{1}{1}{2} \gelt{\omega} \gcn{1}{1}{1}{0} \gcl{1} \gnl
\gvac{1} \gob{5}{F} \gnl
\gend
$$

\subsection{Case four: $\D_3={}_F ^F\YD(\C)$ with quasi (co)actions} \sslabel{Case 3}

In $\K=\hat\C$ the 2-cells $\lambda, \tau_1, \tau_2$ given by \vspace{-0,8cm}
\begin{center} \hspace{-1,4cm}
\begin{tabular}{p{3.6cm}p{1,5cm}p{5cm}p{0cm}p{4.6cm}}
\begin{equation} \eqlabel{right lambda} 
\lambda=
\gbeg{3}{5}
\got{1}{F} \got{2}{F} \gnl
\gcl{1} \gcmu \gnl
\gbr \gcl{1} \gnl
\gcl{1} \gmuf{1} \gnl
\gob{1}{F} \gob{2}{F}
\gend
\end{equation} & & \vspace{0,1cm}
\begin{equation*}
\tau_1=
\gbeg{2}{3}
\got{1}{B} \got{1}{F} \gnl
\gbr \gnl
\gob{1}{F} \gob{1}{B} 
\gend \hspace{2,5cm}
\tau_2=
\gbeg{2}{3}
\got{1}{F} \got{1}{B} \gnl
\gbr \gnl
\gob{1}{B} \gob{1}{F} 
\gend 
\end{equation*}
\end{tabular}
\end{center} \vspace{-0,3cm}
satisfy the distributive law conditions from the above definition ($\lambda$ does it for the same reasons as so did \equref{left lambda} for \deref{bl mixed}: here we have the right 
hand-side version of the previous situation). Now the 2-cells $\Delta_M$ and $\mu_C$ become as below: \vspace{-0,7cm}
\begin{center} 
\begin{tabular}{p{4.6cm}p{2cm}p{4.6cm}}
\begin{equation} \eqlabel{Delta_M}
\Delta_M=
\gbeg{5}{7}
\gvac{1} \got{5}{F} \gnl
\gcn{1}{2}{2}{0} \gelt{\s\Phi_{\lambda}} \gcn{1}{1}{0}{1} \gcl{1} \gnl
\gvac{1} \gcn{1}{1}{1}{0} \gbr \gnl
\gcn{1}{2}{0}{0} \gcn{1}{1}{0}{0} \hspace{-0,22cm} \gcmu \gcn{1}{1}{0}{1} \gnl
\gvac{1} \gbr \gcl{1}  \gcl{2} \gnl
\gmuf{1} \gmuf{1} \gnl
\gob{2}{F} \gob{2}{F} \gob{1}{B} \gnl
\gend
\end{equation} & & \vspace{-0,4cm}
\begin{equation} \eqlabel{mu_C}
\mu_C=
\gbeg{5}{7}
\got{2}{F} \got{2}{F} \got{1}{B} \gnl
\gcmu \gcmu \gcl{2} \gnl
\gcl{2} \gbr \gcl{1} \gnl
\gvac{1} \gcl{1} \gmuf{1} \gcn{1}{1}{1}{0} \gnl
\gcn{1}{2}{1}{3} \gcn{1}{1}{1}{2} \gvac{1} \hspace{-0,34cm} \gbr \gnl
\gvac{2} \gelt{\s\omega} \gcn{1}{1}{1}{0} \gcl{1} \gnl
\gob{9}{F.} \gnl
\gend
\end{equation}
\end{tabular}
\end{center}
The rest of the structures of $F$ are the same as in \ssref{Case 3 alt}: it has (co)associative pre-(co)multiplications so that by \equref{lambda mixed 1-3 lr} and \equref{lr bialgebra} 
$F$ is a bialgebra in $\C$, it has trivial left $B$-(co)module structures and $B$ is a trivial r-ight $F$-(co)module. The difference starts with the alternative structures: 
from \equref{quasi coaction} -- \equref{dual quasi action unity} we now obtain that $B$ is a left {\em quasi $F$-comodule} and a left {\em coquasi $F$-module}. Our 
terminology and notation ``$\Phi_{\lambda}$'' 
alludes to comodule algebras over a quasi-bialgebra defined in \cite{HN} and the dual construction of it. In the quasi-comodule algebra case we recover the same definition as in 
\cite{HN} (but more general: in the context of any braided monoidal category $\C$) without having that $F$ is a quasi-bialgebra, but rather it is a proper bialgebra. 
This situation appears also in \cite{Street} under the name ``twisted coaction''. Namely, the coaction in 
\equref{quasi coaction} is twisted by $\Phi_{\lambda}$ which by \equref{3-cocycle cond fi-lambda} -- \equref{normalized 3-cocycle fi-lambda} is a normalized 3-cocycle, and the action in 
\equref{dual quasi action} is twisted by $\omega$ which by \equref{3-cycle omega} -- \equref{normalized 3-cycle omega} is a normalized 3-cycle. 
Putting the above $\Delta_M, \mu_C, \psi$ and $\phi$ in \equref{quasi coaction} -- \equref{dual quasi action unity}, \equref{3-cycle omega} -- \equref{normalized 3-cocycle fi-lambda} 
we obtain: 
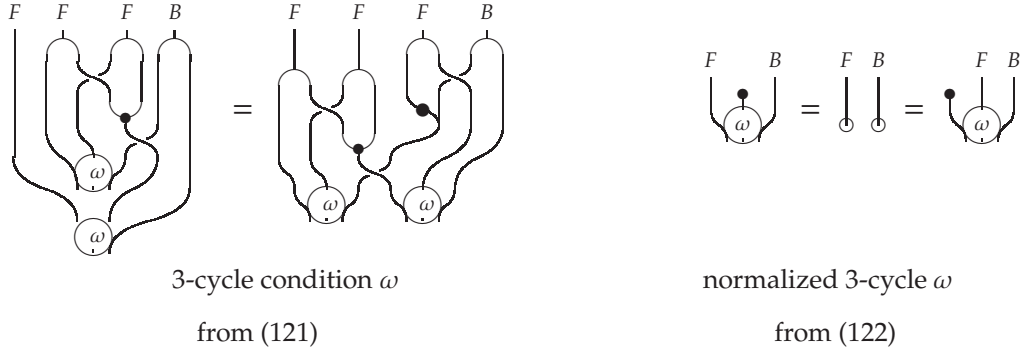
\begin{center} \hspace{-1,4cm} 
\begin{tabular} {p{8cm}p{1cm}p{4cm}} 
\begin{equation} \eqlabel{quasi coaction case 3} 
\gbeg{5}{7}
\got{9}{B} \gnl
\gcn{1}{2}{2}{0} \gelt{\s\Phi_{\lambda}} \gcn{1}{1}{0}{1} \grmo \gvac{1} \gcl{1} \gnl
\gvac{1} \gcn{1}{1}{1}{0} \gbr \gcl{3} \gnl
\gcn{1}{2}{0}{0} \gcn{1}{1}{0}{0} \hspace{-0,22cm} \gcmu \gcn{1}{2}{0}{0} \gnl
\gvac{1} \gbr \gcl{1} \gnl
\gmuf{1} \gmuf{1}  \hspace{-0,22cm} \gmu \gnl
\gvac{1} \gob{1}{F} \gvac{1} \gob{1}{F} \gob{2}{B}
\gend=
\gbeg{3}{9}
\got{3}{B} \gnl
\gvac{1} \gcl{1} \gcn{1}{1}{2}{1} \gelt{\s\Phi_{\lambda}} \gcn{1}{1}{0}{1} \gnl
\grmo \gvac{1} \gcl{1} \gcl{1} \gcl{3} \gcl{5} \gnl
\gcl{1} \gbr \gnl
\gmuf{1} \gcl{1} \gnl
\gcn{1}{1}{2}{1}  \grmo \gvac{1} \gcl{1} \gcl{1} \gnl
\gcl{1} \gcl{1} \gbr \gnl
\gcl{1} \gmuf{1} \gmu \gnl
\gob{1}{F} \gob{2}{F} \gob{2}{B}
\gend 
\end{equation} & & \vspace{0,8cm}
\begin{equation*} 
\gbeg{2}{4}
\got{3}{B} \gnl
\grmo \gvac{1} \gcl{1} \gnl
\gcu{1} \gcl{1} \gnl
\gob{3}{B} 
\gend=
\gbeg{2}{4}
\got{1}{B} \gnl
\gcl{2} \gnl
\gob{1}{B} 
\gend
\end{equation*} 
\end{tabular}
\end{center} \vspace{-0,5cm}
$$ \textnormal{\hspace{1cm}\footnotesize quasi coaction}  \hspace{4,6cm}  \textnormal{\footnotesize quasi coaction counity} $$ \vspace{-0,7cm}
$$ \textnormal{ \hspace{1cm} \footnotesize from \equref{quasi coaction}} \hspace{6cm} \textnormal{\footnotesize from \equref{quasi coaction counity}} $$

\begin{center} 
\begin{tabular}{p{8.6cm}p{1cm}p{5.3cm}}
\begin{equation} \eqlabel{3-cocycle cond fi-lambda case 3} 
\gbeg{7}{7}
\gcn{1}{1}{2}{1} \gelt{\s\Phi_{\lambda}} \gcn{1}{1}{0}{1} \gcn{1}{1}{2}{1} \gelt{\s\Phi_{\lambda}} \gcn{1}{1}{0}{1} \gnl
\gcn{1}{1}{1}{0} \gcn{1}{1}{1}{0} \gbr \gcn{2}{2}{1}{4} \gcn{1}{2}{-1}{2}  \gnl
\gcn{1}{2}{0}{0} \gcn{1}{1}{0}{0} \hspace{-0,22cm} \gcmu \gcn{1}{1}{0}{3} \gnl
\gvac{1} \gbr \gcl{1} \grmo \gvac{1} \gcl{1} \gcl{1} \gcl{2} \gnl
\gmuf{1} \gmuf{1}  \gcl{1} \gbr \gnl
\gcn{1}{1}{2}{2} \gvac{1} \gcn{1}{1}{2}{2} \gvac{1} \gmuf{1} \gmu \gnl
\gob{2}{F} \gob{2}{F} \gob{2}{F} \gob{2}{B} \gnl
\gend=
\gbeg{5}{8}
\gvac{1} \gcn{1}{2}{2}{-2} \gelt{\s\Phi_{\lambda}} \gcn{2}{2}{0}{5}  \gnl
\gvac{2} \gcn{2}{2}{2}{5} \gnl
\gcn{1}{5}{0}{0} \gcn{1}{2}{2}{0} \gelt{\s\Phi_{\lambda}} \gcn{1}{1}{0}{1} \gvac{1} \gcl{1} \gnl
\gvac{2} \gcn{1}{1}{1}{0} \gbr \gcl{3} \gnl
\gcn{1}{3}{0}{0} \gcn{1}{2}{0}{0} \gcn{1}{1}{0}{0} \hspace{-0,22cm} \gcmu \gcn{1}{2}{0}{0} \gnl
\gvac{2} \gbr \gcl{1} \gnl
\gvac{1} \gmuf{1} \gmuf{1}  \hspace{-0,22cm} \gmu \gnl
\gob{2}{F} \gob{1}{F} \gvac{1} \gob{1}{F} \gob{2}{B} \gnl
\gend 
\end{equation} &  & \vspace{0,6cm}
\begin{equation*} 
\gbeg{3}{3}
\gcn{1}{1}{2}{1} \gelt{\s\Phi_{\lambda}} \gcn{1}{1}{0}{1} \gnl 
\gcu{1} \gcl{1} \gcl{1} \gnl
\gvac{1} \gob{1}{F} \gob{1}{B} \gnl
\gend=
\gbeg{2}{3}
\guf{1} \gu{1} \gnl
\gcl{1} \gcl{1} \gnl
\gob{1}{F} \gob{1}{B} \gnl
\gend=
\gbeg{3}{3}
\gcn{1}{1}{2}{1} \gelt{\s\Phi_{\lambda}} \gcn{1}{1}{0}{1} \gnl 
\gcl{1} \gcu{1} \gcl{1} \gnl
\gob{1}{F} \gob{3}{B} \gnl
\gend
\end{equation*}
\end{tabular}
\end{center} \vspace{-0,5cm}
$$ \textnormal{\hspace{1cm}\footnotesize 3-cocycle condition for $\Phi_{\lambda}$}  \hspace{4,6cm}  \textnormal{\footnotesize normalized 3-cocycle $\Phi_{\lambda}$} $$ \vspace{-0,7cm}
$$ \textnormal{ \hspace{1cm} \footnotesize from \equref{3-cocycle cond fi-lambda}} \hspace{6cm} \textnormal{\footnotesize from \equref{normalized 3-cocycle fi-lambda}} $$

$$
\gbeg{5}{9}
\got{1}{F} \got{2}{F} \got{2}{B} \gnl
\gcl{1} \gcmu \gcmu \gnl
\gcl{1} \gcl{1} \gbr \gcl{5} \gnl
\gcn{1}{1}{1}{2}  \glmf \gcn{1}{1}{-1}{-1} \gcn{1}{4}{-1}{-1} \gnl
\gcmu \gcl{1} \gnl
\gcl{1} \gbr \gnl
\glmf \gcn{1}{1}{-1}{-1} \gcn{1}{1}{-1}{-1} \gnl
\gvac{1} \gcl{1} \gcn{1}{1}{1}{2} \gelt{\s\omega} \gcn{1}{1}{1}{0} \gnl
\gob{3}{B} 
\gend=
\gbeg{4}{7}
\got{2}{F} \got{2}{F} \got{1}{B} \gnl
\gcmu \gcmu  \hspace{-0,22cm} \gcmu \gnl
\gvac{1} \hspace{-0,22cm} \gcl{2} \gbr \gcl{1} \gcn{1}{2}{0}{0} \gcn{1}{3}{0}{0} \gnl
\gvac{2} \gcl{1} \gmuf{1} \gnl
\gvac{1} \gcn{1}{2}{1}{3} \gcn{1}{1}{1}{2} \gvac{1} \hspace{-0,32cm} \gbr \gnl
\gvac{3} \gelt{\s\omega} \gcn{1}{1}{1}{0} \glmf \gcn{1}{1}{-1}{-1} \gnl
\gob{13}{B} 
\gend \hspace{4,5cm}
\gbeg{1}{4}
\got{1}{B}\gnl
\gcl{2}  \gnl
\gob{1}{B} 
\gend=
\gbeg{3}{5}
\got{3}{B}\gnl
\guf{1} \gcl{1} \gnl
\glmf \gcn{1}{1}{-1}{-1}  \gnl
\gvac{1} \gcl{1} \gnl
\gob{3}{B} 
\gend
$$
$$\qquad \textnormal{ \footnotesize quasi action}  \hspace{4,5cm}  \textnormal{\footnotesize quasi action unity} $$
$$\qquad \textnormal{\footnotesize from \equref{dual quasi action}} \hspace{5cm} \textnormal{\footnotesize from \equref{dual quasi action unity}} $$

$$
\gbeg{7}{7}
\got{1}{F} \got{2}{F} \got{2}{F} \got{1}{B} \gnl
\gcl{4} \gcmu \gcmu  \hspace{-0,22cm} \gcmu \gnl
\gvac{2} \gcn{1}{2}{0}{0} \hspace{-0,22cm} \gbr \gcl{1} \gcn{1}{2}{0}{0} \gcn{1}{4}{0}{0} \gnl
\gvac{3} \gcl{1} \gmuf{1} \gnl
\gvac{2} \gcn{1}{2}{1}{3} \gcn{1}{1}{1}{2} \gvac{1} \hspace{-0,32cm} \gbr \gnl
\gvac{3} \gcn{1}{2}{-2}{2} \gelt{\s\omega} \gcn{1}{1}{1}{0} \gcn{2}{2}{1}{-2} \gnl 
\gvac{5} \gcn{2}{2}{5}{0} \gnl
\gvac{4} \gelt{\s\omega}  \gnl
\gend=
\gbeg{4}{7}
\got{2}{F} \got{2}{F} \got{2}{F} \got{2}{B} \gnl
\gcn{1}{1}{2}{2} \gvac{1} \gcn{1}{1}{2}{2} \gvac{1} \gcmu \gcmu \gnl
\gcmu \gcmu  \gcl{1} \gbr \gcl{2} \gnl
\gcl{2} \gbr \gcl{1} \glmf \gcn{1}{1}{-1}{-1} \gcn{1}{1}{-1}{-1} \gnl
\gvac{1} \gcl{1} \gmuf{1} \gcn{1}{1}{3}{0} \gcn{2}{2}{3}{0} \gcn{1}{2}{1}{-2} \gnl
\gcn{1}{1}{1}{2} \gcn{1}{1}{1}{2} \gvac{1} \hspace{-0,32cm} \gbr  \gnl
\gvac{1} \gcn{1}{1}{1}{2} \gelt{\s\omega} \gcn{1}{1}{1}{0} \gcn{1}{1}{1}{2} \gelt{\s\omega} \gcn{1}{1}{1}{0} \gnl
\gend \hspace{4cm}
\gbeg{3}{4}
\got{1}{F} \got{3}{B} \gnl
\gcl{1} \guf{1} \gcl{1} \gnl
\gcn{1}{1}{1}{2} \gelt{\omega} \gcn{1}{1}{1}{0} \gnl
\gend=
\gbeg{2}{4}
\got{1}{F} \got{1}{B} \gnl
\gcl{1} \gcl{1} \gnl
\gcu{1} \gcu{1} \gnl
\gend=
\gbeg{4}{4}
\gvac{1} \got{1}{F} \got{1}{B} \gnl
\guf{1} \gcl{1} \gcl{1} \gnl
\gcn{1}{1}{1}{2} \gelt{\omega} \gcn{1}{1}{1}{0} \gnl
\gend
$$

$$\qquad  \textnormal{\footnotesize  3-cycle condition $\omega$} \hspace{4cm}\textnormal{\footnotesize normalized 3-cycle $\omega$} $$
$$\quad\textnormal{\footnotesize from \equref{3-cycle omega}} \hspace{6cm} \textnormal{\footnotesize from \equref{normalized 3-cycle omega}} $$

\begin{ex} \exlabel{3-1}
With the left-right mixed biwreath-like object we recover the definition \cite[Definition 12]{Street} of a ``twisted coaction''. Observe that this definition comes out 
from the 2-cell conditions \equref{quasi coaction} -- \equref{quasi coaction counity} of $\Delta_M, \Epsilon_M$, which underlie in the definition of a mixed wreath. 
\end{ex}

\begin{ex}
For $\C=\R\x\Mod$, the category of modules over a commutative ring $R$, we recover the definition \cite[Definition 7.1]{HN} of a comodule algebra over a (quasi)bialgebra  
(as we said above, the quasi coaction of $H$ in \cite[Definition 7.1]{HN} can be defined no matter if $H$ is a quasi or a proper bialgebra). Dually, we obtain the definition 
of a module algebra over a (coquasi)bialgebra. 
\end{ex}

\begin{ex} \exlabel{3-3}
By \rmref{4 combs of wreaths} we have that with the structures studied in this example $F$ is a mixed wreath around the monad $B$. Thus we recover 
(the left hand-side version of) \cite[Proposition 13]{Street} and \cite[Proposition 5.3]{BC} for $\C=\R\x\Mod$. 

In the latter case the braided monoidal category $\C$ is that of 
modules over a commutative ring $R$, although both $B$ and $F$ are left $B$-modules; observe that as above, the fact that $H$ in \cite[Proposition 5.3]{BC} is a quasi-bialgebra does not play any r\^ole.  

In our example we also obtain the dual construction: with the above structures $F$ is a mixed cowreath around the comonad $B$. 
\end{ex}

If $\mu_C, \Delta_M$ are canonical, we just have that $B$ is a proper left $F$-module and comodule.

\begin{rem}
We might consider $F=B$ and review the construction in \ssref{Case 3} with this assumption. 
In this case we may take $\Phi_{\lambda}=\Phi$, then \equref{quasi coaction case 3} would look like the quasi-bialgebra condition, though the condition 
\equref{3-cocycle cond fi-lambda case 3} is not precisely the necessary 3-cocycle condition defined in \cite{Drinfeld}. 
\end{rem}

\end{document}